\documentclass[11pt]{amsart}

\usepackage{tikz,amssymb,amsmath,fullpage,listings,amsthm,graphicx}

\definecolor{mypink}{RGB}{255,192,203}
\definecolor{myblue}{RGB}{160,203,251}

\newcommand{\rank}{\mathsf{rank}}
\newcommand{\degree}{\mathsf{degree}}
\newcommand{\sort}{\mathsf{sort}}

\newcommand{\xpara}{\mathsf{xpara}}
\newcommand{\ypara}{\mathsf{ypara}}
\newcommand{\configuration}[2]{{#1 \choose #2}}
\newcommand{\topplingconfiguration}[1]{\Delta^{(#1)}}
\newcommand{\topplingequiv}{\equiv_\Delta}
\newcommand{\park}{\mathsf{park}}
\newcommand{\diracconfiguration}[1]{\epsilon^{(#1)}}
\newcommand{\topplingandpermutingequiv}[1]{\equiv_{\Delta,#1}}
\newcommand{\Ta}{\mathsf{T_a}}
\newcommand{\Tb}{\mathsf{T_b}}

\newcommand{\area}{\mathsf{area}}
\newcommand{\width}{\mathsf{width}}
\newcommand{\height}{\mathsf{height}}
\newcommand{\visited}[1]{V(#1)}
\newcommand{\xarea}{\mathsf{xarea}}
\newcommand{\yarea}{\mathsf{yarea}}
\newcommand{\xrow}{\mathsf{xrow}}
\newcommand{\yrow}{\mathsf{yrow}}
\newcommand{\northstep}{\mathsf{N}}
\newcommand{\eaststep}{\mathsf{E}}

\newtheorem*{thm:rank_formula}{Theorem \ref{thm:rank_formula}}
\newtheorem*{thm:gf_main_formula}{Theorem \ref{thm:gf_main_formula}}
\newtheorem{theorem}{Theorem}[section]
\newtheorem{proposition}[theorem]{Proposition}
\newtheorem{lemma}[theorem]{Lemma}
\newtheorem{claim}{Claim}
\theoremstyle{remark}
\newtheorem{remark}[theorem]{Remark}
\newtheorem{example}[theorem]{Example}

\numberwithin{equation}{section}

\begin{document}

\title{The sandpile model on $K_{m,n}$ and\\ the rank of its configurations}

\author[M. D'Adderio]{Michele D'Adderio}
\address{Universit\'e Libre de Bruxelles (ULB)\\D\'epartement de Math\'ematique\\ Boulevard du Triomphe, B-1050 Bruxelles\\ Belgium}\email{mdadderi@ulb.ac.be}

\author[Y. Le Borgne]{Yvan Le Borgne}
\address{LaBRI, Universit\'e de Bordeaux, CNRS, 351 cours de la Lib\'eration, 33405 Talence, France}\email{yvan.le-borgne@u-bordeaux.fr}

\begin{abstract}
We present an algorithm to compute the rank of a configuration of the
sandpile model for the complete bipartite graph $K_{m,n}$ of
complexity $O(m+n)$. Furthermore, we provide a formula for the generating function of parking sorted configurations on complete bipartite graphs $K_{m,n}$ according to $\rank$, $\degree$, and the sizes $m$ and $n$.

The results in the present
paper are similar to those found in \cite{corileborgne} for the complete graph
$K_{n+1}$, and they rely on the analysis of certain operators on
the stable sorted configurations of $K_{m,n}$ developed in
\cite{addl}.
\end{abstract}

\maketitle

\tableofcontents

\section*{Introduction}

The sandpile model was introduced in 1987 by the physicists Bak, Tang and Wiesenfeld as a first example of a dynamical system displaying the remarkable property of self-organized criticality \cite{bak} (see also \cite{dhar1,dhar}). Since then, this model as well as its several variants (among which the notable \emph{chip-firing game} \cite{merino}) has proved to be a fertile ground from which novel and intriguing results in mathematics have emerged (see e.g. \cite{bjorner,merino}).

In 2007, Baker and Norine, in the celebrated article \cite{baker}, introduced the notion of rank of a configuration of the sandpile model on a graph, which yielded a surprising analogue of the classical Riemann-Roch theorem for divisors on curves. This interesting concept stimulated an increasing amount of research, not just in combinatorics, but even more in algebraic and tropical geometry (see e.g. \cite{caporaso,hladky} for some recent developments).

Unfortunately, computing the combinatorial rank directly from its very definition is nontrivial, even for small examples. Despite the attention that this notion has attracted since its appearance in \cite{baker}, the problem of deciding the complexity of its computation (first attributed to Hendrik Lenstra, cf. \cite{hladky}) has been addressed only recently. Indeed it has been shown in \cite{kiss} that computing the rank of a configuration on a general (in fact even eulerian) graph is a NP-hard problem.

This situation motivates the search for efficient algorithms to compute the rank on some special classes of finite graphs.

One of the first works in this direction is \cite{corileborgne}, where an algorithm of linear complexity (in the number of vertices when we assume constant time for any integer division) to compute the rank on the complete graphs has been described. Moreover, in the same work several enumerative byproducts have been provided, which turn out to have intrinsic combinatorial interest.

In the present work we will give several results about the rank of configurations on the complete bipartite graphs, parallel to the ones proved in \cite{corileborgne} for the complete graph.

Given a connected simple graph $G=(V,E)$, where $V$ is the set of its vertices, a \emph{configuration} on $G$ (a \emph{divisor} in the terminology of \cite{baker}) is simply a function $u:V\to \mathbb{Z}$, while the \emph{degree} of a configuration $u$ is simply the sum of its values, so $\degree(u)\in \mathbb{Z}$. The \emph{rank} of a configuration is an integer $\rank(u)\geq -1$ associated to $u$ (see Section 1 for the definition). In our \emph{sandpile model} we will always distinguish a vertex of $G$, which we will call the \emph{sink}. 

The first of our main results is an algorithm of linear complexity (in the number of vertices) to compute the rank of the configurations on the complete bipartite graphs: see Section 2 for its pseudocode. The analysis of this algorithm relies heavily on the combinatorial study in \cite{addl} of the compact (and in particular stable) sorted configurations on the complete bipartite graph.

The design of our algorithm leads to several enumerative byproducts. One of them is the introduction of the \emph{$r$-vector} $r(u)=(r_1,r_2,\dots,r_n)\in \mathbb{Z}^n$ of a compact \emph{sorted configuration} (i.e. a configuration $u$ up to an automorphism of the graph that fixes the sink) on the complete bipartite graph $K_{m,n}$. 

The second of our main results is an explicit formula for the rank of a sorted configuration on $K_{m,n}$ that is also \emph{parking} (a \emph{reduced} divisor in the terminology of \cite{baker}):
\begin{thm:rank_formula}
Let $u$ be a parking sorted configuration on $K_{m,n}$, let $r(u)=(r_1,r_2,\dots,r_n)$ be its $r$-vector, and let
$u_{a_m}\geq 0$ be the value of $u$ on the sink $a_m$. Let
$$
u_{a_m}+1=nQ+R,\quad \text{ with }Q,R\in \mathbb{N},\,\, \text{
and }0\leq R\lneqq n.
$$
Then
$$
\rank(u)+1=\sum_{i=1}^n\max\{0,Q+\chi(i\leq R)+r_i-1\},
$$
where $\chi(\mathcal{P})$ is $1$ if the proposition $\mathcal{P}$
is true, and $0$ otherwise.
\end{thm:rank_formula}

In turn, a deeper understanding of the combinatorics behind the algorithm will bring our third main result, which is a formula for the generating function of the parking sorted configurations on complete bipartite graphs according to $\rank$, $\degree$, and the sizes of the two parts of the set of vertices: given
$$
\mathcal{F}(x,y,w,h):=\sum_{n\geq 1,m\geq 1}K_{m,n}(x,y)w^mh^n,
$$
where
\begin{equation} \label{eq:intro}
K_{m,n}(x,y)=x^{(m-1)(n-1)}y\widetilde{K}_{m,n}(x^{-1},xy),
\end{equation}
and
$$
\widetilde{K}_{m,n}(d,r):=\sum_{u\text{ parking sorted on }K_{m,n}}d^{\degree(u)}r^{\rank(u)},
$$
the following theorem relates $\mathcal{F}(x,y,w,h)$ with the generating function $P(q;w,h)$ of \emph{parallelogram polyominoes} according to their \emph{area} (counted by $q$), their \emph{width} (counted by $w$) and their \emph{height} (counted by $h$) computed in \cite{bousquetmelouviennot}.
\begin{thm:gf_main_formula}
$$
\mathcal{F}(x,y,w,h)=\frac{(1-xy)(hw-P(x;w,h)P(y;w,h))}{(1-x)(1-y)(1-h-w-P(x;w,h)-P(y;w,h))}.
$$
\end{thm:gf_main_formula}

The paper is organized as follows.

In the first section we recall some basic definitions and results about the sandpile model, but we will restrict our attention mostly on the complete bipartite graph $K_{m,n}$. In particular we will define the notion of the rank of a configuration following \cite{baker}.

In the second section we state a first version of our algorithm to compute the rank of a configuration on $K_{m,n}$.

We will use Sections 3-7 to recall some key results from \cite{dukesleborgne} and \cite{addl}. In particular we will introduce our graphical interpretation of stable sorted configurations on $K_{m,n}$, and the corresponding graphical interpretation of the operators $\varphi$, $\psi$, $\Ta$ and $\Tb$, already studied in \cite{addl}. All these results will be used all along the rest of the paper. Notice that in these sections we will use a notation slightly different (but lighter) from the one used in \cite{addl}.

In Section 8 we will prove the correctness of our algorithm.

In Section 9 we will provide a new formulation of our algorithm that uses the operators $\Ta$ and $\Tb$. We will prove in this section that indeed the two algorithms are equivalent, by providing a graphical interpretation of the original algorithm.

This last analysis will enable us to introduce in Section 10 the useful notion of a cylindric diagram of a parking sorted configuration on $K_{m,n}$. This will provide a very efficient way to compute the rank of such a configuration.

Indeed in Section 11 we will introduce the notion of the $r$-vector of a stable sorted configuration on $K_{m,n}$, which, together with the cylindric diagram, will give us the formula of Theorem \ref{thm:rank_formula}.

In Section 12 we provide a detailed analysis of the complexity of our algorithm, making more explicit some of its steps, proving in this way the announced linear bound.

In Section 13 we introduce two new statistics, $\xpara$ and $\ypara$, on the parking sorted configurations on $K_{m,n}$. The strict relation between the bistatistics $(\xpara,\ypara)$ and $(\degree,\rank)$ is encoded in the change of variables \eqref{eq:intro}. 

In Section 14 we show that the bistatistic $(\xpara,\ypara)$ is symmetrically distributed, and its symmetry is explained by the Riemann-Roch theorem of Baker and Norine \cite[Theorem 1.12]{baker}, cf. Theorem \ref{thm:x_y_symmetry}, but also combinatorially via our graphical interpretation.

Finally in Section 15 we prove the formula in Theorem \ref{thm:gf_main_formula}.

\section{Basic definitions and results}

The \emph{complete bipartite graph} $K_{m,n}$ is the graph $(V,E)$
whose vertex set $V=A_m\sqcup B_n$ is the disjoint union of the
sets $A_m := \{a_1,a_2,\dots,a_m\}$ and $B_n:=
\{b_1,b_2,\dots,b_n\}$ and where there is exactly one unoriented
edge $\{a_i,b_j\}\in E$ for any $a_i\in A_m$ and $b_j \in B_n$.

To work in the sandpile model framework, we will often distinguish one
vertex, that will be called the \emph{sink}. For $K_{m,n}$ we
typically choose $a_m$ to be the sink. For this reason, the notation
$A_{m-1}:=A_m\setminus \{a_m\}$ will be sometimes useful.

A \emph{configuration} $u$ on the graph $K_{m,n}$ is a function
$u:V\to \mathbb{Z}$, i.e. $u\in \mathbb{Z}^V$, that we will also
denote as
$$
u := \configuration{u_{a_1},u_{a_2},\ldots,
u_{a_{m-2}},u_{a_{m-1}}; u_{a_m}}{u_{b_1},u_{b_2},\ldots, u_{b_n}},
$$
where for any vertex
$c_k \in A_m\cup B_n$, $u_{c_k}:=u(c_k)\in\mathbb{Z}$ denotes the
\emph{value} of the configuration $u$ at the vertex $c_k$. Notice the semi-colon which distinguishes the value at the sink
$a_m$.

The \emph{degree} of a configuration $u$ is defined as
$$
\degree(u)
:= \sum_{c_k \in A_m\cup B_n} u_{c_k}.
$$

We need some more notation. For any $c_k\in V$, let $e_{c_k,c_h}$ be
$1$ if $\{c_k,c_h\}\in E$ and $0$ otherwise and let
$d_{c_k}:=\sum_{c_h\in V} e_{c_k,c_h}$ be the degree of the vertex
$c_k$. Also, let us denote by $\diracconfiguration{c_k}\in
\mathbb{Z}^V$ the \emph{dirac configuration} which has value $1$ on
$c_k$ and $0$ elsewhere.

Given a vertex $c_k\in V$, the \emph{toppling operator}
$\Delta^{(c_k)}\in \mathbb{Z}^V$ is defined as
$$
\Delta^{(c_k)}=d_{c_k}\diracconfiguration{c_k}-\sum_{c_h\in
V}e_{c_k,c_h}\diracconfiguration{c_h}.
$$
The \emph{toppling} of the vertex $c_k\in A_m\cup B_n$ in the
configuration $u$ corresponds to computing $u-\Delta^{(c_k)}$, which
means that an amount of $1$ is sent from $c_k$ to every neighbor
along the corresponding edge.
\begin{remark} \label{rem:Delta} In graph theory, the matrix defined by the rows $(\Delta^{(c_k)})_{c_k \in V}$ is called the \emph{laplacian matrix} of the underlying graph. 
It is well known and also clear from the previous observation that
$$
\sum_{c_k\in A_m\cup B_n}\Delta^{(c_k)}=0\in \mathbb{Z}^V.
$$
So for example
$$
\Delta^{(c_k)}=-\mathop{\sum_{c_h\in A_m\cup B_n}}_{c_h\neq
c_k}\Delta^{(c_h)}.
$$
\end{remark}

We say that the configurations $u$ and $v$ are \emph{toppling
equivalent}, also denoted $u\topplingequiv v$, if there exists a
finite sequence of topplings specified by the toppling vertices
$(c_k)_{k=1,\ldots, K}$ such that $v = u-\sum_{k=1}^{K}
\topplingconfiguration{c_k}$.

Using Remark \ref{rem:Delta}, it is easy to prove the following
proposition.
\begin{proposition}[Dhar]
The toppling equivalence is an equivalence relation.
\end{proposition}

A vertex $c_k$ is \emph{unstable} in a configuration $u$ if
$u_{c_k}\geq d_{c_k}$, otherwise this vertex is \emph{quasi-stable}.
Toppling once an unstable vertex $c_k$ preserves the non-negativity of its value, as the only decremented value is at vertex $c_k$.  Such a
toppling is called a \emph{legal} toppling. A configuration is
\emph{quasi-stable} if any vertex distinct from the sink is quasi-stable.  The
\emph{relaxation process} of a configuration $u$ consists of iteratively performing a
legal toppling of an unstable vertex distinct from the sink, as long as there exists one.

The following proposition is due to Dhar \cite{dhar} (cf. also \cite[Proposition 2.1]{corirossin}).
\begin{proposition}[Dhar]
Given any configuration $u$, the relaxation process terminates on a
quasi-stable configuration independent of the ordering of legal topplings.
\end{proposition}

A configuration $u$ is \emph{non-negative} if $u_{c_k}\geq 0$ for
every $c_k \in A_m \cup B_n$.

A configuration $u$ is \emph{non-negative outside the sink} if for any
vertex $c_k \in A_{m-1}\cup B_n$, hence any $c_k$ distinct from the
sink $a_m$, the value $u_{c_k}$ is non-negative. Otherwise, the
configuration $u$ has a negative value outside the sink.

In this paper, we define a \emph{stable configuration} to be a quasi-stable configuration which is also non-negative outside the sink.
\begin{remark}
Notice that sometimes in the literature what we called ``quasi-stable configuration'' is called ``stable configuration''.
\end{remark}

A configuration $u$ is \emph{effective} if $u$ is toppling equivalent
to a non-negative configuration.

The \emph{rank} of a configuration $u$ is defined as
$$
\rank(u) := -1+\min\{ \degree(f)\mid \mbox{$f$ is non-negative and
$u-f$ is non-effective}\}.
$$
In other words, the rank incremented by one is the minimal
cumulative decrease of values in the configuration $u$ needed to
reach a non-effective configuration.

It is clear from the definitions that if $u\topplingequiv v$ then
$\rank(u)=\rank(v)$.

A non-negative configuration $f$ such that $u-f$ is non-effective
and $\degree(f)=\rank(u)+1$ is called a \emph{proof for the rank} of $u$.

A configuration $u$ is \emph{parking} (with respect to the sink
$a_m$) if $u$ is non-negative outside the sink and for any
non-empty subset $C$ of $A_{m-1}\cup B_n$ the configuration
$u-\Delta^{(C)}$ has a negative value outside the sink, where
$$
\Delta^{(C)}:= \sum_{c_k \in C} \topplingconfiguration{c_k}.
$$

For a proof of the following proposition see \cite[Proposition 3.1]{baker} or \cite[Corollary 33]{corileborgne}.
\begin{proposition}[Dhar]
Any configuration $u$ is toppling equivalent to exactly one
parking configuration (with respect to the sink $a_m$), which we
will denote by $\park(u)$.
\end{proposition}

Baker and Norine observed that the test that a configuration $u$
is effective may be reduced to the fact that the value of $a_m$ is
non-negative in the configuration $\park(u)$.

For a proof of the following theorem see \cite[Theorem 3.3]{baker} or \cite[Proposition 37]{corileborgne}.
\begin{theorem}[Baker-Norine] \label{thm:effective_park_nonnegative}
A configuration $u$ is effective if and only if $\park(u)$ is non-negative.
\end{theorem}

\section{A greedy algorithm for the rank on $K_{m,n}$: statement}

We first describe the following non-deterministic greedy algorithm
which computes a proof for the rank of any configuration $u$ of
$K_{m,n}$.

All the algorithms in this paper will be written in a pseudocode resembling the programming language PYTHON (compare this algorithm with the one in \cite[Section 2.2]{corileborgne}).
\lstset{language=Python}
\begin{lstlisting}[frame=single,texcl,mathescape]
def compute_rank($u$):
  $u$ = park($u$)
  rank = $-1$
  $f$ = $0$  # $\,\, f$ is the $0$ configuration
  while $u_{a_m}$ >= $0$:
    let $i$ be such that $u_{b_i}$ = $0$
    $u$ = park($u - \diracconfiguration{b_i}$)
    $f$ = $f + \diracconfiguration{b_i}$
    $rank$ = $rank + 1$
  return ($rank$,$f$)
\end{lstlisting}

The first main result of our paper is the following theorem.

\begin{theorem} \label{thm:correctness}
The algorithm returns in $f$ a proof for the rank of the input
$u$, hence $\degree(f)-1=rank=\rank(u)$.
\end{theorem}

In order to prove this theorem, we will first need to introduce
some important notions and to recall some results from
\cite{addl}.

In a later section we will provide an optimization of this algorithm on improved data-structures to reach the announced linear arithmetic complexity.

\section{Stable and sorted configurations} \label{sec:stable_sorted}

The complete bipartite graph $K_{m,n}$ admits numerous automorphisms, the vertex $a_m$ being fixed or not.

Let $S_k$ denote the set of permutations of $\{1,2,\ldots, k\}$. We define the action of $\sigma=(\sigma^{a},\sigma^{b})\in
S_m\times S_n$ on the configuration $u$ on $K_{m,n}$ by
$$
\sigma \cdot u := \configuration{u_{a_{\sigma^a(1)}},
u_{a_{\sigma^a(2)}},\ldots
u_{a_{\sigma^a(m-1)}};u_{a_{\sigma^a(m)}}}{u_{b_{\sigma^b(1)}},
u_{b_{\sigma^b(2)}},\ldots u_{b_{\sigma^b(n)}}}.
$$

We sometimes restrict this action to the elements
$\sigma=(\sigma^a,\sigma^b)$ such that $\sigma^a(m)=m$, so that the
value on the distinguished sink $a_m$ is preserved. Such
elements are also denoted as elements of $S_{m-1}\times S_n$
instead of $S_m\times S_n$. Clearly they all act as automorphisms of $K_{m,n}$.

Most of our definitions fit well with these symmetries. For
example, two configurations $u$ and $v$ are called \emph{toppling
and ($S$-)permuting equivalent}, where $S=S_m\times S_n$ or
$S=S_{m-1}\times S_n$, if there exists an element $\sigma
\in S$ such that $\sigma\cdot u$ and $v$ are toppling equivalent
($\sigma\cdot u \topplingequiv v$). This binary relation is denoted by $u
\topplingandpermutingequiv{S} v$.

All the properties claimed in the following lemma are easy to check directly from the definitions, so the proofs are left to the reader.
\begin{lemma} \label{lem:symmetry}
Let $u$ and $v$  be configurations of $K_{m,n}$, $\sigma$ a
permutation of $S_m\times S_n$ and $\tau$ a permutation of
$S_{m-1}\times S_n$.
\begin{enumerate}
\item The binary relations $u \topplingandpermutingequiv{S_m\times
S_n} v$ and $u \topplingandpermutingequiv{S_{m-1}\times S_n} v$
are equivalence relations;

\item $\sigma\cdot u$ is non-negative if and only if $u$ is
non-negative;

\item $\sigma\cdot u$ is effective if and only if $u$ is
effective;

\item $\tau\cdot u$ is non-negative outside the sink if and only
if $u$ is non-negative outside the sink;

\item $\park(\tau\cdot u)=\tau\cdot
\park(u)$;

\item $f$ is a proof for the rank of $u$ if and only if
$\sigma\cdot f$ is a proof for the rank of $\sigma\cdot u$;

\item $\degree(\sigma\cdot u) = \degree(u)$;

\item $\rank(\sigma\cdot u) = \rank(u)$.
\end{enumerate}
\end{lemma}

On $K_{m,n}$, the degree of any vertex $a_i$ is $n$ and the degree of
any vertex $b_j$ is $m$, hence a \emph{stable} configuration $u$ on
$K_{m,n}$ (with respect to the sink $a_m$) is a configuration such
that $0\leq u_{a_i}<n$ for all $i=1,2,\dots,m$ and $0\leq u_{b_j}<m$
for all $j=1,2,\dots,n$. Notice that we have no conditions on the
value $u_{a_m}$ of $u$ on the sink $a_m$.  We will sometimes identify
configurations that differ only by their value on the sink. In the
sequel, in these situations, the value $u_{a_m}$ of $u$ on
the sink will be denoted by the symbol $*$.

A \emph{sorted} configuration $u$ on $K_{m,n}$ (with respect to
the sink $a_m$) is a configuration such that $u_{a_i}\leq
u_{a_{i+1}}$ for all $i=1,2,\dots,m-2$ and $u_{b_j}\leq
u_{b_{j+1}}$ for all $j=1,2,\dots,n-1$. Notice again that there
are no conditions on $u_{a_m}$. Clearly, given a configuration $u$
on $K_{m,n}$, there exists a unique sorted configuration
$\sort(u)$ for which there exists a $\tau \in S_{m-1}\times S_n$
such that $\sort(u)=\tau \cdot u$.
\begin{example}
Let $m=7$, $n=5$, and consider the configuration on $K_{7,5}$
$$
u := \configuration{2,0,2,2,0,0;*}{4,4,0,0,4},
$$
where remember that $*$ denotes the value at the sink. This is a
stable configuration. Also, the corresponding sorted configuration
is
$$
\sort(u)=\tau\cdot u=\configuration{0,0,0,2,2,2;*}{0,0,4,4,4}
$$
where $\tau=(\tau^a,\tau^b)\in S_6\times S_5$, $\tau^a=(1,\,
5)(3,\, 6)\in S_6$ and $\tau^b=(1,\, 3)(2,\, 4)\in S_5$ written as products of transpositions.
\end{example}

Since we are interested in computing the rank, because of Lemma
\ref{lem:symmetry}, we can restrict ourselves to work with sorted
configurations.
More precisely, the rank of any configuration $u$ is also the rank of $\sort(\park(u))$.
Indeed, the first step of our algorithm will be to efficiently compute this latter configuration.
Then the second step of our algorithm will compute the rank of this parking sorted configuration, which is a particular case of stable sorted configuration.
In order to work with stable sorted configurations, it will be very
convenient to introduce a pictorial interpretation of our objects.

\section{A pictorial interpretation}

Given a stable sorted configuration $u :=
\configuration{u_{a_1},u_{a_2},\ldots u_{a_{m-2}},u_{a_{m-1}};
u_{a_m}}{u_{b_1},u_{b_2},\ldots u_{b_n}}$, we draw its \emph{diagram}, i.e. two lattice paths in a $m\times
n$ square grid, from the southwest corner to the
northeast corner, moving only north or east steps, which will
encode $u$: a green path whose north step in the $i$-th row is
$u_{b_i}+1$ squares distant from the west edge of the grid, and a
red path whose east step in the $j$-th column is $u_{a_j}+1$
squares distant from the south edge of the grid, except for the
$m$-th column (the one corresponding to the sink), which will
always be at distance $n$, i.e. on the upper edge of the $m\times
n$ square grid.

\begin{example}
Let $m=7$, $n=5$, and consider the stable sorted configuration on
$K_{7,5}$
$$
u :=\configuration{0,0,0,2,2,2;*}{0,0,4,4,4}.
$$
Its diagram is drawn in Figure \ref{fig:configKmn}.

\begin{figure}[h]
\includegraphics[width=60mm,clip=true,trim=15mm 215mm 90mm 25mm]{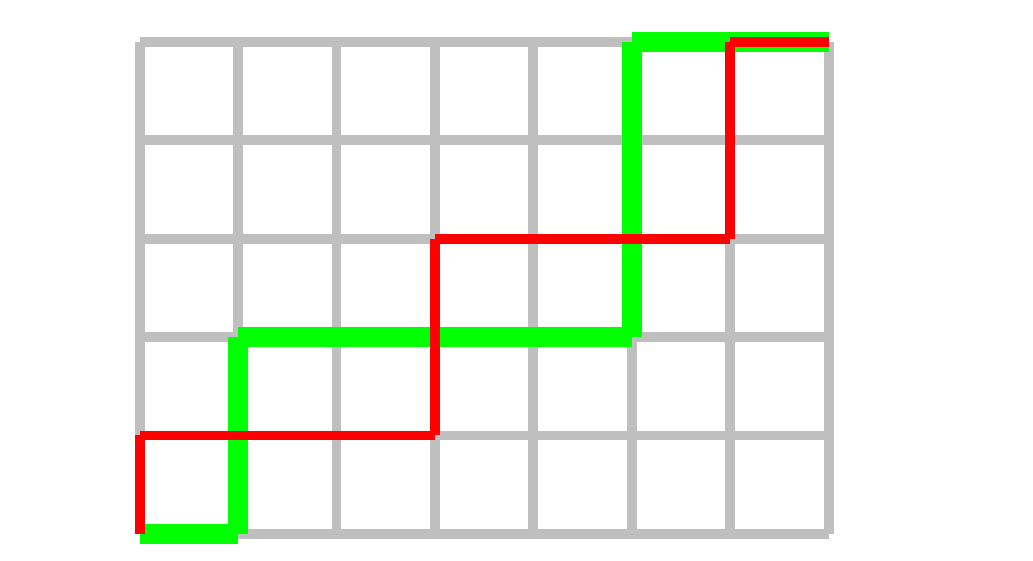}
\caption[ciccia]{The stable sorted configuration
$u=\configuration{0,0,0,2,2,2;*}{0,0,4,4,4}$.}
\label{fig:configKmn}
\end{figure}
\end{example}

Notice that in this representation the red path will always start
at the southwest corner with a north step, and it will always end
with a east step (because of our convention). Similarly, the green
path will always start at the southwest corner with an east step.

\section{The operators $\varphi$ and $\psi$ on $K_{m,n}$}

In \cite{addl} two operators on stable sorted configurations have
been studied. We recall here some of those results, but possibly
with a slightly different (and lighter) notation. We refer to
\cite{addl} for details and proofs.\newline

We define an operator $\psi$ on stable sorted configurations in
the following way: fix a total order on the vertices $A_{m-1}\cup
B_n$ of $K_{m,n}$ different from the sink $a_m$, say
$a_1<a_2<\cdots<a_{m-1}<b_1<b_2<\cdots <b_n$.

Given a stable configuration $u$ on $K_{m,n}$, if for every nonempty
$C\subseteq A_{m-1}\cup B_n$ the configuration $u+\Delta^{(C)}$ is
not stable, then define $\psi(u):=u$.

Otherwise, consider the following order on the subsets of
$A_{m-1}\cup B_n$: if $C$ and $D$ are two distinct subsets of
$A_{m-1}\cup B_n$, we say that $C$ is smaller than $D$ if
$|C|<|D|$, or if $|C|=|D|$ and $C$ is smaller then $D$ in the
lexicographic order induced by the fixed order on the vertices;
i.e. if the smallest element of $C\setminus D$ is smaller than any
element of $D\setminus C$.

Now let $C\subseteq A_{m-1}\cup B_n$ be minimal (in this order)
such that $u+\Delta^{(C)}$ is still stable, and define
$\psi(u):=\sort(u+\Delta^{(C)})$.

The fixed points of this operator are called \emph{recurrent}
sorted configurations (cf. \cite[Theorem 2.4]{addl}).

We define another operator, $\varphi$, which is a sort of a
``dual'' operator to $\psi$. We use the same total order on the
vertices $A_{m-1}\cup B_n$ that we used for $\psi$. Then, given a
stable sorted configuration $u$, if for every nonempty $C\subseteq
A_{m-1}\cup B_n$ the configuration $u-\Delta^{(C)}$ is not stable,
then define $\varphi(u):=u$. Otherwise, let $C\subseteq
A_{m-1}\cup B_n$ be minimal (in the order that we defined above)
such that $u-\Delta^{(C)}$ is still stable, and define
$\psi(u):=\sort(u-\Delta^{(C)})$.

It can be shown that the fixed points of this operator are
exactly the parking  sorted configurations (cf. \cite[Theorem 2.6]{addl}).

It is known that given a configuration $u$ on $K_{m,n}$ there is
exactly one recurrent sorted configuration $v$ and exactly one
parking sorted configuration $w$ equivalent to $u$, i.e. such that
$u \topplingandpermutingequiv{S} v$ and $u
\topplingandpermutingequiv{S} w$ for $S=S_{m-1}\times S_n$ (cf. \cite[Theorem 28 and Corollary 33]{corileborgne}).

It is also known that if we start with a stable sorted
configuration $u$ and if we apply iteratively the operator $\psi$,
then in finitely many steps we will get such a recurrent sorted
configuration $v$, while if we apply iteratively the operator
$\varphi$, then in finitely many steps we will get such a parking
sorted configuration $v$ (cf. \cite[Remark 5.16 and Proposition 5.18]{addl}).

\section{A graphic interpretation of $\varphi$ and $\psi$}

Again, it is convenient to have a pictorial interpretation of
these operators in terms of our diagrams. We refer to \cite[Section 5]{addl}
for details and proofs.\newline

We start by describing the operator $\varphi$. Consider the diagram of a stable sorted configuration $u$ on $K_{m,n}$. We consider the
area below the red path, and the area to the left of the green
path. Their intersection (see Figure \ref{fig:configKmn_2}) will
play an important role. We will call it the \textit{intersection
area}.

\begin{figure}[h]
\includegraphics[width=60mm,clip=true,trim=15mm 200mm 90mm 10mm]{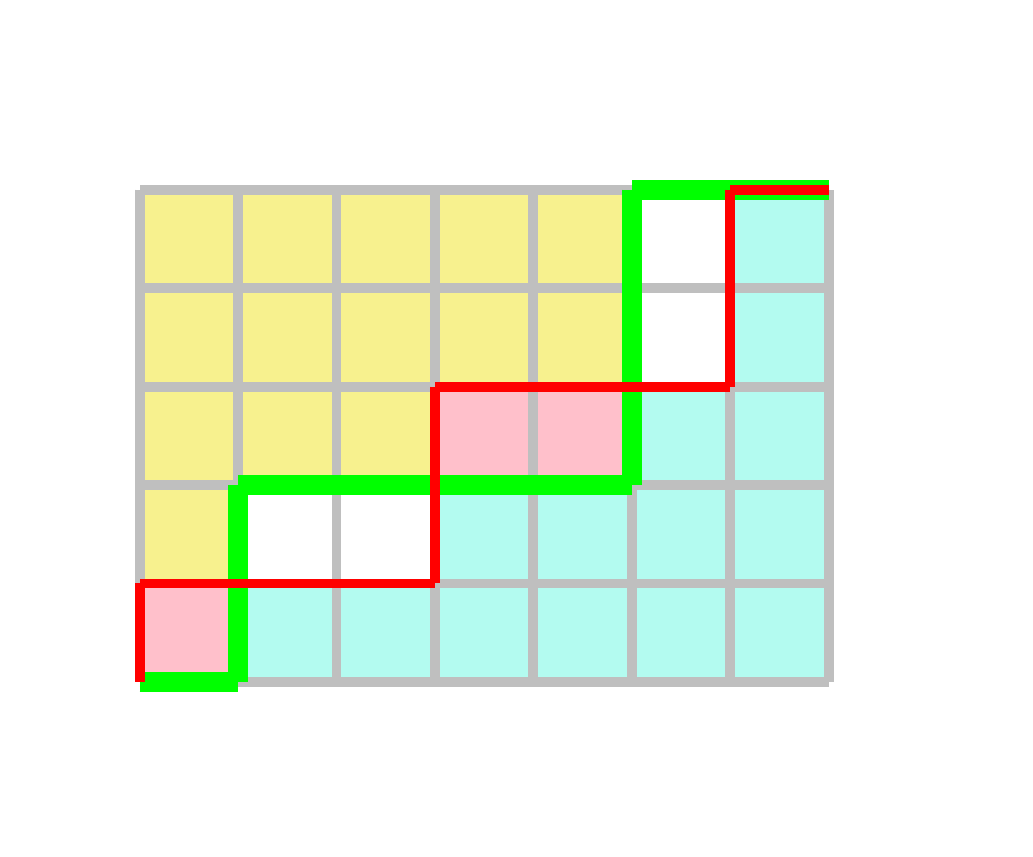}
\caption[ciccia]{The stable configuration
$u=\configuration{0,0,0,2,2,2;*}{0,0,4,4,4}$. The yellow together
with the pink area is the area to the left of the green path; the
light blue together with the pink are is the area below the red
path; the pink area is the intersection area.}
\label{fig:configKmn_2}
\end{figure}

We denote the squares of the grid using integral coordinates in
the obvious way, so that the squares of the grid will have
coordinates $(x,y)$, with $1\leq x\leq m$ and $1\leq y\leq n$.

For example the square in the southwest corner will be the $(1,1)$
square, while the square in the northeast corner will be the
$(m,n)$ square.

Now a square $(i,j)$ will be in the intersection area if and only
if $u_{a_i}\geq j-1$ and $u_{b_j}\geq i-1$.

Notice that this is always the case for the square $(1,1)$.

In order to study the operator $\varphi$, consider the (open) connected components
of the intersection area (here two squares with only one vertex in
common are not considered connected). We can order them by moving
along the red (or equivalently the green) path starting from the
southwest corner. So the first connected component will always be
the one that contains the square $(1,1)$.

If there are no connected components that contain two squares in
the same row, than the configuration $u$ is parking, hence
$\varphi(u)=u$.

If there is at least one connected component that contains two
squares in the same row, than consider the highest occurrence of two adjacent squares in the same row of the intersection area with a red north step crossing their leftmost vertical edge and a green east step crossing their rightmost lower horizontal unit edge, and let $(i,j)$ and $(i+1,j)$  be the coordinates of such squares. We have
that
\begin{equation} \label{eq:varphi_toppling}
\varphi(u)=\sort\left(u-\sum_{k= i}^{m-1}\Delta^{(a_k)}-\sum_{h=
j}^n\Delta^{(b_h)}\right).
\end{equation}

Let us see how this translates into pictures.\newline

Given the diagram of a stable configuration $u :=
\configuration{u_{a_1},u_{a_2},\ldots u_{a_{m-2}},u_{a_{m-1}};
u_{a_m}}{u_{b_1},u_{b_2},\ldots u_{b_n}}$, we draw another
$m\times n$ grid in such a way that the intersection of the new
grid with the old one consists exactly of the upper edge of the
square $(m,n)$ of the new one and the lower edge of the square
$(1,1)$ of the old one. Then, in the new grid we reproduce
precisely the red path as it appears in the old one, except for
the last step, that we omit (so we get a long continuous red
path); while we reproduce in the new grid the green path as it
appears in the old grid, but shifted by one to the left. In this way we have a long red path and a long green
path going through the two grids.

All this is better understood in an example: see Figure
\ref{fig:configKmn_9}.

\begin{figure}[h]
\includegraphics[width=70mm,clip=true,trim=15mm 160mm 50mm 10mm]{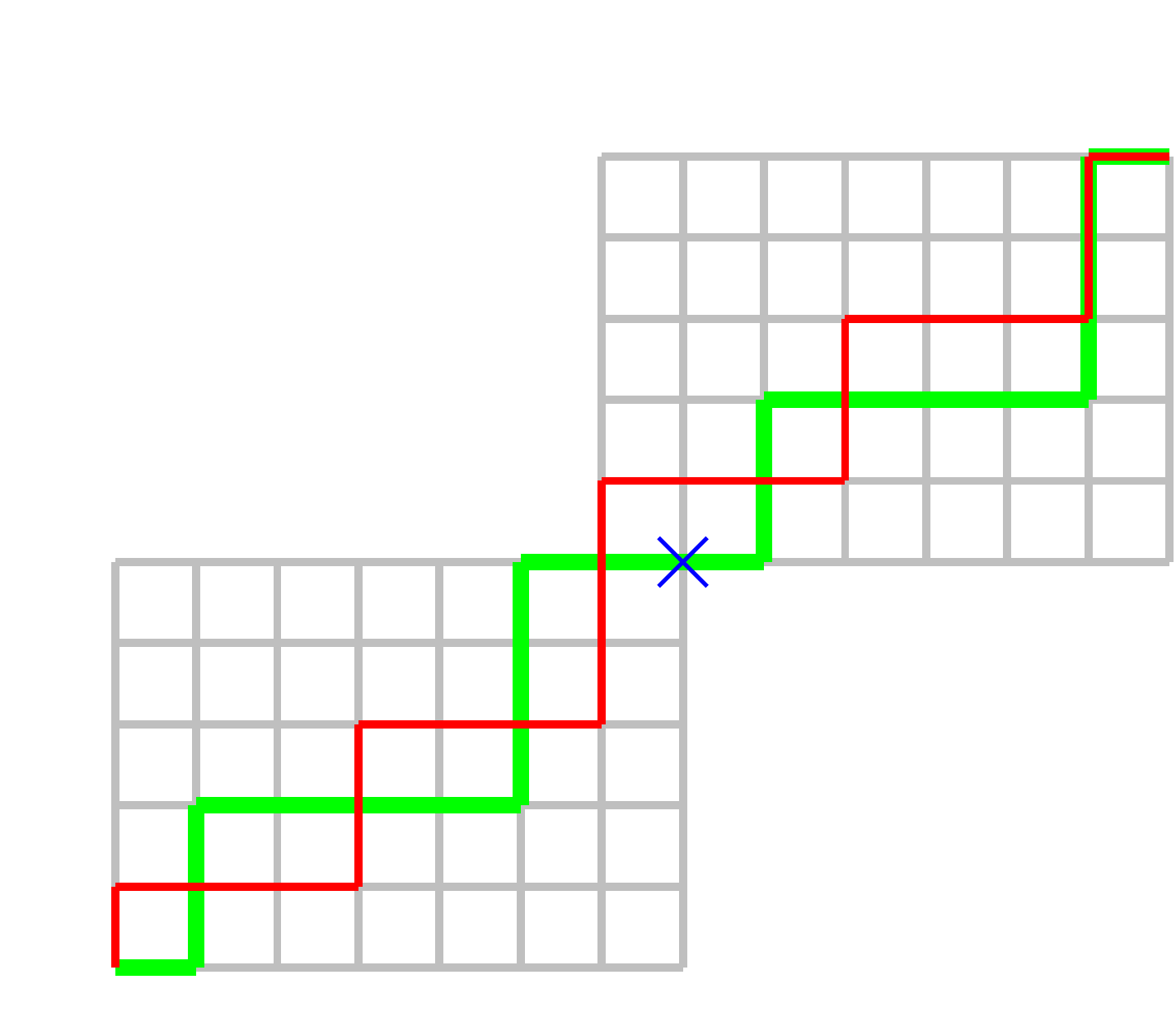}
\caption[ciccia]{The stable configuration
$u=\configuration{0,0,0,2,2,2;*}{1,1,5,5,5}$, with the new grid
and the new paths. The blue cross indicates the northeast corner
of the new grid.} \label{fig:configKmn_9}
\end{figure}

Once we have drawn the new grid with the new paths, we look for
the rightmost occurrence of two adjacent squares in the
intersection area on the same row, say $(i,j)$ and $(i+1,j)$, such
that the squares $(i-1,j)$ and $(i+1,j-1)$ are not in the
intersection area. In other words the left edge of the $(i,j)$
square is a vertical step of the red path, and the lower edge of
the square $(i+1,j)$ is a horizontal step of the green path. We
settle a new $m\times n$ grid with the northeast corner at the
lower-right corner of the square $(i,j)$.

If no such squares exist, then we put the new grid simply where
the old one is.

We claim that the diagram that we get inside this last grid
(completing the paths along the left and upper edges if necessary)
will be the diagram of $\varphi(u)$ (cf. \cite[Theorem 5.11]{addl}). See Figure
\ref{fig:configKmn_10}.

\begin{figure}[h]
\includegraphics[width=70mm,clip=true,trim=15mm 160mm 50mm 10mm]{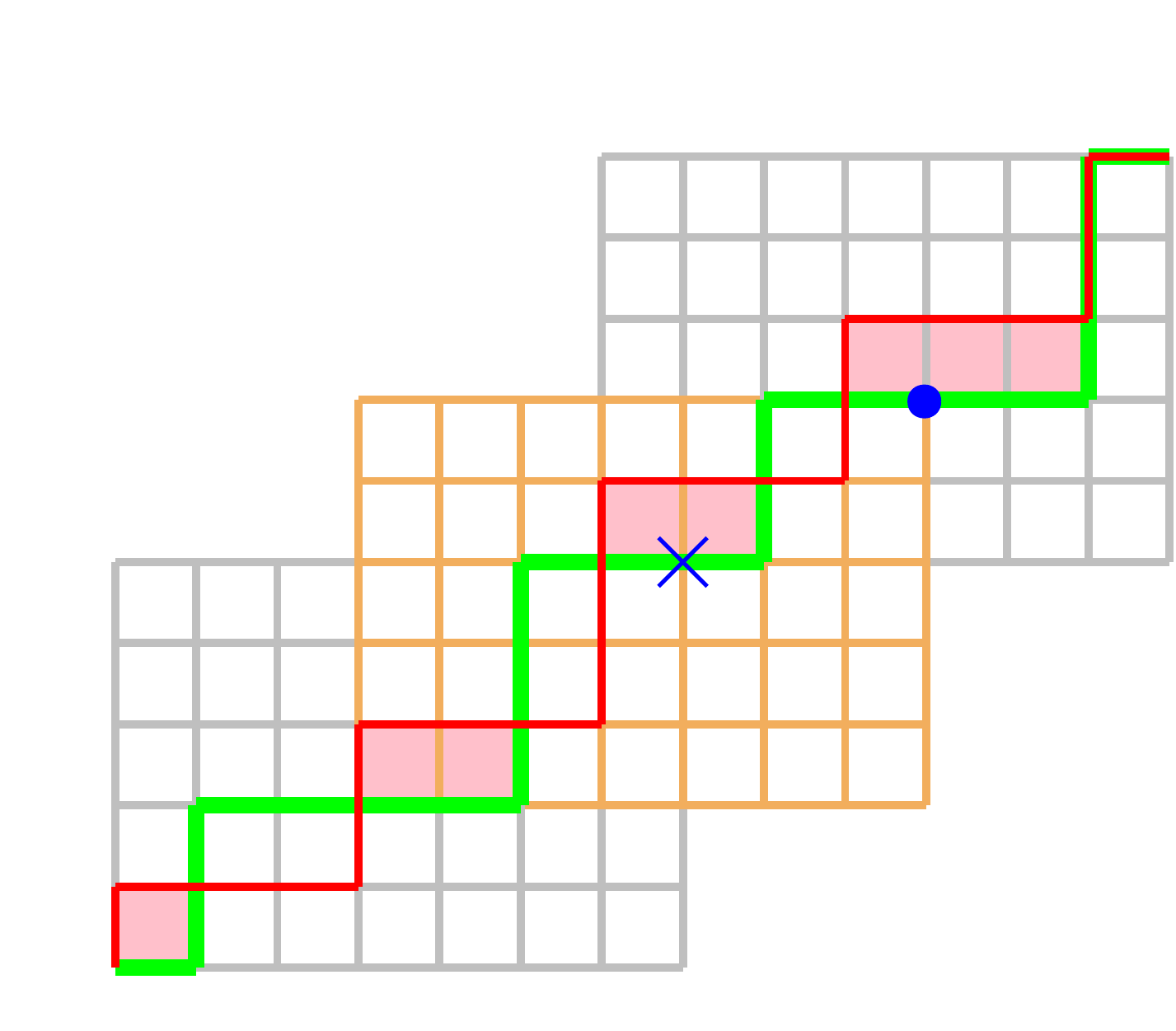}
\caption[ciccia]{The orange grid is the last grid in the
construction, while the blue spot indicates its northeast corner.
The pink area is the intersection area of the long paths.}
\label{fig:configKmn_10}
\end{figure}

The fact that this is indeed the diagram of what we stated in the
previous section to be $\varphi(u)$ is now quite clear from the
picture. This is understood easily by looking at an example: we
can see in Figure \ref{fig:configKmn_10} that
$\varphi(u)=\configuration{0,0,0,3,3,3;*}{1,1,1,4,4}$.

The following remark is crucial.
\begin{remark} \label{rem:periodic}
If iterating the operator $\varphi$ we keep track of the diagrams
that we draw, what we are really doing is to draw periodically the
green path, and to draw periodically the red path with its last
(east) step removed. Then acting with $\varphi$ corresponds simply
to move down (weakly southwest) our $m\times n$ grid stopping by
all the stable sorted configurations that we see along this
periodic picture (cf. \cite[Theorem 5.11]{addl}). The last of such configurations will be
precisely our parking sorted configuration. See Figure
\ref{fig:Yvan1} for an example.
\end{remark}
To make this remark more clear, we complete our graphic
computation by iterating this construction in our running example,
until we get a parking configuration.

So we read in the diagram of the orange grid of Figure
\ref{fig:configKmn_10} that
$\varphi(u)=\configuration{0,0,0,3,3,3;*}{1,1,1,4,4}$.

Iterating we get Figure \ref{fig:configKmn_11}, hence
$\varphi(\varphi(u))=\configuration{0,0,0,2,2,2;*}{0,0,4,4,4}$.

\begin{figure}[h]
\includegraphics[width=70mm,clip=true,trim=15mm 160mm 50mm 10mm]{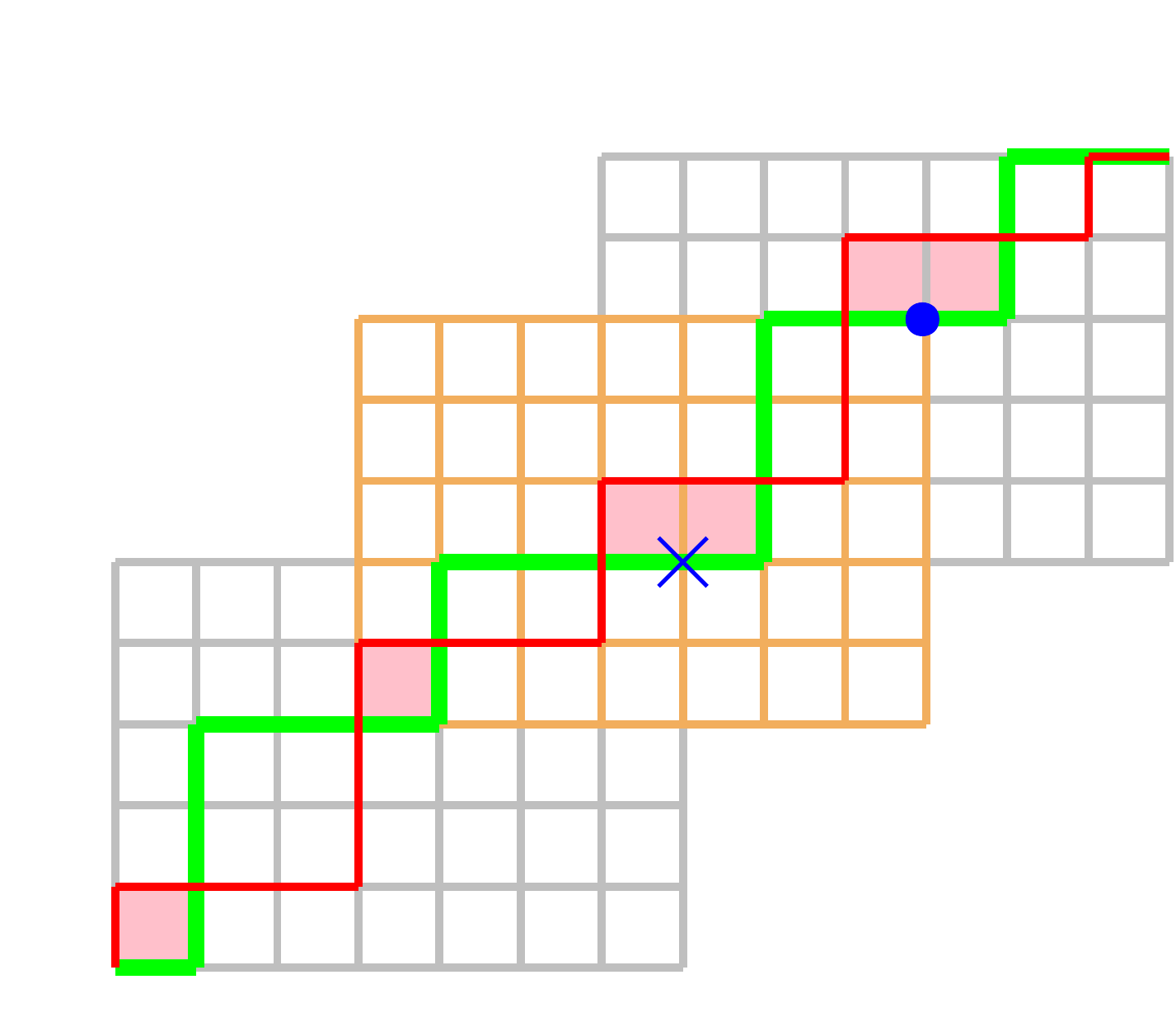}
\caption[ciccia]{This is the diagram of the construction of
$\varphi(\varphi(u))$.} \label{fig:configKmn_11}
\end{figure}

Iterating again we get Figure \ref{fig:configKmn_12}, hence
$\varphi(\varphi(\varphi(u)))=\configuration{0,0,0,3,3,3;*}{0,0,0,3,3}$,
which is parking on $\mathcal{K}_{7,5}$. Indeed
$\varphi(\varphi(\varphi(u)))=\sort (\park(u))$.

\begin{figure}[h]
\includegraphics[width=70mm,clip=true,trim=15mm 160mm 50mm 10mm]{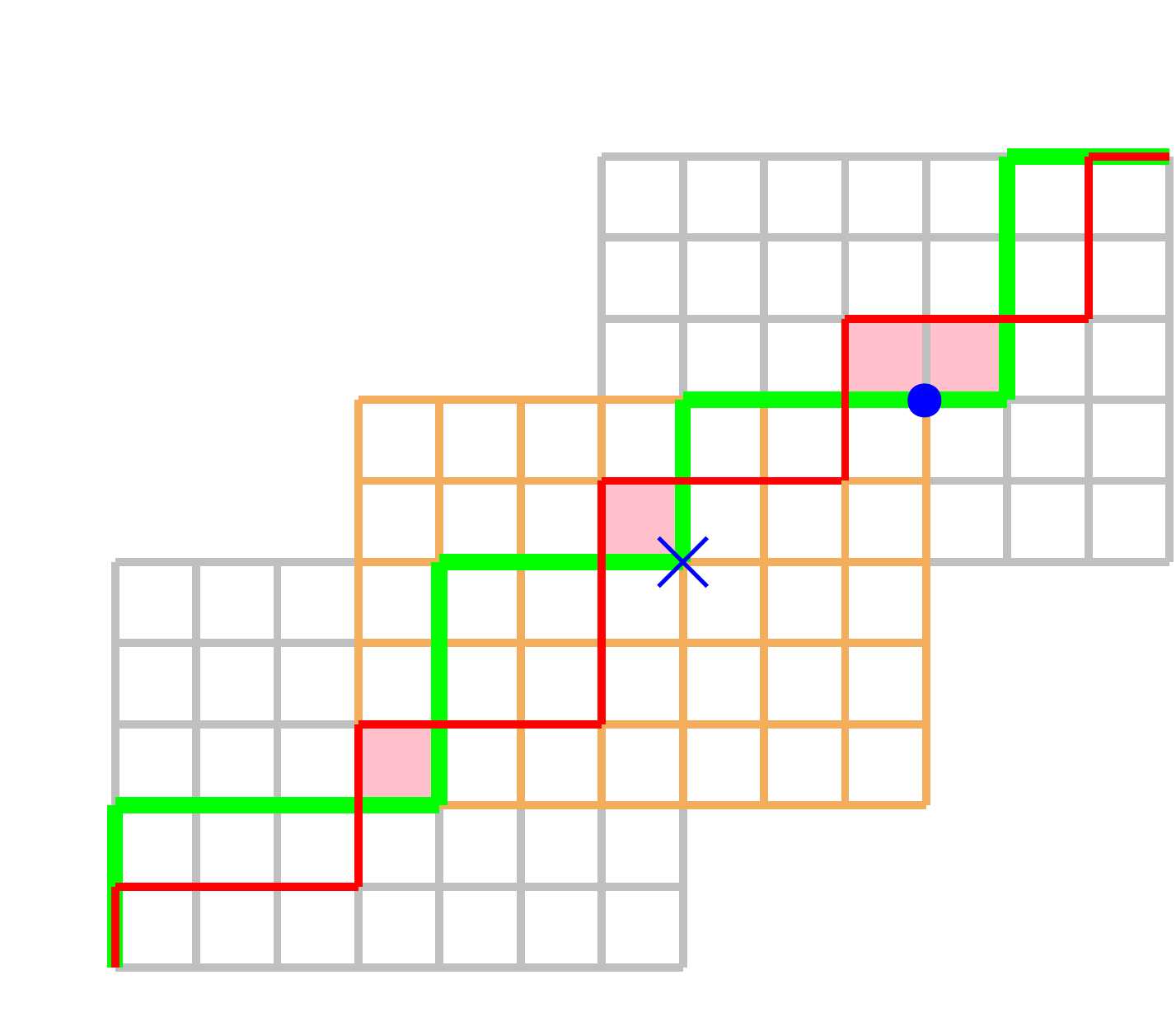}
\caption[ciccia]{This is the diagram of the construction of
$\varphi(\varphi(\varphi(u)))=\configuration{0,0,0,3,3,3;*}{0,0,0,3,3}$.}
\label{fig:configKmn_12}
\end{figure}
\begin{remark}
Notice that we could follow all the steps of this computation on
the periodic diagram mentioned in the Remark \ref{rem:periodic},
just starting with $u$ and sliding down (weakly southwest) our
$m\times n$ grid, reading out all the stable configurations: see
Figure \ref{fig:Yvan1}.

\begin{figure}[h]
\includegraphics[width=130mm,clip=true,trim=25mm 225mm 125mm 25mm]{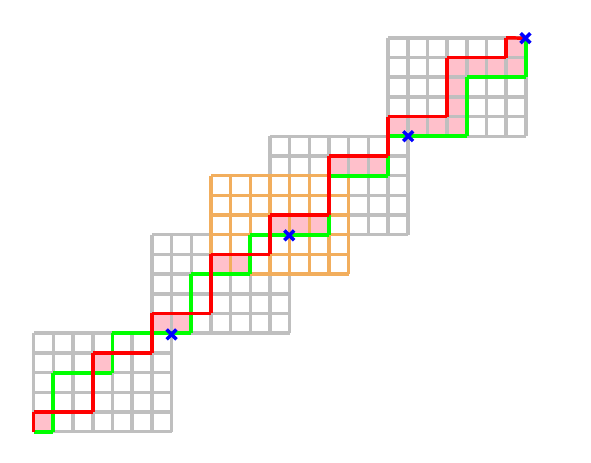}
\caption[ciccia]{This is the periodic diagram of the stable sorted
configuration  $u=\configuration{0,0,0,2,2,2;*}{1,1,5,5,5}$ (which you can see in the orange grid).}
\label{fig:Yvan1}
\end{figure}
\end{remark}

Observe that our pictorial description gives the following
characterization of the parking sorted configurations on
$\mathcal{K}_{m,n}$ (cf. \cite[Corollary 5.15]{addl}).
\begin{proposition} \label{prop:park_sort}
The parking sorted configurations on $\mathcal{K}_{m,n}$ are
precisely the stable sorted configurations on $\mathcal{K}_{m,n}$ whose intersection area does not contain two squares in the same row.
\end{proposition}

It can be shown that the action of the operator $\psi$ is
essentially the inverse of the action of $\varphi$: in the
periodic diagram mentioned in Remark \ref{rem:periodic}, the
action of $\psi$ is precisely to move up (weakly northeast) our
$m\times n$ grid stopping by all the stable sorted configurations
that we see along this periodic picture (cf. \cite[Theorem 5.13 and Remark 5.16]{addl}).

The last of such configurations will be the corresponding
recurrent sorted configuration. See Figure \ref{fig:Yvan1} for an
example.

This gives the following characterization of recurrent sorted
configurations on $K_{m,n}$ (cf. \cite[Theorem 3.7]{dukesleborgne}).
\begin{proposition} \label{prop:sorted_recurrent}
The recurrent sorted configurations on $\mathcal{K}_{m,n}$ are
precisely the stable sorted configurations on $\mathcal{K}_{m,n}$ whose intersection area is connected and it touches both the southwest and the northeast corners of the
$m\times n$ grid.
\end{proposition}

\section{The operators $\Ta$ and $\Tb$}

For our analysis it will be convenient to introduce two operators
that will explain the meaning of moving our $m\times n$ grid on
the periodic diagrams. We refer again to \cite{addl} for details
and proofs (though we will use a slightly different
notation).\newline

For any sorted configuration $u$ we define two operators:
$$
\Ta(u):= \sort(u+\topplingconfiguration{a_{1}}) \mbox{ and }
\Tb(u) := \sort(u+\topplingconfiguration{b_1}).
$$

A configuration $u$ is \emph{compact} if $\displaystyle \max_{a_i
\in A_{m-1}}u_{a_{i}}-\min_{a_i \in A_{m-1}}u_{a_{i}} \leq n$ and
$\displaystyle \max_{b_i \in B_{n}}u_{b_{j}}-\min_{b_j \in
B_{n}}u_{b_{j}} \leq m$.

The main examples of compact configurations that we will encounter
in our work will be configurations with $-1\leq u_{a_i}\leq n-1$
and $-1\leq u_{b_i}\leq m-1$ for all $i$ and $j$.

The following lemma follows immediately from the definitions.
\begin{lemma} \label{lem:Ta_Tb}
The operators $\Ta$ and $\Tb$ are invertible on compact sorted
configurations. More precisely, for any compact sorted
configuration $u$, we have $\Ta^{-1}(u)=\sort (u-\Delta^{(a_{m-1})})$
and $\Tb^{-1}(u)=\sort(u-\Delta^{(b_{n})})$. In formulae
\begin{align*}
\Ta(u) & = \configuration{u_{a_2},u_{a_3},\ldots,
u_{a_{m-1}},u_{a_1}+n;u_{a_m}}{u_{b_1}-1,u_{b_2}-1,\ldots,
u_{b_n}-1};\\
 \Ta^{-1}(u) & =
\configuration{u_{a_{m-1}}-n,u_{a_1},u_{a_2},\ldots, u_{a_{m-1}};
u_{a_m}}{u_{b_1}+1,u_{b_2}+1,\ldots, u_{b_n}+1}
\end{align*}
and
\begin{align*}
\Tb(u) & = \configuration{u_{u_{a_1}}-1,u_{a_2}-1,\ldots,
u_{a_{m}}-1;u_{a_m}-1}{u_{b_2},u_{b_3},\ldots,
u_{b_{n}},u_{b_1}+m};\\
 \Tb^{-1}(u) & =
\configuration{u_{u_{a_1}}+1,u_{a_2}+1,\ldots, u_{a_{m}}+1;
u_{a_m}+1}{u_{b_n}-m,u_{b_1},u_{b_2},\ldots, u_{b_{n-1}}}.
\end{align*}
\end{lemma}
In terms of our graphical description, the operator $\Ta$
corresponds to move our $m\times n$ grid of $1$ step in direction
east. Similarly the operator $\Tb$ corresponds to move the grid of
$1$ step in direction north. The graphical meaning of $\Ta^{-1}$
and $\Tb^{-1}$ is now clear (cf. \cite[Lemma 5.7]{addl}, where $\Ta^{-1}$ and $\Tb^{-1}$ are denoted by $T^{>n}$ and $T^{\leq n}$ respectively). Therefore our operators $\varphi$ and
$\psi$ can be seen easily as suitable iterations of the operators
$\Ta^{-1}$ and $\Tb^{-1}$, and $\Ta$ and $\Tb$ respectively. For example, if $u$ is a stable sorted configuration on $K_{m,n}$ and $(i,j)$ is the cell of the diagram of $u$ described right before equation \eqref{eq:varphi_toppling}, then $\varphi(u)=\Ta^{-m+i}\Tb^{-n+j-1}(u)$

We
will use these observations later in our work.\newline

We conclude this section by recalling that any compact sorted
configuration $u$ on $K_{m,n}$ is well described by
$\sort(\park(u))$ and two integers (cf. \cite[Proposition 5.12]{addl}).

\begin{proposition}
For any compact sorted configuration $u$ on $K_{m,n}$, we have a
unique pair $(k_a,k_b)\in\mathbb{Z}^2$ such that
$$
u = \Ta^{k_a}\Tb^{k_b}\sort(\park(u)).
$$
\end{proposition}

\section{A greedy algorithm for the rank on $K_{m,n}$:
correctness}\label{sec:correctness}

The aim of this section is to prove Theorem \ref{thm:correctness},
i.e. the correctness of our algorithm.

We start with a lemma.

\begin{lemma}\label{lem:0parking}
If $u$ is a parking configuration on $K_{m,n}$ (with respect to the sink $a_m$), then there exists $b_i \in B_n$ such that $u_{b_i}=0$.
\end{lemma}
\begin{proof}
The complete bipartite graph $K_{m,n}$ is a particular case of a simple
graph without loops, i.e. there is at most one edge between two of
its vertices and the endpoints of any of its edges are distinct.

We prove by contradiction the more general result that in a
parking configuration $u$ on a simple graph without loops, at
least one vertex $c_k$ among the neighbors $N(c_s)$ of the sink
$c_s$ has value $u_{c_k}=0$.

Since in $K_{m,n}$ the neighbors $N(a_m)$ of the sink $a_m$ form
exactly the set $B_n$, this would imply our result.

Let $C$ be the set of vertices of our graph, and for any $c_k\in
C$ let us denote by $\topplingconfiguration{c_k}$ the
corresponding toppling operator of the vertex $c_k$. Notice that
the Remark \ref{rem:Delta} is still valid in this more general
setting.

Now, if $u_{c_k} \neq 0$ for any $c_k \in N(c_s)$, since $u$ is
parking, it is non-negative outside the sink $c_s$, hence $u_{c_k}
\geq 1$ for $c_k \in N(c_s)$. But in this case the configuration
$$
u-\Delta^{(C-\{c_s\})}= u-\sum_{c_h \in C-\{c_s\}}
\topplingconfiguration{c_k}=u+\topplingconfiguration{c_s}
$$
is also non-negative outside the sink $c_s$, since only each
neighbor of $c_s$ loses exactly $1$. This contradicts the
assumption that $u$ is parking.
\end{proof}

As we already observed in Section \ref{sec:stable_sorted}, \underline{we} can always \underline{assume that our configurations are sorted}. So from now on we will do it. We will observe later (in Section \ref{sec:optimization}) that sorting does not increase the overall complexity of the algorithm. 

\begin{remark}
Notice that Lemma \ref{lem:0parking} is an immediate consequence
of our pictorial description of parking sorted configurations: the
square $(1,1)$ is always in the intersection area, and there
cannot be two squares in the same row, therefore the first two
steps of the green path (starting from the southwest corner) must
be east and then north, hence the value at $b_1$ must be $0$. We gave anyway our
proof, since it works in a more general setting.
\end{remark}

It follows from Lemma \ref{lem:0parking} that in a parking sorted
configuration $u$ on $K_{m,n}$ we must have $u_{b_1}=0$. Hence in
the \textbf{while} of our algorithm we can always take $i=1$.

So for our analysis we need to understand how to compute $\sort
(\park(u-\diracconfiguration{b_1}))$ for a parking sorted
configuration $u$ on $K_{m,n}$. As usual, a pictorial
interpretation of this operation will be very useful.

First of all, recall that in the intersection area of the diagram
of a parking sorted configuration there cannot be two squares in
the same row, and there is always the square $(1,1)$. Now notice
that the diagram of $u-\diracconfiguration{b_1}$ can be obtained
from the diagram of $u$ by simply replacing the first two steps
(starting from the southwest corner) of the green path, which have
to be east and then north, by the steps north and then east. See Figure \ref{fig:algo_1_2} for
an example.

\begin{figure}[h]
\includegraphics[width=60mm,clip=true,trim=15mm 200mm 90mm 10mm]{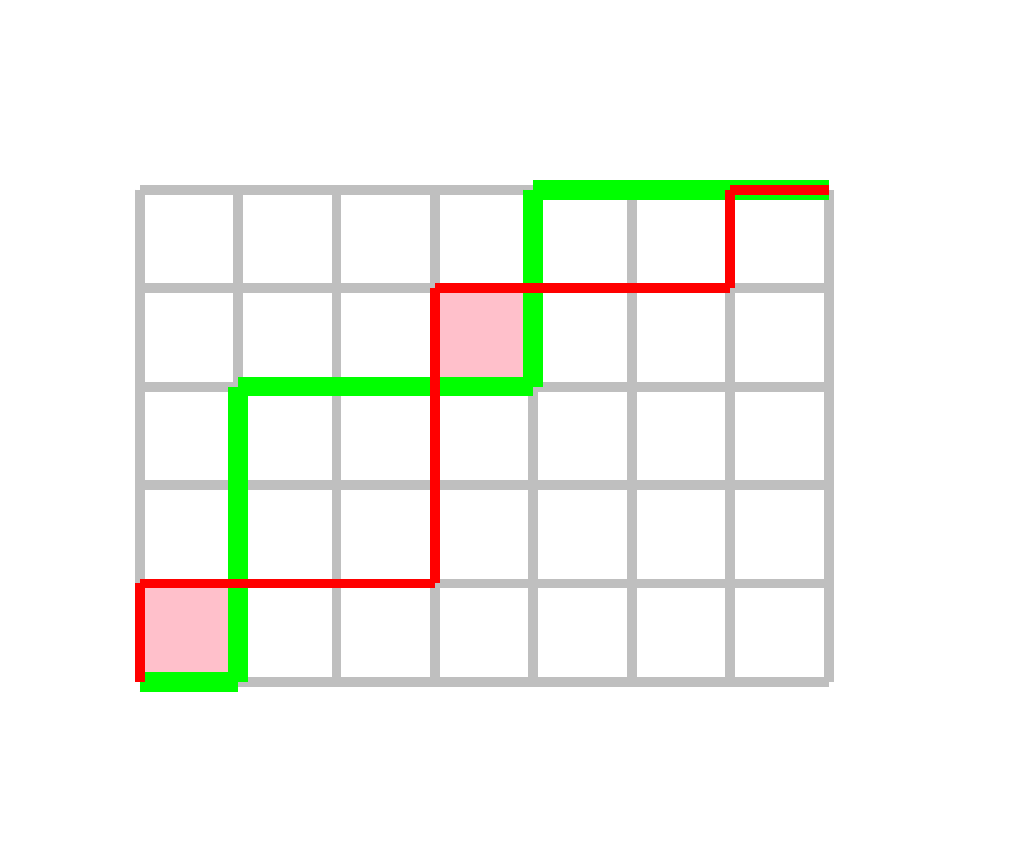}
\includegraphics[width=60mm,clip=true,trim=15mm 200mm 90mm 10mm]{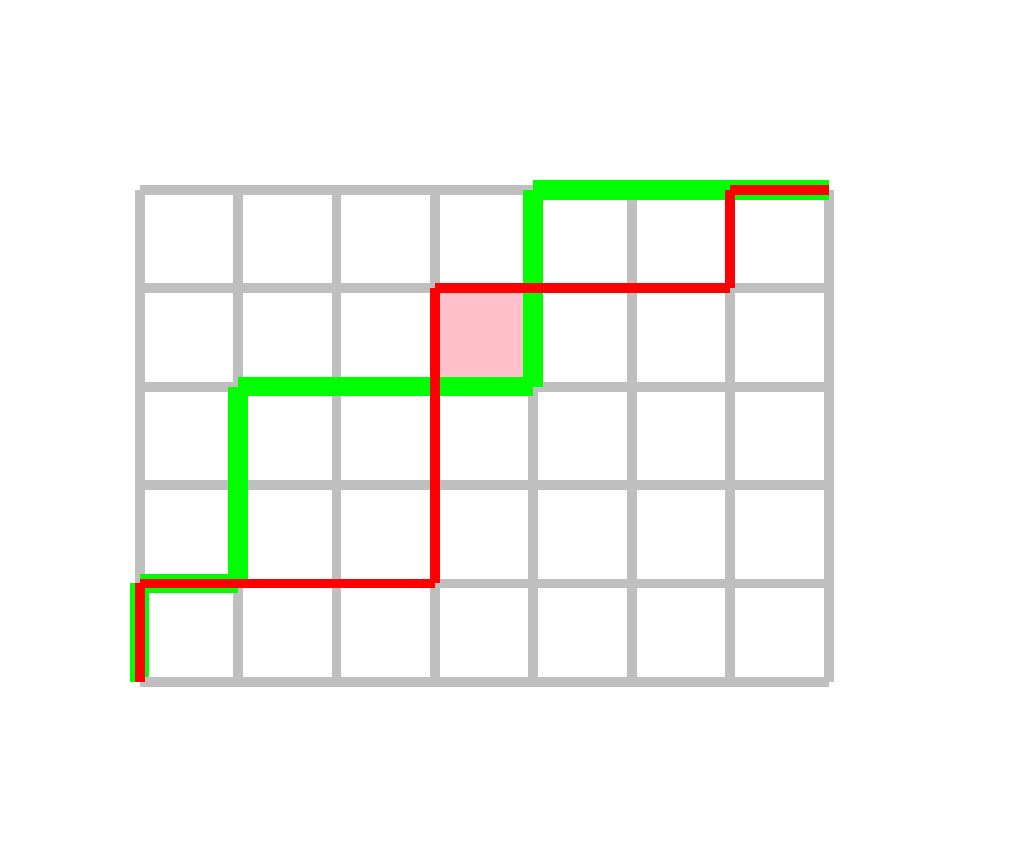}
\caption[ciccia]{On the left, the parking sorted configuration
$u=\configuration{0,0,0,3,3,3;*}{0,0,0,3,3}$. On the right, the
configuration
$u-\diracconfiguration{b_1}=\configuration{0,0,0,3,3,3;*}{-1,0,0,3,3}$.}
\label{fig:algo_1_2}
\end{figure}

Notice that the graphical interpretation of the operators
$\varphi$, $\psi$, $\Ta$ and $\Tb$ works also for the diagrams of
$u-\diracconfiguration{b_1}$, as this is still a compact
configuration. In particular, recall from Remark
\ref{rem:periodic} that in order to find the parking sorted
configuration equivalent to $u-\diracconfiguration{b_1}$ we need
to look at the periodic diagram obtained by reproducing
periodically the green path and the red path without its last east
step (next to the northeast corner) of the diagram of
$u-\diracconfiguration{b_1}$. Since $u$ was parking, it is clear
from the graphic interpretation that the parking sorted
configuration of the periodic diagram is obtained by moving up
(weakly northeast) the frame to reach the first stable
configuration (moving down we will not encounter any stable
configuration): this is going to be our $\sort(
\park(u-\diracconfiguration{b_1}))$. See Figure \ref{fig:algo_3}
for the example of the computation of $\sort (
\park(u-\diracconfiguration{b_1}))$ of Figure \ref{fig:algo_1_2}.
\begin{figure}[h]
\includegraphics[width=70mm,clip=true,trim=15mm 160mm 50mm 10mm]{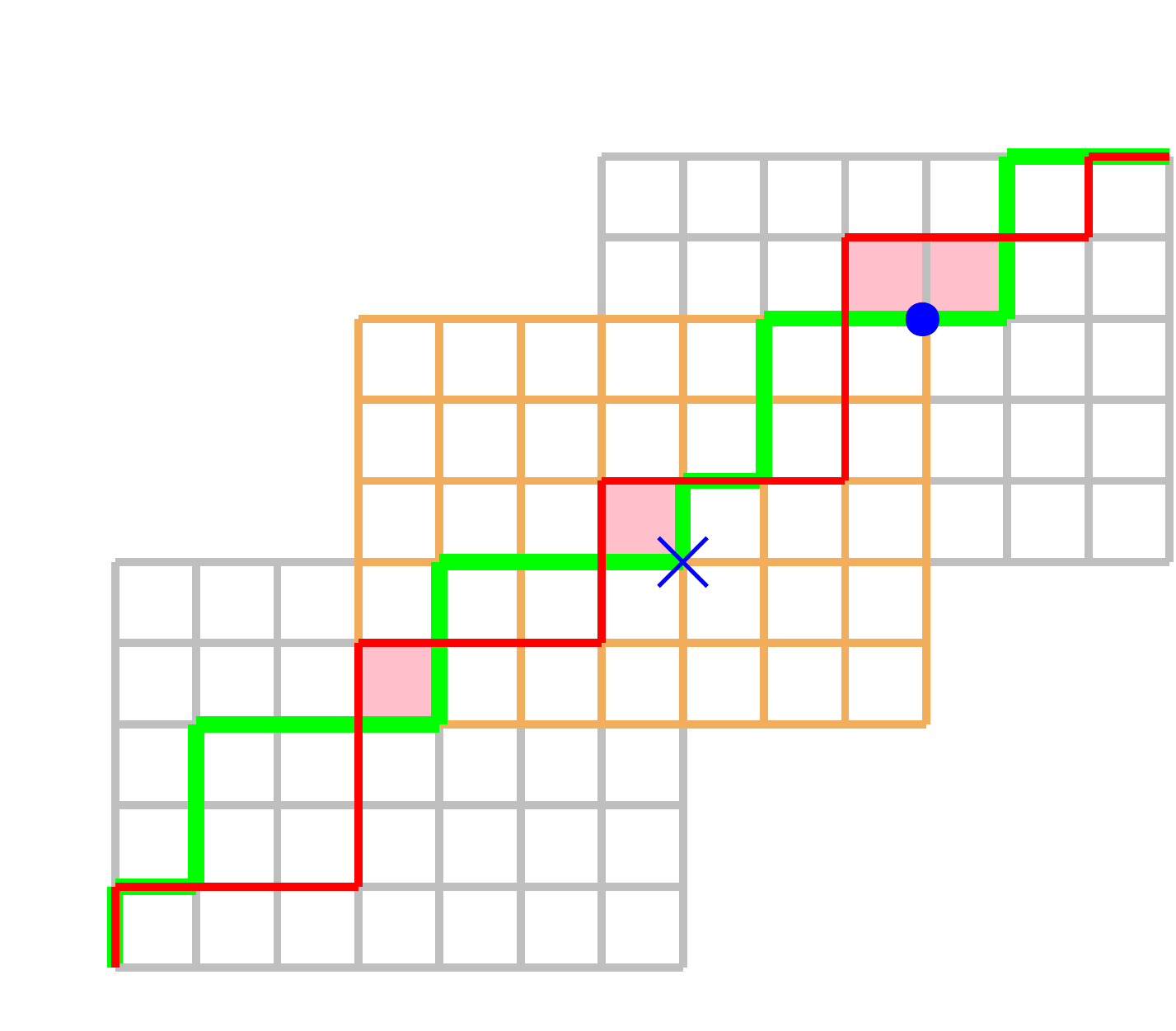}
\caption[ciccia]{This is the diagram of the construction of
$\sort(
\park(u-\diracconfiguration{b_1}))=\configuration{0,0,0,2,2,2;*}{0,0,3,4,4}$.}
\label{fig:algo_3}
\end{figure}

This provides the pictorial interpretation we were looking
for.\newline

In order to prove the correctness of our algorithm we need one
more definition and two lemmas.\newline

The \emph{support} of a non-negative configuration $f$ is the set
of vertices $\{c_k\}_k\subseteq A_m\cup B_n$ such that
$f_{c_k}>0$.

\begin{lemma}\label{lem:restrictedsupport}
Let $u$ be a configuration on $K_{m,n}$. Then there exists a proof $f$ for the rank of $u$ whose support is included in $B_n$.
\end{lemma}

\begin{proof}
Let $H(u)=\min \{ \sum_{a_i\in A_m} f_{a_i} | \mbox{ $f$ is a
proof for the rank of $u$}\}$.

We show by contradiction that $H(u)=0$, implying the lemma. We
assume that $H(u)>0$ and we consider a proof $f$ for the rank of
$u$ such that $\sum_{a_i \in A_m} f_{a_i} = H(u)$. Since $H(u)\geq
1$ there exists an $i$ such that $f_{a_i}\geq 1$. Consider the
configuration $g:=f-\diracconfiguration{a_i}$, which is still
non-negative, but it is not a proof for the rank of $u$. So in
particular, by Theorem \ref{thm:effective_park_nonnegative}, $\park(u-g)$ must be non-negative. Now if
$\park(u-g)-\diracconfiguration{a_i}$ is still non-negative, than, because of Proposition \ref{prop:park_sort}, it must be still parking,
and in fact $\park(u-g)-\diracconfiguration{a_i}=\park(u-f)$,
contradicting the fact that $f$ is a proof for the rank of $u$. Therefore we must have
that the value of $\park(u-g)$ at $a_i$ is $0$.

Set $v:=\sort(\park(u-g))$. We just saw that $v_{a_1}=0$. As $v$
is a parking sorted configuration on $K_{m,n}$, by Lemma
\ref{lem:0parking} we must have $v_{b_1}=0$.

So in the diagram of $v$ the box $(1,1)$ constitutes its own
connected component of the intersection area.

The crucial observation is that when we compute
$\sort(\park(v-\diracconfiguration{a_1}))$ and
$\sort(\park(v-\diracconfiguration{b_1}))$, in both cases in our
periodic diagram we are going to move our $m\times n$ grid up
(strictly northeast), moving the southwest corner of the grid onto
the southwest corner of the next connected component of the
intersection area, and this corner will be exactly the same in the
two computations. It turns out that the number removed from the
sink in this operation is exactly the height of this corner: in
fact this is the number of iterations of the operator $\Tb$ needed
to reach the corner, each of which decreases by $1$ the value at
the sink (the number of iterations of the operator $\Ta$ needed to
reach the corner is irrelevant, as $\Ta$ does not change the value
at the sink).

All this shows that
$\sort(\park(u-f))=\sort(\park(v-\diracconfiguration{a_1}))$ and
$\sort(\park(u-f+\diracconfiguration{a_i}-\diracconfiguration{b_j}))=\sort(\park(v-\diracconfiguration{b_1}))$ for a suitable $j$, and they both
have the same value at the sink, so this value must be negative in both
cases. Therefore also
$\widetilde{f}:=f-\diracconfiguration{a_i}+\diracconfiguration{b_j}$
is a proof for the rank of $u$. But $\sum_{a_i\in
A_{m}}\widetilde{f}=H(u)-1$, giving a contradiction. Therefore we
must have $H(u)=0$ as claimed.
\end{proof}

\begin{lemma}\label{lem:greedychoice}
Let $u$ be a configuration on $K_{m,n}$. If $u$ is non-negative and $u_{b_i}=0$, then there exists a proof
$g$ for the rank of $u$ such that $g_{b_i} > 0$.
\end{lemma}
\begin{proof}
Let $f$ be a proof for the rank of $u$ of support included in
$B_n$, whose existence is guaranteed by
Lemma~\ref{lem:restrictedsupport}. If $f_{b_i} > 0$, then $g=f$
has the stated properties. Otherwise, we have $f_{b_i}=0$. Since
$u-f$ is non-effective and the support of $f$ is in $B_n$, there
exists $b_j$  ($\neq b_i$) such that $u_{b_j}-f_{b_j} < 0$. Set
$\widetilde{f}:=f-(f_{b_j}-u_{b_j})\diracconfiguration{b_j}$ and
$w:=u-\widetilde{f}$. Notice that $\widetilde{f}$ is still
non-negative, and, by definition, we have $w_{b_i} = 0 = w_{b_j}$.

We claim that $(f_{b_j}-u_{b_j})\diracconfiguration{b_j}$ is a
proof for the rank of $w$. Indeed, $w-(f_{b_j}-u_{b_j})\diracconfiguration{b_j}=u-f$ is non-effective, and if there is a non-negative configuration $g$ with 
$$
\degree(g)<\degree((f_{b_j}-u_{b_j})\diracconfiguration{b_j})=(f_{b_j}-u_{b_j})
$$ 
for which $w-g$ is non-effective, then $h:=f-(f_{b_j}-u_{b_j})+g$ is non-negative, 
$$
\degree(h)<\degree(f)=\rank(u)+1
$$ 
and $w-g=u-h$ is non-effective, contradicting the definition of $\rank(u)$.

Now, since $w_{b_i} = 0 = w_{b_j}$,
using the symmetry exchanging exactly $b_i$ and $b_j$, we deduce
that also $(f_{b_j}-u_{b_j})\diracconfiguration{b_i}$ is a proof
for the rank of $w$. Therefore
$g:=\widetilde{f}+(f_{b_j}-u_{b_j})\diracconfiguration{b_i}$ is
also a proof for the rank of $u$, and by construction $g_{b_i}>0$, as we wanted.
\end{proof}

We are now ready to prove the correctness of our algorithm, i.e.
to prove Theorem \ref{thm:correctness}.

\begin{proof}[Proof of Theorem \ref{thm:correctness}]
The run of the algorithm is possible since the vertex $b_i$ in the
\textbf{while} loop always exists according to
Lemma~\ref{lem:0parking}. This algorithm terminates since each
loop iteration decreases the degree of the configuration by $1$
and a configuration of negative degree is necessarily
non-effective. Hence it remains to show that $f$ is a proof for
the rank of the input $u$, where $(rank,f)$ is the output of the
algorithm. Since it is clear that $f$ is non-negative and $rank =
\degree(f)-1$, it remains to show that $rank$ is the actual rank
of the input $u$.

We prove the correctness by induction on the total number $k$ of
iterations of the \textbf{while} loop. If $k=0$, then the first
computed parking configuration is non-effective and $f=0$ is the
expected proof for the rank of $u$, which is indeed $-1$ in this
case.

If $k>0$, then the first loop iteration removes $1$ from the value
of $\park(u)$ at some vertex $b_i$, and, by inductive hypothesis,
the following $k-1$ iterations compute a proof
$f':=f-\diracconfiguration{b_i}$ for the rank of
$u':=\park(u)-\diracconfiguration{b_i}$. Let $g$ be a proof for
the rank of $\park(u)$ such that $g_{b_i} > 0$, which exists by
Lemma~\ref{lem:greedychoice}.

Observe that $g':=g-\diracconfiguration{b_i}$ is a proof for the
rank of $u'=\park(u)-\diracconfiguration{b_i}$. Indeed, on one
hand we have that $u'-g'=\park(u)-g$ is non-effective, hence
$$
\rank(u')\leq \degree(g')-1=(\degree(g)-1)-1=\rank(\park(u))-1;
$$
on the other hand, for any configuration $w$ on $K_{m,n}$, we have
the obvious inequality $\rank(w+\diracconfiguration{b_i})\leq
\rank(w)+1$, hence
$$
\rank(\park(u))\leq
\rank(\park(u)-\diracconfiguration{b_i})+1=\rank(u')+1.
$$
Therefore $\rank(u')=\rank(\park(u))-1=\degree(g')-1$, so $g'$ is
a proof of $u'$ as claimed.

Hence, by definition of proofs for the rank of
$u'=\park(u)-\diracconfiguration{b_i}$, we must have
$\degree(g')=\degree(f')$. As a consequence
$\degree(g)=\degree(f)$, so $f$ is also a proof for the rank of
$\park(u)$ as expected.
\end{proof}

\section{A useful reformulation of the algorithm} \label{sec:reformulation}

Our main goal is to provide an efficient algorithm to compute the
rank on $K_{m,n}$. To achieve this we need to provide a detailed
analysis of our original algorithm, so that we will be able to
optimize it.

In order to do this, it will be convenient to reformulate our
algorithm in terms of the operators $\Ta$ and $\Tb$. We need some
more definitions.

\emph{The cell of a north step} of a path is the cell whose east
side is the mentioned step. In our periodic diagram of a stable
sorted configuration $u$ on $K_{m,n}$, the periodic red path
(whose period consists of $m+n-1$ steps) is called \emph{the cut},
and it disconnects the cells of the underlying infinite grid into
a \emph{left component} and a \emph{right component}.

We claim that, as long as we are interested in computing the rank
of a configuration $u$ on $K_{m,n}$, the following algorithm is
equivalent to our original algorithm, i.e. it computes the rank of
$u$ using the same principles.

\begin{lstlisting}[frame=single,texcl,mathescape]
def new_compute_rank(u):
  $u$ = $\sort(\park(u))$
  $rank$ = $-1$
  while $u_{a_m}$ >= $0$:
    while $u_{b_1}$ >= $0$:
      $u$ = $\Ta(u)$
    $u$ = $\Tb(u)$
    if $u_{a_{m-1}}$ >= $n-1$:
      $rank$ = $rank+1$
    $u_{a_m}$ = $u_{a_m}-1$ # each iteration of $\Tb$ takes $1$ from the sink
  return $rank$
\end{lstlisting}

\begin{theorem} \label{thm:new_correctness}
This new algorithm computes the rank of a configuration $u$ on
$K_{m,n}$.
\end{theorem}

In the rest of this section we will give a proof of Theorem
\ref{thm:new_correctness}, via the following pictorial
interpretation of it.\newline

Recall that in terms of our periodic diagrams, the operators $\Ta$
and $\Tb$ correspond to move our $m\times n$ grid of $1$ step in
directions east and north respectively.

First of all we claim that what this algorithm does is to move the
grid in such a way that its southwest corner moves along the
periodic green path (whose period consists of $m+n$ steps) in
direction northeast, following it step by step; it decreases the
value $u_{a_m}$ in the sink at every iteration of $\Tb$, and it
increases the value in $rank$ after every iteration of $\Tb$ for
which the southwest corner of the $m\times n$ grid crosses a north
step whose cell is in the right component, and in fact in the
intersection area.

The only non-obvious statement is the last one, which is proved
once we show that the condition in the \textbf{if}, i.e.
$u_{a_{m-1}}\geq n-1$, corresponds exactly to the fact that the
cell of the just crossed north step of the periodic green path is
in the right component. But this follows simply from the
observation that the period of the red path consists of $m+n-1$
steps, hence the horizontal red step in the column $m-1$ of our
$m\times n$ grid is precisely $n$ cells higher than the horizontal
red step in the column to the left of the $m\times n$ grid. This
is easily understood by looking at the diagrammatic example in
Figure \ref{fig:Yvan2}.
\begin{figure}[h]
\includegraphics[width=70mm,clip=true,trim=15mm 170mm 50mm 25mm]{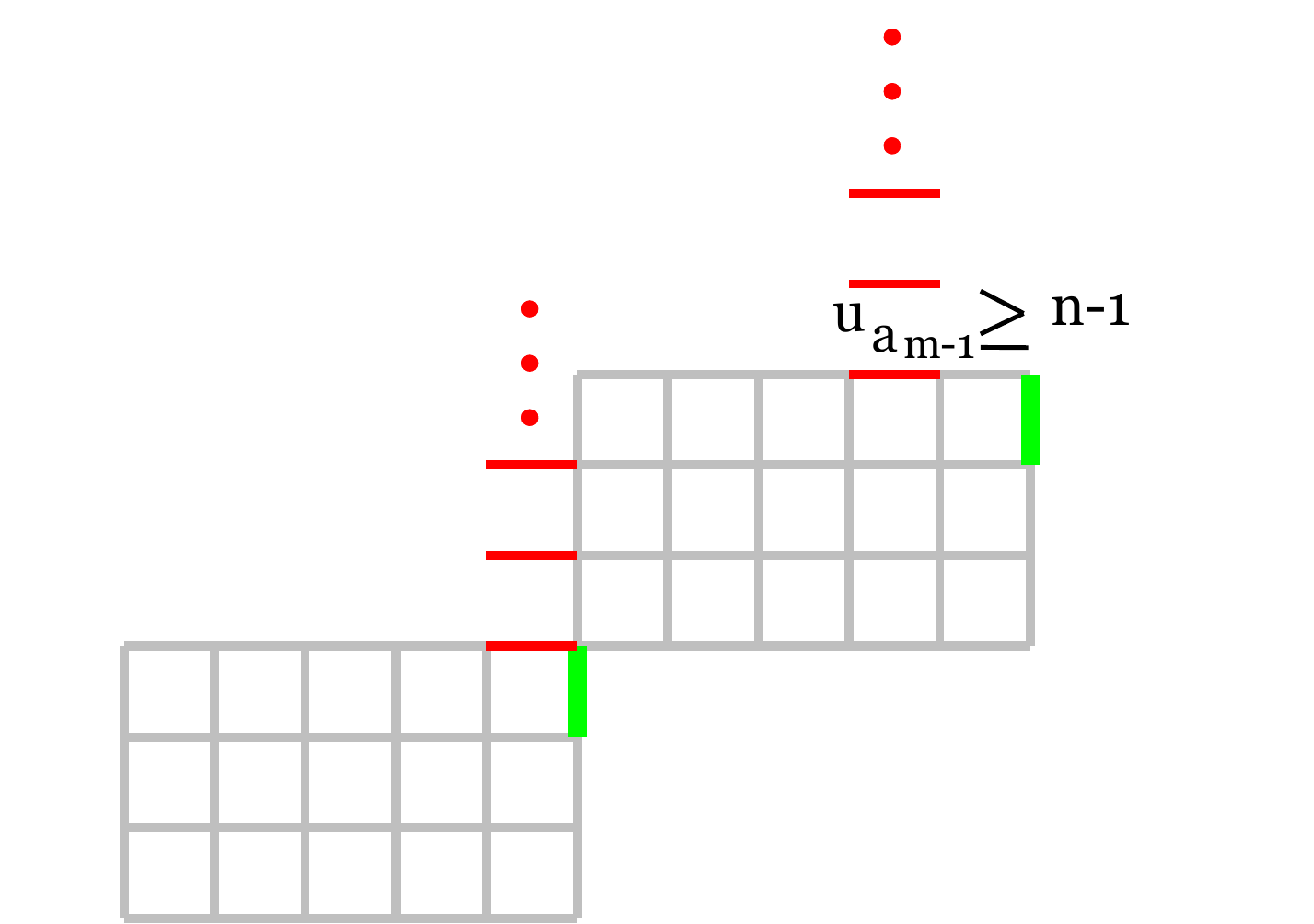}
\caption[ciccia]{This diagrammatic example on $K_{5,3}$ explains
how the condition $u_{a_{m-1}}\geq n-1$ translates into pictures.}
\label{fig:Yvan2}
\end{figure}

Now that we understand the pictorial interpretation of this new
algorithm, we need to understand why this is equivalent to our
original algorithm.

Given a parking sorted configuration $u$ on $K_{m,n}$, let us set
$$
\psi_0(u):=\sort(\park(u-\diracconfiguration{b_1})).
$$
Recall that computing $\psi_0(u)$ corresponds to a loop iteration
of our original algorithm. We already gave a graphical description
of the action of the operator $\psi_0$ in Section
\ref{sec:correctness}. We saw that applying $\psi_0$ to a parking
stable configuration $u$ corresponds to replacing the first two
steps (starting from the southwest corner) of the green path,
which have to be east and then north, by the steps north and then east, and then, on the new periodic diagram, moving the $m\times n$ grid up (weakly
northeast) in order to have the next available cell of the
intersection area (there is at most $1$ of them in each row) in
the southwest corner. In our original algorithm at every such
iteration of $\psi_0$, we increased the local variable $rank$ by
$1$, while we subtracted from the sink the number of north steps
that we needed to move our grid.

We want to see what happens to the diagram when we apply $\psi_0$.
To see this, it is useful to keep track of the periodic diagram
that we built by iterating the operator $\psi_0$. In other words,
we can construct a single diagram that encodes all the iterations
of $\psi_0$. This diagram will be similar to the periodic diagram
that we used to compute the iterations of $\psi$ (or  $\varphi$),
but NOT identical. Indeed, this diagram will not be periodic. This
is due to the fact that at every iteration of $\psi_0$ we first
subtract $\diracconfiguration{b_1}$, and then we follow its own
periodic diagram.

In particular, notice that $u$ and $\psi_0(u)$ are not toppling
equivalent, as they have different degrees, while this is
obviously the case for $u$ and $\psi(u)$.

We know that without subtracting $\diracconfiguration{b_1}$ for
each iteration of $\psi_0$, if we just looked at the periodic
diagram and moved our $m\times n$ grid $m-1$ steps east and $n$
steps north, we would see the same red path, but we would see the
green path shifted by $1$ step towards east. Notice in particular
that starting with a parking sorted configuration $u$, we would
not see a parking configuration, since for example in the first
row the intersection area would contain exactly two cells.

Now iterating $\psi_0$, every time we have $1$ cell in the
intersection area, we modify the green path so that in the next
periodic repetition, i.e. $n$ rows northern in the periodic
diagram, we would still see exactly one cell, exactly in the same
position: this is due to the subtraction of
$\diracconfiguration{b_1}$. On the other hand, for any row that
does not contain a cell in the intersection area, in the periodic
repetition, i.e. $n$ rows northern in the periodic diagram, we
will see the corresponding north step of the green path shifted by
$1$ towards east. At this point an example is in order.
\begin{example} \label{ex:reformulation}
We already considered the parking sorted configuration
$u=\configuration{0,0,0,3,3,3;*}{0,0,0,3,3}$. Notice that there
are only two rows in which there is a cell of the intersection
area, so we want to compare $u$ with $\psi_0(\psi_0(u))$. We
already draw a picture of how to compute $\psi_0(u)$, see Figure
\ref{fig:algo_3}. If we compute $\psi_0(\psi_0(u))$, we get
$\psi_0(\psi_0(u))=\configuration{0,0,0,3,3,3;*}{0,1,1,3,4}$, see
Figure \ref{fig:algo_4}.

\begin{figure}[h]
\includegraphics[width=70mm,clip=true,trim=15mm 160mm 50mm 10mm]{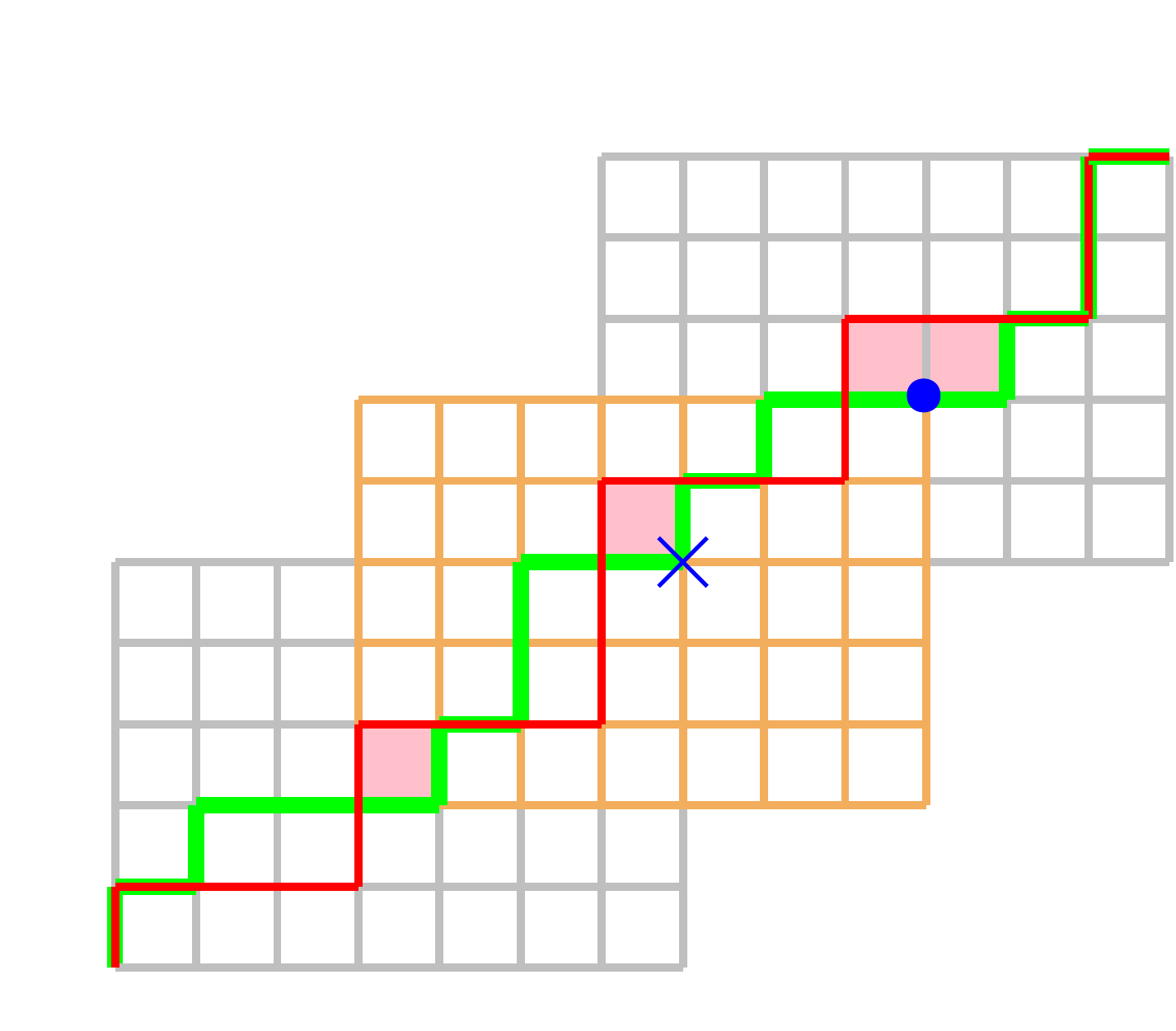}
\caption[ciccia]{This is the diagram of the construction of
$\psi_0(\psi_0(u))=\configuration{0,0,0,3,3,3;*}{0,1,1,3,4}$.}
\label{fig:algo_4}
\end{figure}

As claimed, the same picture could be obtained by producing a
single diagram that keeps track of the iterations of $\psi_0$. In
order to draw this diagram, for every iteration of $\psi_0$ we
move the $m\times n$ grid to the next (weakly northeast) stable
configuration along the usual periodic diagram of the
configuration appearing in the $m\times n$ grid, like we would do
in the computation of $\psi$, but in this new diagram we modify
the green path at the $n$-th row (we start counting the rows from
the bottom row of the $m\times n$ grid, which will be row $0$), by
keeping the cell of the intersection area in that row where it
used to be. This counts for the subtraction of
$\diracconfiguration{b_1}$.

Then we repeat the procedure with $\psi_0(u)$ in place of $u$. The
diagram for $\psi_0(\psi_0(u))$ of our example is shown in Figure
\ref{fig:algo_5}: the pink cells are the intersection area of the
diagram, and the numbers in them are the number of the
corresponding row (the bottom row is the row $0$); the yellow
cells are in the positions in which the iterations of $\psi_0$
modify the green path, and the number in them corresponds to the
number of the bottom row of the $m\times n$ grid at the time of
its creation, i.e. at the time of the iteration of $\psi_0$ that
modified the green path.

\begin{figure}[h]
\includegraphics[width=70mm,clip=true,trim=15mm 160mm 50mm 10mm]{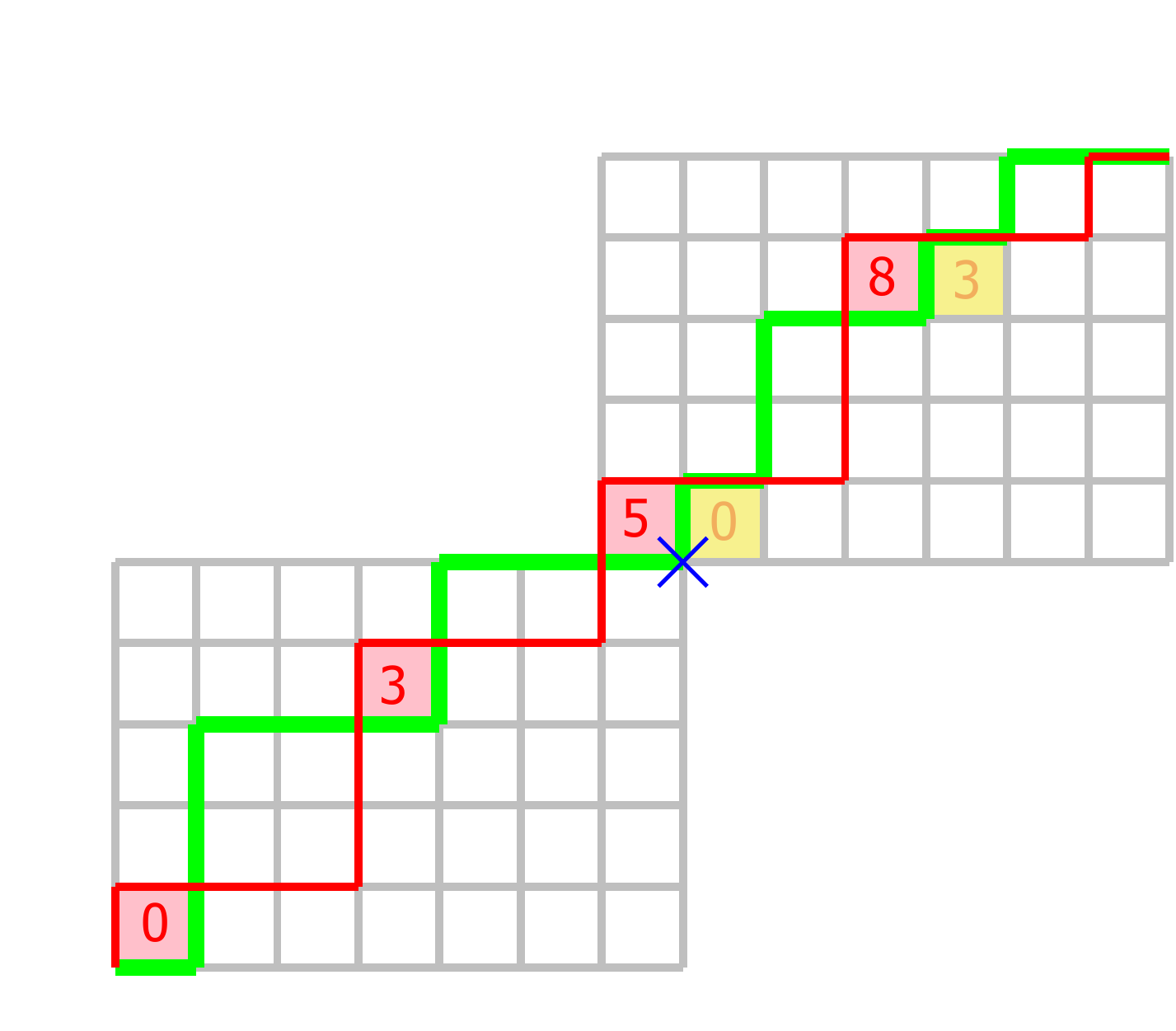}
\caption[ciccia]{This is the diagram of the construction of
$\psi_0(\psi_0(u))=\configuration{0,0,0,3,3,3;*}{0,1,1,3,4}$. The
yellow cells underscore the positions in which the iterations of
$\psi_0$ modify the green path.} \label{fig:algo_5}
\end{figure}
\end{example}

As expected, looking at the diagram $n$ rows northern, the rows
corresponding to the ones in which $u$ had a cell in the
intersection area are unchanged, while in the other rows the green
path got shifted by $1$ step in direction east.

Observe that in general the red path in this diagram will be
periodic of a period consisting of $m+n-1$ steps (as in the case
of the periodic diagram for computing $\psi$), while in general
the green path will not be periodic. In fact it becomes periodic
when, at a given iteration of $\psi_0$, in the (parking sorted)
configuration every row has precisely $1$ cell in the intersection
area: from then on the green path will also be periodic of a
period consisting of $m+n-1$ steps, so that every $n$ iterations
$\psi_0$ we will see always the same configuration outside the
sink.

See Figure \ref{fig:algo_6} for more iterations of $\psi_0$ on the
same example $u=\configuration{0,0,0,3,3,3;*}{0,0,0,3,3}$. The
meaning of the colored cells and the numbers inside them is the
same as the one explained for Figure \ref{fig:algo_5}.

\begin{figure}[h]
\includegraphics[width=130mm,clip=true,trim=25mm 210mm 110mm 25mm]{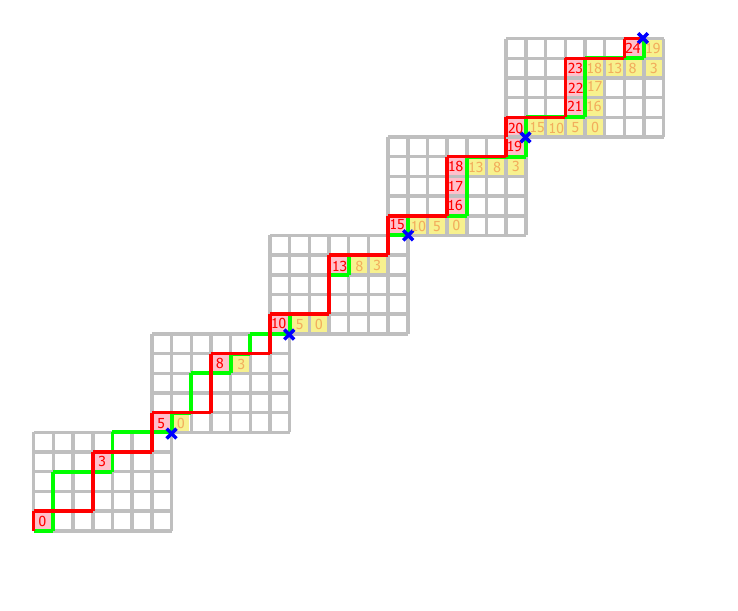}
\caption[ciccia]{This is the diagram of the first few iterations
of $\psi_0$ on $u=\configuration{0,0,0,3,3,3;*}{0,0,0,3,3}$.}
\label{fig:algo_6}
\end{figure}

Notice that indeed after we get a configuration with precisely $1$
cell in the intersection area in every row, the rest of the
diagram becomes periodic (look at the top $10$ rows of Figure
\ref{fig:algo_6}).

Notice also that with our enhanced diagram with the yellow cells,
like the one in Figure \ref{fig:algo_6}, one could also deduce the
proof $f$ given by the algorithm by looking at the last iteration
of $\psi_0$: the proof $f$ for the configuration in the $m\times
n$ diagram of the last iteration of $\psi_0$ will be the
configuration with $f_{a_j}=0$ for all $j$, and $f_{b_i}$ equal to
the number of yellow cells in row $i+1$ (remember that the bottom
row is row $0$) for all $i=1,2,\dots,n$.\newline

Coming back to our new algorithm, let's see why it is equivalent
to our original one. In our new algorithm we move the southwest
corner of the $m\times n$ grid along the green path in the
northeast direction, subtracting $1$ from the sink at every north
step, and adding $1$ to the local variable $rank$ every time we
cross a north step whose cell is in the right component, i.e. in
the intersection area. The observation is that, since we are
moving in our original periodic diagram, every $n$ rows each
vertical green step gets shifted by $1$ step towards east. Now,
once such a green vertical step falls in the right component, i.e.
the correspond cell is in the right component, then it remains in
that component indefinitely. Observe that our new algorithm will
ignore such a vertical green step until it falls in the right
component, and since then it will always increase by $1$ the local
variable $rank$ after crossing it.

This corresponds precisely to the situation of our original
algorithm, where the operator $\psi_0$ skips the rows where the
vertical green steps are in the left component, until they fall in
the right one, so that their cell is in the intersection area, and
since then the algorithm will always stop by that cell, increasing
the local variable $rank$ by $1$ at each visit.

Since obviously both algorithms decrease the value at the sink by
$1$ at every vertical move of the $m\times n$ grid, this proves
the equivalence of the two algorithms.

\section{Cylindric diagrams} \label{sec:cylindric}

Given a parking sorted configuration $u$ on $K_{m,n}$, we
associate to it a \emph{cylindric diagram}, which allows to
compute efficiently the rank of $u$ given the diagram of $u$ and
the value $u_{a_m}$ at the sink.

Consider a parking sorted configuration $u$ on $K_{m,n}$, look at
the corresponding diagram, and consider all the cells to the right
of the green path in its $n$ rows. If the value $u_{a_m}$ at the
sink is negative, then of course $\rank(u)=-1$, and there is
nothing else to do. If $u_{a_m}\geq 0$, start writing the numbers
$0,1,2,\dots,u_{a_{m}}$ in the cells to the left of the vertical
steps of the green path, from the bottom row to the top row. After
writing $0,1,\dots,n-1$, write $n+i$ in the cell to the right of
$i$ for all $i=0,1,\dots,n$. And then write $2n+i$ in the cell to
the right of $n+i$ for $i=0,1,\dots,n-1$. And so on until you
write the number $u_{a_m}$. The only thing to recall is to write
the number $j$ in red if the cell containing $j$ is in the right
component (i.e. to the right of the red path), and in green if it
is in the left component (i.e. to the left of the red path). This
will be the \emph{cylindric diagram} of $u$. See Figure
\ref{fig:new_algo_2} for an example.

\begin{figure}[h]
\includegraphics[width=90mm,clip=true,trim=15mm 200mm 30mm 10mm]{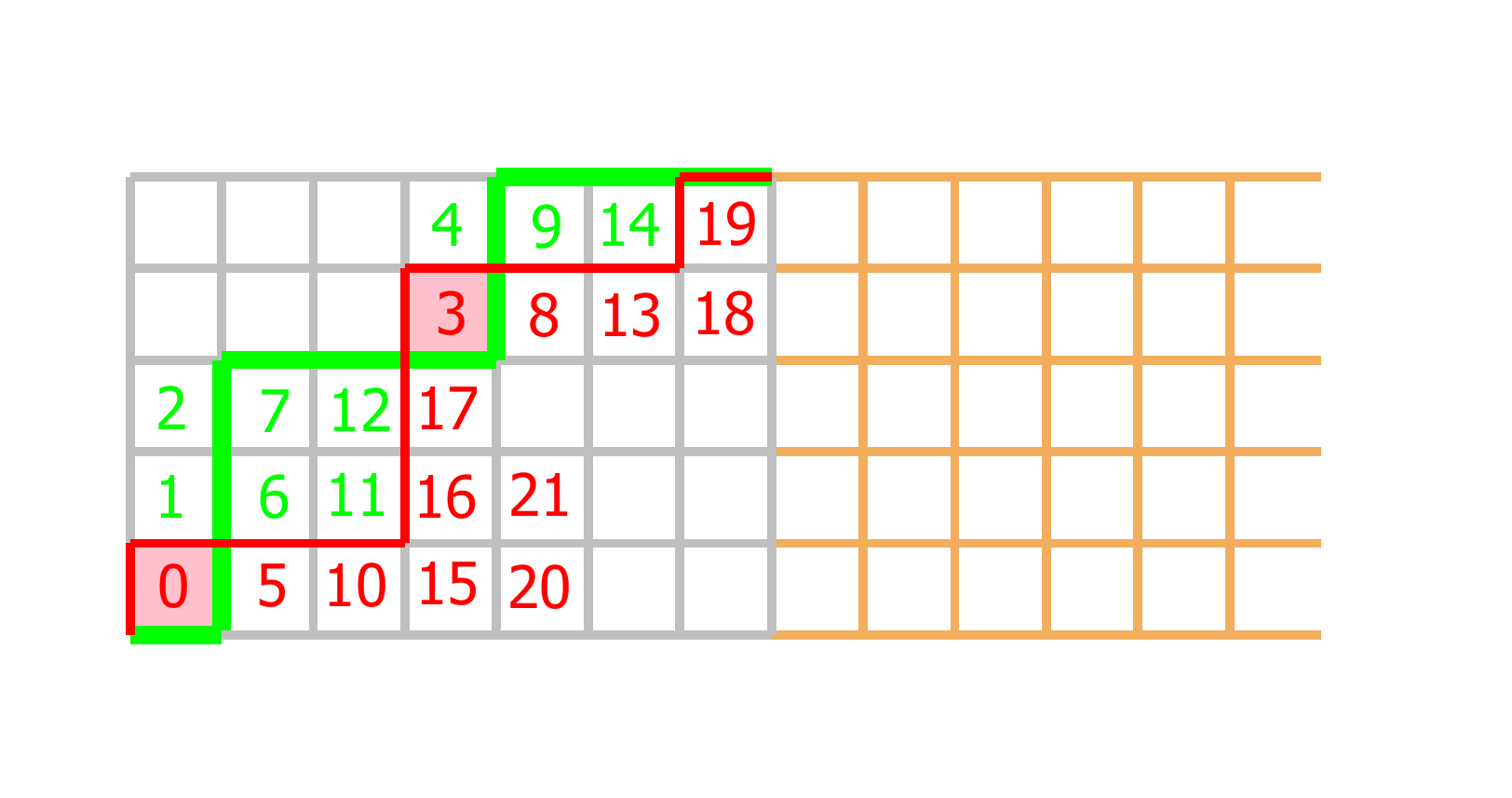}
\caption[ciccia]{The parking sorted configuration
$u=\configuration{0,0,0,3,3,3;21}{0,0,0,3,3}$ with the
corresponding cylindric diagram.} \label{fig:new_algo_2}
\end{figure}

We claim that $\rank(u)$ is equal to $-1$ plus the number of red
numbers in the cylindric diagram of $u$.

To see this, imagine that running our new algorithm on the
periodic diagram of $u$, any time the southwest corner of the
$m\times n$ grid crosses a vertical green step, we write to its
left a label with the number $i$ of the vertical steps that we
already crossed. So for example next to the first vertical green
step we write $0$. Because of the periodicity of the diagram,
since every $n$ rows the vertical green steps get shifted by $1$
step in direction east, we can imagine to ``roll up'' (that's why
the name ``cylindric'' diagram) our periodic diagram modulo $n$ in
the north-south direction and modulo $m-1$ in the west-east
direction. In this way we get always the same red path, but the
green path would be reproduced shifted by $i$ steps in direction
east at the $i+1$-th cycle of $n$ rows: instead of drawing it all
the times, we just draw the first original green path in our
picture, but we keep track of the numerical labels by recording
them in our cylindric diagram where they should appear. So this is
going to give us precisely the cylindric diagram of $u$, except
for the colors of the numbers.

In this way, it is clear from our pictorial description of the
algorithm that we increase the local variable $rank$ by $1$
precisely when the cell of the vertical green step that we just
crossed is in the right component, i.e. when in our cylindric
diagram we recorded a red number, while ignoring all the other
vertical green steps, i.e. when we recorded a green number.

Since the labels count $-1$ plus the number of vertical steps that
we performed in our new algorithm, and since we stop writing our
labels at the number $u_{a_m}$, this means that when we wrote the
last red number it was the last time that the value in our sink
was non-negative, and at the next red number, i.e. at the next
iteration of our original algorithm, it would become negative.
This and the correctness of our new algorithm proves our claim.

\section{The $r$-vectors and a formula for the rank}

In this section we introduce a useful tool that will enable us to
optimize our algorithm. We need some notation.

Consider a stable sorted configuration $u$ on $K_{m,n}$. For all
$i=1,2,\dots,n$, let
$$
r_i:=u_{b_i}+1-|\{u_{a_j}\mid j \neq m, u_{a_j}+1\leq i-1\}|.
$$
Notice that $r_i$ is simply the difference between the distance
from the vertical step of the red upper path in row $i$ (the bottom row is row $1$) from the left edge of
the $m\times n$ grid, and the distance from the vertical step of
the green lower path in row $i$ from the left edge of the $m\times n$ grid.

We remark that for a stable sorted configuration, we always have
$r_1\geq 1$, and $-m+2\leq r_i\leq m$ for all $i$. Moreover, for a
parking sorted configuration, Proposition \ref{prop:park_sort}
says that we always have $r_1=1$, and $-m+2\leq r_i\leq 1$ for all
$i$. This is the translation in terms of the $r$-vector of the
fact that in each row of the diagram of a parking sorted
configuration there is at most one box in the intersection area.

Let us call $r(u):=(r_1,r_2,\dots,r_n)$ the \emph{$r$-vector} of
the stable sorted configuration $u$ on $K_{m,n}$.
\begin{example}
For example, for the parking sorted configuration
$u=\configuration{0,0,0,3,3,3;*}{0,0,0,3,3}$ on $K_{7,5}$ we have
$r(u)=(1,-2-2,1,-2)$, cf. Figure \ref{fig:algo_1_2}.
\end{example}

Recall that in Section \ref{sec:reformulation} we introduced the
operator $\psi_0$, which is defined on any parking sorted
configuration $u$ by
$\psi_0(u)=\sort(\park(u-\diracconfiguration{b_1}))$. In that same
section we saw how the action of $\psi_0$ can be seen as a
suitable composition of the operators $\Ta$ and $\Tb$. Moreover,
we draw a diagram to explain the iterations of the operator
$\psi_0$, cf. Figure \ref{fig:algo_6}.

We want to understand the relation between $r(u)$ and
$r(\psi_0(u))$. In order to see this, we define an operator
$\widetilde{\psi_0}$ on the set of vectors $r=(r_1,r_2,\dots,r_n)$
with $r_i\in \mathbb{Z}$ and $-m+2\leq r_i\leq 1$ for all $i$:
given such a vector $r=(r_1,r_2,\dots,r_n)$, we set
$$
\widetilde{\psi_0}(r):=\left\{\begin{array}{ll}
(r_2,r_3,\dots,r_{n-1},r_n,1)
& \text{ if } r_1=1\\
 (r_2,r_3,\dots,r_{n-1},r_n,r_1+1)
& \text{ if } r_1\lneqq 1\\
\end{array}\right..
$$
Now we can translate all the observations that we made in the
discussion after Example \ref{ex:reformulation} in terms of
$r$-vectors. For example, it is now clear that if $r(u)$ is the
$r$-vector of a parking sorted configuration $u$ on $K_{m,n}$,
then
$$
r(\psi_0(u))=\widetilde{\psi_0}^{h-1}(r(u)),
$$
where $h$ is the smallest integer $h\geq 2$ such that $r_h=1$, if
such an integer exists, otherwise $h:=n+1$. Moreover, what
$\widetilde{\psi_0}^{n}$ does to $r(u)$ is to increase by $1$ all
the $r_i\lneqq 1$ and leave the other $r_i$ fixed. The remark
about the eventual periodicity of the green path of the diagram of
$\psi_0$ can be translated into the following observation.

We can sort the set $\{r(u)\mid u\text{ is a stable sorted
configuration on }K_{m,n}\}$ of all possible $r$-vectors with the
\emph{reverse degree-lexicographic order}: given two distinct
vectors $r(u):=(r_1,r_2,\dots,r_n)$ and
$r(u'):=(r_1',r_2',\dots,r_n')$, we set $r(u)< r(u')$ if
$|r(u)|:=\sum_ir_i< \sum_ir_i'=:|r(u')|$, or if $|r(u)|=|r(u')|$
and for the biggest $i\in \{1,2,\dots,n\}$ such that the
difference $r_i'- r_i$ is non-zero we have $r_i'- r_i> 0$.

Then for a parking sorted configuration $u$,
$r(\widetilde{\psi}_0(u))\geq r(u)$, so in particular
$r(\psi_0(u))\geq r(u)$, and we have the equality $r(\psi_0(u))=
r(u)$ if and only if $r(u)=(1,1,\dots ,1)$.

We are in a position to give a formula for the rank of a parking
sorted configuration $u$ on $K_{m,n}$ in terms of $u_{a_m}$ and
the $r$-vector $r(u)$ (compare the following theorem with \cite[Theorem 12]{corileborgne}).
\begin{theorem} \label{thm:rank_formula}
Let $u$ be a parking sorted configuration on $K_{m,n}$, let $r(u)=(r_1,r_2,\dots,r_n)$ be its $r$-vector, and let
$u_{a_m}\geq 0$ be the value of $u$ on the sink $a_m$. Let
$$
u_{a_m}+1=nQ+R,\quad \text{ with }Q,R\in \mathbb{N},\,\, \text{
and }0\leq R\lneqq n.
$$
Then
\begin{equation}\label{eq:rank_formula}
\rank(u)+1=\sum_{i=1}^n\max\{0,Q+\chi(i\leq R)+r_i-1\},
\end{equation}
where $\chi(\mathcal{P})$ is $1$ if the proposition $\mathcal{P}$
is true, and $0$ otherwise.
\end{theorem}
\begin{proof}
The proof of this theorem follows immediately from our analysis of
the cyclic diagram of $u$: we claim that for every
$i=1,2,\dots,n$, the $i$-th summand of the right hand side of
\eqref{eq:rank_formula} is precisely the number of red labels in
the $i$-th row of our cylindric diagram (the bottom row being row
$1$). Indeed, the $\max$ takes care of the fact that all the
labels of the row could be to the left of the red path, i.e. that
they are all green; the term $Q+\chi(i\leq R)$ takes care of the
fact that all the $u_{a_m}+1$ labels (both green and red) wrap
around the rows $Q$ times with a remainder of $R$; finally the
term $r_i-1$ is there to remove the
green labels from the counting.
\end{proof}
\begin{example}
Consider the parking sorted configuration
$u=\configuration{0,0,0,3,3,3;21}{0,0,0,3,3}$ on $K_{7,5}$. In
this case we have $r(u)=(1,-2-2,1,-2)$, $21+1=4\cdot 5 + 2$, hence
$Q=4$ and $R=2$, so that the summands of the right hand side of
\eqref{eq:rank_formula} are (in order) $5,2,1,4,1$, which are
precisely the number of red labels in row $1,2,3,4,5$ respectively
of the cyclic diagram of $u$, see Figure \ref{fig:new_algo_2}.
\end{example}

\section{An algorithm for the rank on $K_{m,n}$ of linear arithmetic complexity} \label{sec:optimization}

In this section we make explicit all the steps involved in our
original algorithm to compute the rank of a configuration on
$K_{m,n}$. Along the way, we will make an analysis of the
complexity of this algorithm.

In order to do this, in our complexity model we will make the
following two assumptions:
\begin{enumerate}
    \item the four elementary binary operations, i.e. addition, subtraction,
    product, and Euclidean division, they each cost $1$ (\emph{linear arithmetic complexity model});
    \item we can access the position of an array in constant time
    (this will be needed for sorting).
\end{enumerate}
We will use the usual $O$ notation to estimate the number of
operation of the algorithm: we say that we performed $O(n)$
operations if there exists a constant $C>0$ such that for every
$n\geq 1$ we actually performed $\leq C\cdot n$ operations.

In our algorithm we need to perform the following steps.

\underline{Step 1}: given any configuration
$u=\configuration{u_{a_1},u_{a_2},\dots,u_{a_{m-1}};u_{a_m}}{u_{b_1},u_{b_2},\dots,u_{b_n}}$
on $K_{m,n}$, compute a stable configuration
$u'=\configuration{u_{a_1}',u_{a_2}',\dots,u_{a_{m-1}'};u_{a_m}'}{u_{b_1}',u_{b_2}',\dots,u_{b_n}'}$
toppling-equivalent to $u$.

In order to do this, what we do in practice is to use Euclidean
division to compute first
$$
u_{b_i}:=q_i m + u_{b_i}',\quad \text{ with }q_i,u_{b_i}'\in
\mathbb{Z}\text{ and }0\leq u_{b_i}'\lneqq m\text{ for all
}i=1,2,\dots,n,
$$
and then we compute
$$
u_{a_i}+\sum_{j=1}^nq_i=\widetilde{q}_in+u_{a_i}',\quad \text{
with }\widetilde{q}_i,u_{a_i}'\in \mathbb{Z}\text{ and }0\leq
u_{a_i}'\lneqq n\text{ for all }i=1,2,\dots,m-1;
$$
finally we set
$$
u_{a_m}:=\degree(u)-\sum_{i=1}^nu_{b_i}'-\sum_{j=1}^{m-1}
u_{a_j}'.
$$

\begin{claim}
The configuration $u'$ is toppling equivalent to $u$.
\end{claim}
\begin{proof}
It is straightforward to check that
$$
u'=u-\sum_{i=1}^nq_i\Delta^{(b_i)}-\sum_{j=1}^{m-1}\widetilde{q_j}(\Delta^{(a_j)}-\Delta^{(a_m)}).
$$
\end{proof}
Notice that to compute $u'$ we used $n+(m-1)$ Euclidean divisions
and two sums with $\leq m+n$ terms each. Hence, using our
assumption (1) on the complexity model, we performed a total of
$O(m+n)$ operations.

\underline{Step 2}: given a stable configuration $u$ (with respect
to the sink $a_m$), we sort it, i.e. we compute $\sort(u)$. To do
this, we use the so called \emph{counting sort} (cf. \cite[Section 8.2]{cormen}), which is an algorithm that takes an integer $m$ and a list
$w=(w_1,w_2,\dots,w_n)$ of $n$ integers $w_i$ with $0\leq w_i\leq
m$, and it returns a list $out$ with the $w_i$ in increasing
order. This algorithm uses an array of lists: we give a pseudocode for completeness.

\lstset{language=Python}
\begin{lstlisting}[frame=single,texcl,mathescape]
def counting_sort($m$,$w$): # the input $w=[w_1,w_2,\dots,w_n]$ is a list of integers $0\leq w_i\leq m$
  for $i$ in $[0,1,2,\dots,m]$:
    $\ell_i$ = $[\,\,\, ]$  # each $\ell_i$ is initialized as an empty list
  $\ell$ = $<\ell_0,\ell_1,\ell_2,\dots,\ell_{m}>$ # an array of $m+1$ empty lists
  for $i$ in $[1,2,\dots,n]$:
    $\ell_{w_i}$.append($w_i$)  # we append $w_i$ to $\ell_{w_i}$
  $out$ = $[\,\,\, ]$  #  $out$ is initialized as an empty list
  for $j$ in $[0,1,2,\dots,m]$: # we join the lists $\ell_i$ in the given order
    $out$ = $out$ joined to $\ell_j$
  return $out$
\end{lstlisting}

This algorithm, using our assumption (2) on the complexity model,
performs $O(m+n)$ operations.

Now we can use this algorithm to order first the $u_{a_i}$ where
$0\leq u_{a_i}\leq n-1$ for $i=1,2,\dots,m-1$, and then the
$u_{b_j}$ where $0\leq u_{b_j}\leq m-1$ for $j=1,2,\dots,n$.
Therefore, to compute $\sort(u)$, we perform in total $O(m+n)$
operations.

\underline{Step 3}: given a stable sorted configuration
$u=\configuration{u_{a_1},u_{a_2},\dots,u_{a_{m-1}};u_{a_m}}{u_{b_1},u_{b_2},\dots,u_{b_n}}$
on $K_{m,n}$, we compute $u''=\sort(\park(u))$.

Recall the definition of the $r$-vector $r(u)=(r_1,r_2,\dots,r_n)$
of a stable sorted configuration $u$: for all $i=1,2,\dots,n$, we
set
$$
r_i:=u_{b_i}+1-|\{u_{a_j}\mid j \neq m, u_{a_j}+1\leq i-1\}|.
$$
By Proposition \ref{prop:park_sort}, we know that $u$ is parking
if and only if $r_i\leq 1$ for all $i=1,2,\dots,n$. In particular
we must have $r_1=1$.

Now let $h$ be the minimal integer $i$ such that $1\leq i\leq n$
and $r_i=\max\{r_j\mid j=1,2,\dots,n\}$, and let
$k:=u_{b_h}-r_h+2=1+|\{u_{a_j} \mid j \neq m, u_{a_j}+1\leq
h-1\}|$.
\begin{remark}
For $h=1$ we have $k=1$, while for $h\geq 2$, we have also
$\max\{r_j\mid j=1,2,\dots,n\}\geq 2$ (since $r_1\geq 1$), hence
$k\leq 1+(m-2)=m-1$.
\end{remark}
Now we set
\begin{equation} \label{eq:uprime}
u':=u+(r_h-2)\Delta^{(a_m)}-\sum_{s=k}^{m-1}\Delta^{(a_s)}-\sum_{t=h}^n\Delta^{(b_t)},
\end{equation}
and
$$
u'':=\sort(u').
$$

We claim that $u'=\park(u)$, so that $u''=\sort(u')=\sort(\park(u))$.

In fact, before proving the claim, we can give explicit formulae for the values of $u'$ and $u''$ on the vertices.

It is straightforward to check that, given
$u'=\configuration{u_{a_1}',u_{a_2}',\dots,u_{a_{m-1}'};u_{a_m}'}{u_{b_1}',u_{b_2}',\dots,u_{b_n}'}$,
\begin{align} \label{eq:optim}
u_{b_i}' & =u_{b_i}-u_{b_h}+\chi(i\leq h-1)\cdot m,\quad \text{
for }i=1,2,\dots,n,\\
\notag u_{a_j}'& =u_{a_j}-(h-1)+\chi(j\leq k-1)\cdot n,\quad
\text{ for }j=1,2,\dots,m-1,
\end{align}
and
$$
u_{a_m}':=\degree(u)-\sum_{i=1}^nu_{b_i}'-\sum_{j=1}^{m-1}u_{a_j}'.
$$
Moreover, it can be shown that 
\begin{equation} \label{eq:usecond}
u''=\configuration{u_{a_{k}}',u_{a_{k+1}}',\dots ,
u_{a_{m-1}}',u_{a_1}',u_{a_2}',\dots,u_{a_{k-1}}';u_{a_m}'}{u_{b_h}',
u_{b_{h+1}}',\dots,u_{b_n}',u_{b_1}',u_{b_2}',\dots,u_{b_{h-1}}'}.
\end{equation}

To see where the formula \eqref{eq:uprime} comes from, we should
look at our periodic diagram used to understand the action of
$\varphi$, $\psi$, $\Ta$ and $\Tb$. Given the diagram of a stable
sorted configuration $u$, the diagram of $u+\Delta^{(a_m)}$ has
the same red path, while the green path is shifted by $1$ step
towards west (notice that in this situation we may get some $-1$ on the vertices $b_i$). On the other hand, the diagram of
$\sort\left(u-\sum_{s=k}^{m-1}\Delta^{(a_s)}-\sum_{t=h}^n\Delta^{(b_t)}\right)$
is the diagram that we see in the $m\times n$ grid after moving it
on the periodic diagram of $u$ $m-k$ steps west and $n-h+1$ steps
south: this comes from the graphic interpretation of the operators
$\Ta^{-1}$ and $\Tb^{-1}$ (cf. the comments after Lemma \ref{lem:Ta_Tb}).

Now, by definition of $h$ and $r_h$, in $u+(r_h-2)\Delta^{(a_m)}$
the $h$-th row is the lowest row in which the
intersection area has precisely two squares, while in all the
other rows the intersection area has at most $2$ squares.
Moreover, by definition of $k$, the leftmost box in the
intersection area of the $h$-th row has coordinates $(k,h)$. So we have to bring the northeast corner of our $m\times n$ grid on
the southeast corner of the $(k,h)$ cell. In order to do this, we
need to move our $m\times n$ grid $m-k$ steps west and $n-h+1$
steps south.

In this way, by the periodicity of our diagram, we must have
reached a parking sorted configuration. This explains our formula \eqref{eq:uprime}
for $u'$ and proves the claim that $u'=\park(u)$, so that $u''=\sort(u')=\sort(\park(u))$. 

Let us look at an example.
\begin{example}
Consider the stable sorted configuration
$u=\configuration{0,1,2,3,3,3;*}{2,4,4,6,6}$ on $K_{7,5}$, cf.
Figure \ref{fig:optim_1_2}.

\begin{figure}[h]
\includegraphics[width=60mm,clip=true,trim=15mm 200mm 90mm 10mm]{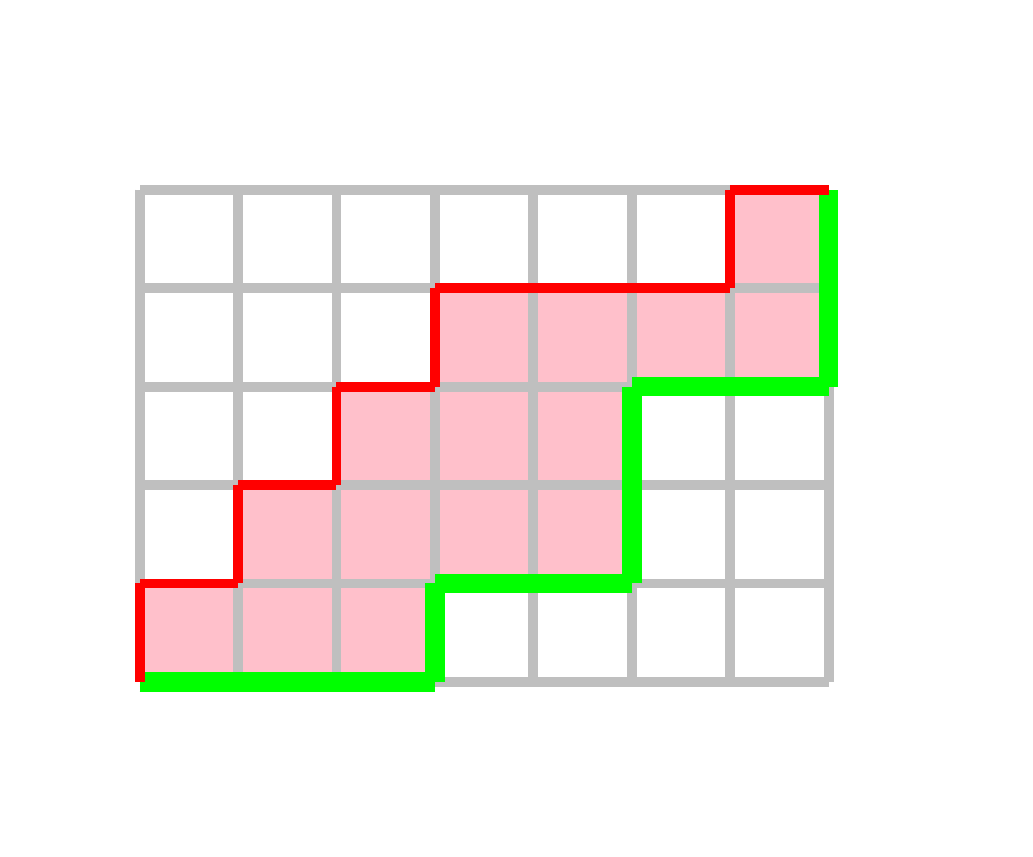}
\includegraphics[width=60mm,clip=true,trim=15mm 200mm 90mm 10mm]{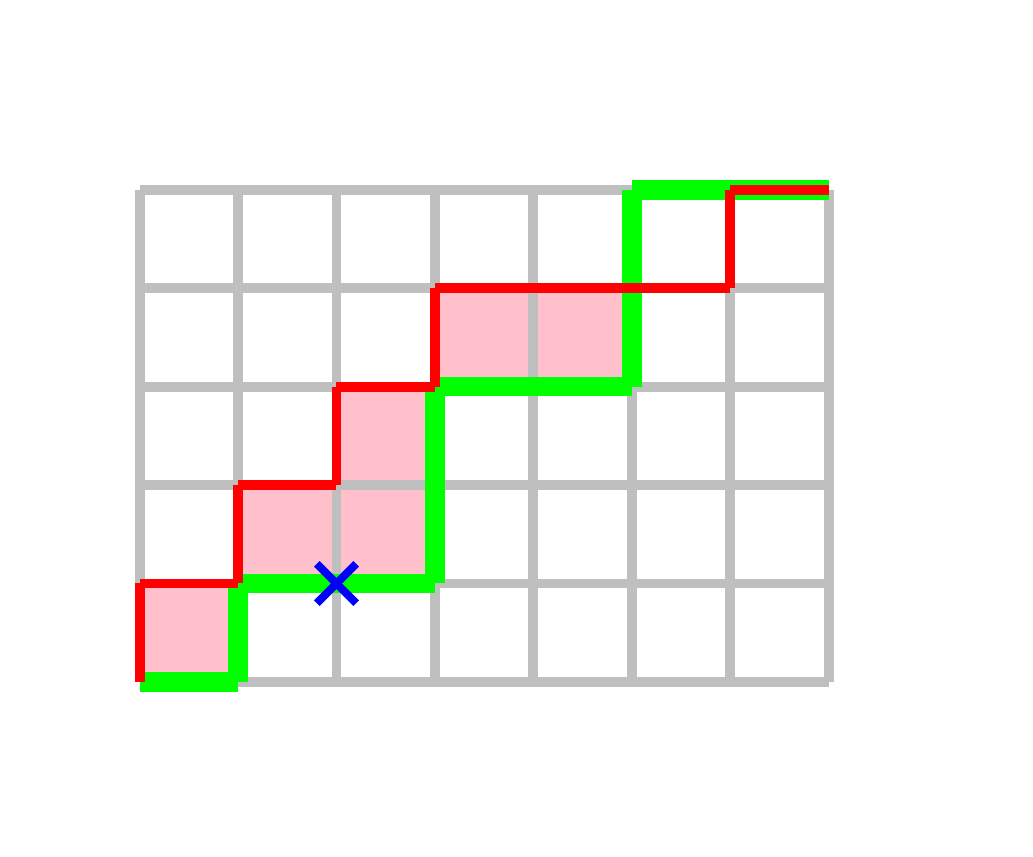}
\caption[ciccia]{On the left, the stable sorted configuration
$u=\configuration{0,1,2,3,3,3;*}{2,4,4,6,6}$. On the right, the
configuration
$u+2\Delta^{(a_7)}=\configuration{0,1,2,3,3,3;*}{0,2,2,4,4}$.}
\label{fig:optim_1_2}
\end{figure}

\begin{figure}[h]
\includegraphics[width=70mm,clip=true,trim=15mm 160mm 50mm 10mm]{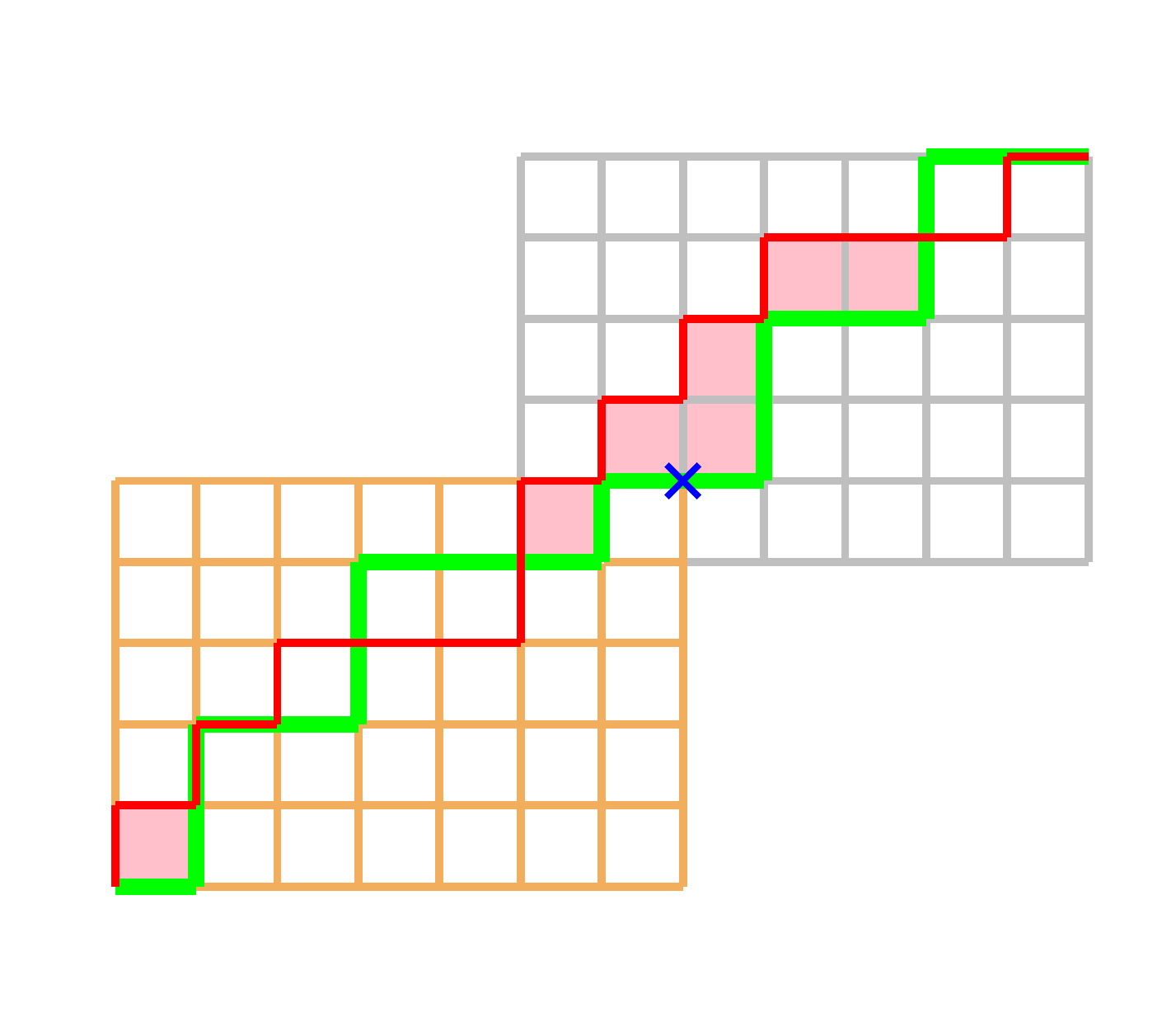}
\caption[ciccia]{This is the diagram of the construction of
$u''=\configuration{0,1,2,2,2,4;*}{0,0,2,2,5}$ from
$u+2\Delta^{(a_7)}$.} \label{fig:optim_3}
\end{figure}

In this case $h=2$, $r_h=4$, and $k=2$. So $r_h-2=2$, and we have
\begin{align*}
u' &
=u+2\Delta^{(a_7)}-\sum_{s=2}^{6}\Delta^{(a_s)}-\sum_{t=2}^5\Delta^{(b_t)}\\
 & = \configuration{4,0,1,2,2,2;*}{5,0,0,2,2},
\end{align*}
and
$$
u''= \configuration{0,1,2,2,2,4;*}{0,0,2,2,5}.
$$
See Figure \ref{fig:optim_3}.

\end{example}

Now that we have a formula for computing $\sort(\park(u))$, we
want to count how many operations it requires.

It is clear from the definition of the $r$-vector that to compute
the $r$-vector of a stable sorted configuration $u$ on $K_{m,n}$
we can use the following algorithm written in pseudocode.

\lstset{language=Python}
\begin{lstlisting}[frame=single,texcl,mathescape]
def compute_r_vector($u$): # the input $u=\configuration{u_{a_1},u_{a_2},\dots,u_{a_{m-1}};u_{a_m}}{u_{b_1},u_{b_2},\dots,u_{b_n}}$ is stable and sorted
  $out$ = $[\,\,\, ]$  # $out$ is initialized as an empty list
  $h$ = $1$
  for $i$ in $[1,2,\dots,n]$:
    while $u_{a_h}+1\leq i-1$ and $h\leq m-1$:
      $h$ = $h+1$
    $out$ = $out$.append($u_{b_i}+1-(h-1)$)
  return $out$
\end{lstlisting}

Clearly this algorithm costs $m+n$ operations.

Moreover, we need $n$ operations to find $h$, and $m$ operations
to find $k$. Finally, we need $2(n+m)$ operations to compute the
$u_{b_i}'$ and the $u_{a_j}'$. So for computing $u''=\sort(u')$, thanks to the Step 3, or simply using formula \eqref{eq:usecond}, in total we performed again
$O(m+n)$ operations.

\underline{Step 4}: given a parking sorted
configuration $u$, we compute $\rank(u)$. We can do this using the
formula of Theorem \ref{thm:rank_formula}: we first compute
$$
u_{a_m}+1=nQ+R,\quad \text{ with }Q,R\in \mathbb{N},\,\, \text{ and
}0\leq R\lneqq n,
$$
and then
$$
\rank(u)=-1+\sum_{i=1}^n\max\{0,Q+\chi(i\leq R)+r_i-1\}.
$$
All this clearly requires $O(m+n)$ operations as well.

All the discussion of this section shows that our algorithm to
compute the rank of a configuration $u$ on $K_{m,n}$ performs
$O(m+n)$ operations under our assumptions on the complexity model.


\section{Two new statistics on parking sorted configurations on $K_{m,n}$}

We consider the generating function of parking sorted configurations on $K_{m,n}$ according to degree and rank:

$$ \widetilde{K}_{m,n}(d,r) := \sum_{u\text{ parking sorted on }K_{m,n}} d^{\degree(u)}r^{\rank(u)}.$$

This formal series is not a power series since the degree takes both arbitrary positive and negative values.
It encodes the distribution of the bistatistic $(\degree,\rank)$. In Figure~\ref{fig:degree_rank_on_K_5_3} we displayed in a table the coefficients of $d^{\degree(u)}r^{\rank(u)}$ for the configurations $u$ on $K_{5,3}$ with degree between $-3$ and $17$.
To get the full table we should add an already started half diagonal of entries equal to $105$ in the top right corner and an already started horizontal half line of entries equal to $105$ in the bottom left corner.

\begin{figure}[ht!]
\begin{tikzpicture}[scale=0.5]
\draw[black!30!white] (-3,-1) grid (18,10);
\foreach \x/\y/\z in {10/4/3,3/1/1,9/1/57,-3/-1/105,8/0/35,11/3/89,12/5/8,7/-1/15,6/2/1,5/1/9,1/-1/102,4/-1/75,7/2/6,4/0/27,11/5/1,8/1/49,10/3/27,0/0/1,3/0/15,-2/-1/105,4/1/3,5/-1/57,11/4/15,5/0/39,7/1/36,-1/-1/105,2/-1/97,6/0/49,14/6/104,13/6/3,8/3/1,15/7/105,1/0/3,8/2/20,9/3/9,7/0/48,17/9/105,9/2/39,6/1/20,16/8/105,14/7/1,13/5/102,3/-1/89,6/-1/35,2/0/8,12/4/97,0/-1/104,10/2/75}{
\draw node at (\x+.5,\y+0.5) {\tiny $\z$};
}
\foreach \degree in {-3,-2,...,17}{
\draw node at (\degree+0.5,-1.5) {\tiny $d^{\degree}$};
}
\foreach \rank in {-1,0,...,9}{
\draw node at (-3.5,\rank+0.5) {\tiny $r^{\rank}$};
}
\end{tikzpicture}
\caption{Partial table of coefficients of $\widetilde{K}_{5,3}(d,r)$.\label{fig:degree_rank_on_K_5_3}}
\end{figure}

There seems to be a symmetry on the distribution of the
coefficients of $\widetilde{K}_{5,3}(d,r)$. For example we can
read the values $15,35,57,75,89,97,102,104,105, 105, 105,\dots$ in
the first row towards west but also in the rightmost diagonal
towards northeast. We observe similar identities on the next rows and
diagonals. These observations and some additional inspections
suggest the following change of parameters: we set

$$\left\{\begin{array}{lcl}
\xpara(u) & = & (m-1)(n-1)+\rank(u)-\degree(u)\\
\ypara(u) & = & \rank(u)+1\\
\end{array}\right.$$
This leads to consider the formal power series
$$ K_{m,n}(x,y) := \sum_{u} x^{\xpara(u)}y^{\ypara(u)} $$
where $u$ still runs on parking sorted configurations of $K_{m,n}$. Observe that
\begin{equation} \label{eq:x_y_from_d_r}
K_{m,n}(x,y)=x^{(m-1)(n-1)}y\widetilde{K}_{m,n}(x^{-1},xy)\,\, \text{ and }\,\, \widetilde{K}_{m,n}(d,r)=d^{(m-1)(n-1)}r^{-1}K_{m,n}(d^{-1},dr),
\end{equation}
hence the study of $\widetilde{K}_{m,n}(d,r)$ is equivalent to the study of $K_{m,n}(x,y)$.

We displayed a partial table of the coefficients of $K_{5,3}(x,y)$ in Figure~\ref{fig:xpara_ypara_on_K_5_3}. Notice that now this formal power series seems to be symmetric in $x$ and $y$.

\begin{figure}[ht!]
\begin{tikzpicture}[scale=0.5]
\draw[black!30!white] (0,0) grid (11,11);
\foreach \x/\y/\z in {1/3/39,3/0/75,8/0/105,2/1/49,0/0/15,1/6/8,0/10/105,5/1/15,2/5/3,0/3/75,4/0/89,1/2/49,9/0/105,3/3/6,0/6/102,8/1/1,1/5/15,5/0/97,0/4/89,10/0/105,4/1/27,1/1/48,3/2/20,2/6/1,7/1/3,2/2/36,6/0/102,1/4/27,2/3/20,0/7/104,4/2/9,1/0/35,0/8/105,0/1/35,7/0/104,3/4/1,6/1/8,3/1/39,2/4/9,2/0/57,1/8/1,6/2/1,4/3/1,1/7/3,0/9/105,0/5/97,5/2/3,0/2/57}{
\draw node at (\x+0.5,\y+0.5) {\tiny $\z$};
}
\foreach \x in {0,1,...,10}{
\draw node at (\x+0.5,-0.5) {\tiny $x^{\x}$};
}
\foreach \y in {0,1,...,10}{
\draw node at (-0.5,\y+0.5) {\tiny $y^{\y}$};
}
\draw[line width=2,opacity=0.2] (0,0) -- (11,11);
\end{tikzpicture}
\caption{Partial table of coefficients of $K_{5,3}(x,y)$.\label{fig:xpara_ypara_on_K_5_3}}
\end{figure}

In this section we study the formal power series $K_{m,n}(x,y)$. In particular we want to understand the apparent symmetry in $x$ and $y$, and we will eventually provide a formula for the generating function of the $K_{m,n}(x,y)$ for $m,n\geq 1$.

The key observation is that the parameters $\xpara$ and $\ypara$ have a natural interpretation in the following extension of our preceding algorithm for computing the rank. We need some more notation.

Given a configuration $u$ on $K_{m,n}$, we denote by $u[*]$ a \emph{partial configuration} undefined on the sink $a_m$. Similarly, we denote by $u[s]$ the configuration on $K_{m,n}$ whose values outside the sink $a_m$ are the ones of $u$, while it takes the value $s$ on the sink $a_m$.
We naturally define the $r$-vector $r(u):=r(u[0]) = (r_1,r_2,\ldots, r_n)$ of the partial configuration $u[*]$: indeed notice that $r(u[s])=r(u[0])$ for all $s\in \mathbb{Z}$ as clearly the $r$-vector does not depend on the value at the sink.

Consider the bi-infinite strip of unit square cells in the plane for which the bottom left corner has coordinates $\{(i,j)| i\in \mathbb{Z},j\in \{0,1,\dots,n-1\}\}$. This is where we draw the cylindric diagram of $u[s]$ for $s\geq 0$ (cf. Section \ref{sec:cylindric}).
%
%
The labeling of the cells of the cylindric diagram that we used to compute the rank can be naturally extended to the negative values $s$ on the sink $a_m$: the cell $L(s)$ labelled by $s$ will have the bottom left corner of coordinates $(u_{b_{t+1}}+q,t)$ where $s=qn+t$ with $q,t\in \mathbb{Z}$ and $0\leq t<n$ is given by euclidean division.
We define the set $\visited{u[s]}$ of \emph{visited cells} by the algorithm for the parking sorted configuration $u[s]$ as all the cells labelled by a number less than or equal to $s$.
We recall that the red path of the diagram of $u$ disconnects the strip of the cylindric diagram into the left component and the right component.

We define $\xpara(u)$, respectively $\ypara(u)$, as the number of unvisited left cells, respectively of visited right cells.

The following lemma proves the coherence of the two given definitions of $\xpara$ and $\ypara$.

\begin{lemma}\label{lem:xy_to_degrank}
For any parking sorted configuration $u$ on $K_{m,n}$ we have

$$\xpara(u) = (m-1)(n-1) +\rank(u)-\degree(u)$$
and
$$\ypara(u) = \rank(u)+1$$
\end{lemma}

The proof of this lemma relies on the following one, which describes the effect of the increment by one of the value on the sink by either a decrement of $\xpara$ or by an increment of $\ypara$.

\begin{lemma}\label{lem:inc_sink}
For any parking sorted configuration $u[s]$ on $K_{m,n}$ we have
$$ \xpara(u[s+1]) = \xpara(u[s])-\chi(\mbox{$L(s+1)$ is at left})$$
and
$$\ypara(u[s+1]) = \ypara(u[s])+\chi(\mbox{$L(s+1)$ is at right}).$$
\end{lemma}

\begin{proof}[Proof of Lemma~\ref{lem:inc_sink}.]
The difference between the visited cells for $u[s]$ and $u[s+1]$ is the extra cell $L(s+1)$ labelled by $s+1$, which is added to $\visited{u[s]}$ to obtain $\visited{u[s+1]}$:
$$ \visited{u[s+1]} = \visited{u[s]}\cup\{L(s+1)\}.$$
Then $L(s+1)$ is either in the left component or in the right component of the cylindric diagram, which are determined by the cut of the red path of $u$.

In the first case, an unvisited left cell in $u[s]$ is now visited in $u[s+1]$ hence $\xpara(u[s+1]) = \xpara(u[s])-1$.
In the second case, the extra visited cell in $u[s+1]$ add one visited right cell hence $\ypara(u[s+1]) = \ypara(u[s])+1$.
\end{proof}

\begin{proof}[Proof of Lemma~\ref{lem:xy_to_degrank}.]
The relation $\ypara(u)=\rank(u)+1$ is clear from our description of the computation of the rank in terms of the cylindric diagram.

To understand $\xpara(u)$ is a bit more involved.
First we observe that for any partial configuration $u$ the quantity
$$ Q(u) := \degree(u[s])+\xpara(u[s])-\ypara(u[s]) $$
is well defined, i.e. it does not depend on $s$.
Indeed, according to Lemma~\ref{lem:inc_sink}, when $s$ is incremented by one, $\degree(u[s])$ is incremented by one, and exactly one of $\xpara(u[s])$ and $-\ypara(u[s])$ is decremented by one.

Hence to prove the relation $\xpara(u)=(m-1)(n-1)+\rank(u)-\degree(u)$ it is enough to show that $Q(u) = (m-1)(n-1)-1$. In order to do this, we will compute $Q(u)= \degree(u[0])+\xpara(u[0])-\ypara(u[0]) $.

The general argument is better understood by looking at an example: we consider the parking sorted configuration $u[0]=\configuration{1,1,2,2,2,4;0}{0,0,2,2,5}$ and draw its diagram, see Figure~\ref{fig:xpara-proof} on the left.

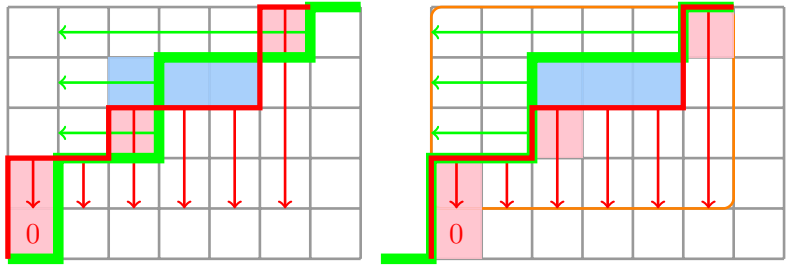
\begin{figure}[ht!]
\begin{tikzpicture}[scale=0.67]
\draw[draw=black!40!white,line width=1] (0,0) grid (7,5);
\foreach \x/\y/\acolor in {0/0/mypink,0/1/mypink,2/2/mypink,5/4/mypink,3/3/myblue,4/3/myblue,2/3/myblue}{
\fill[\acolor,opacity=0.9] (\x,\y) rectangle (\x+1,\y+1);
}
\foreach \row/\bvalue in {2/2,3/2,4/5}{
\draw[green,->,line width=1] (\bvalue+1,\row+0.5) -> (1,\row+0.5);
}
\foreach \column/\avalue in {0/1,1/1,2/2,3/2,4/2,5/4}{
\draw[red,->,line width=1] (\column+0.5,\avalue+1) -> (\column+0.5,1);
}
\draw[draw=green,line width=4] (0,0) -- (1,0) -- (1,2) -- (3,2) -- (3,4) -- (6,4) -- (6,5) -- (7,5);
\draw[draw=red,line width=2] (0,0) -- (0,2) -- (1,2) -- (2,2) -- (2,3) -- (5,3) -- (5,5) -- (6,5);
\foreach \x/\y/\label in {0/0/0}{
\draw node at (\x+0.5,\y+0.5) {$\textcolor{red}{\label}$};
}
\end{tikzpicture}
\begin{tikzpicture}[scale=0.67]
\draw[draw=black!40!white,line width=1] (0,0) grid (7,5);
\draw[draw=orange,rounded corners,line width=1] (0,1) rectangle (6,5);
\foreach \x/\y/\acolor in {0/0/mypink,0/1/mypink,2/2/mypink,5/4/mypink,3/3/myblue,4/3/myblue,2/3/myblue}{
\fill[\acolor,opacity=0.9] (\x,\y) rectangle (\x+1,\y+1);
}
\foreach \row/\bvalue in {2/2,3/2,4/5}{
\draw[green,->,line width=1] (\bvalue,\row+0.5) -> (0,\row+0.5);
}
\foreach \column/\avalue in {0/1,1/1,2/2,3/2,4/2,5/4}{
\draw[red,->,line width=1] (\column+0.5,\avalue+1) -> (\column+0.5,1);
}
\draw[draw=green,line width=4] (-1,0) -- (0,0) -- (0,2) -- (2,2) -- (2,4) -- (5,4) -- (5,5) -- (6,5);
\draw[draw=red,line width=2] (0,0) -- (0,2) -- (1,2) -- (2,2) -- (2,3) -- (5,3) -- (5,5) -- (6,5);
\foreach \x/\y/\label in {0/0/0}{
\draw node at (\x+0.5,\y+0.5) {$\textcolor{red}{\label}$};}
\end{tikzpicture}
\caption{The diagram of the configuration $u[0]=\configuration{1,1,2,2,2,4;0}{0,0,2,2,5}$ and its shifted version for the evaluation of $Q(u)$.  \label{fig:xpara-proof}}
\end{figure}

We recall that the length of the green arrow in Figure~\ref{fig:xpara-proof} in the row $i$ is $u_{b_i}$ and the length of the red arrow in column $j$ is $u_{a_{j}}$.
We assume that $u$ is parking hence the (pink) intersection area contains at most one (pink) cell on each row.
This implies that if one shifts the green path by one unit step to the west, the green and red arrows do not intersect, cf. Figure~\ref{fig:xpara-proof} on the right.
Then we partition the $(m-1)(n-1)$ cells of the orange rectangle defined by the two opposite corners $(0,1)$ and $(m-1,n)$.
There a three kind of cells:
\begin{itemize}
\item the cells crossed by an horizontal green arrow counted by $\sum_{i=2}^{n} u_{b_i}$, (we remind that $u_{b_1}=0$),
\item the cells crossed by a vertical red arrow counted by $\sum_{j=1}^{m-1} u_{a_j}$, and
\item the blue cells which are the unvisited left cells in $u[0]$ hence counted by $\xpara(u[0])$.
\end{itemize}
This partition of cells implies that
$$ (m-1)(n-1) = \degree(u[0])+\xpara(u[0]).$$
Since $\rank(u[0]) = 0$ we finally have
$$ (m-1)(n-1) = \degree(u[0])+\xpara(u[0])-\rank(u[0]) = Q(u)+1.$$
\end{proof}

\section{Symmetry of $K_{m,n}(x,y)$}

We are now in position of proving the symmetry in $x$ and $y$ of $K_{m,n}(x,y)$. It turns out that its explanation relies on the Riemann-Roch theorem for graphs. Compare the following theorem with \cite[Proposition 15 and Corollary 16]{corileborgne}

\begin{theorem} \label{thm:x_y_symmetry}
For any bipartite complete graph $K_{m,n}$ we have  $K_{m,n}(x,y) = K_{m,n}(y,x)$.
\end{theorem}

\begin{proof}
The \emph{canonical divisor} $K_G$ on the bipartite complete graph $G=K_{m,n}$ is the configuration $d$ such that $d_{a_i}=n-2$ and $d_{b_j}=m-2$ for all $i$ and $j$.
The Riemann-Roch theorem for graphs of Baker and Norine (cf. \cite[Theorem 1.12]{baker}) in the particular case of $K_{m,n}$ states that
$$\rank(u)-\rank(d-u)=\degree(u)- |E_{K_{m,n}}| + |V_{K_{m,n}}|$$
where $|E_{K_{m,n}}| = mn$, $|V_{K_{m,n}}|=m+n$.

We have
$$\begin{array}{lcl}
 \xpara(d-u) & = &(m-1)(n-1)+\left(\rank(d-u)-\degree(d-u)\right)\\
\ & = & (m-1)(n-1)+\left(\rank(u)-mn+(m+n)\right)\\
\ & = & \rank(u)+1=\ypara(u)\\
\end{array}$$
where in the second equality we applied the Riemann-Roch theorem to the configuration $d-u$ in order to get $\xpara(d-u)=\ypara(u)$.

Similarly for $u=d-(d-u)$ we obtain that $\ypara(d-u)=\xpara(u)$, hence the involution $u\leftrightarrow d-u$ shows the expected symmetry.
\end{proof}

\begin{remark}
As on $K_n$ \cite[Lemma 18]{corileborgne}, a \emph{superimposition principle} on the graphical interpretation for $K_{m,n}$ also gives  a combinatorial proof of this symmetry.
We only sketch the argument here, as the details already given in \cite{corileborgne} should help to fill in the gaps.
It appears that the cylindric diagram of $\sort(\park(d-u))$ may be superimposed with the cylindric diagram of $\sort(\park(u))$ with the algorithm computing the rank labeling the cells in the reverse order and the right and left components switched.
There is also a shift to be fixed in this reverse labeling of cells, so that the first cell incrementing the rank from $-1$ is labeled by $0$.
It turns out that, via this involution, at any point of the algorithm, the notion of visited and unvisited cells are switched, as well as the notions of left and right cells.
Hence this involution shows in terms of cells that $\xpara(\sort(\park(d-u))) = \ypara(\sort(\park(u)))$ and $\ypara(\sort(\park(d-u))) = \xpara(\sort(\park(u)))$ proving combinatorially the symmetry of $K_{m,n}(x,y)$. 
\end{remark}

\section{The generating function $\mathcal{F}(x,y,w,h)$}

We now turn our attention to the generating function
$$
\mathcal{F}(x,y,w,h):=\sum_{n\geq 1,m\geq 1}K_{m,n}(x,y)w^mh^n.
$$
In order to provide a formula for this formal power series, we start by considering the generating function
$$ F_u(x,y) := \sum_{s\in \mathbb{Z}} x^{\xpara(u[s])}y^{\ypara(u[s])}$$
of $(u[s])_{s\in\mathbb{Z}}$ according to the bistatistic $(\xpara,\ypara)$, where $u[*]$ is a parking sorted partial configuration on $K_{m,n}$. Notice that
$$
K_{m,n}(x,y)=\sum_{u}F_{u}(x,y)
$$
where the sum on the right hand side is over all parking sorted partial configurations $u[*]$ on $K_{m,n}$.

We observe that Lemma~\ref{lem:inc_sink} reveals some geometric sums which are hidden in $F_u(x,y)$.
Indeed this lemma is equivalent to saying that if $L(s+1)$ is at the left of the cut (i.e. the red path) in the cylindric diagram of the partial configuration $u[*]$ then
$$ x^{\xpara(u[s+1])}y^{\ypara(u[s+1])} = \frac{1}{x}.x^{\xpara(u[s])}y^{\ypara(u[s])},$$
while if $L(s+1)$ is at its right then
$$ x^{\xpara(u[s+1])}y^{\ypara(u[s+1])} = y.x^{\xpara(u[s])}y^{\ypara(u[s])}.$$

This suggests to partition $\mathbb{Z}$ into the maximal intervals with the property that the cells in the cylindric diagram of $u$ labelled by the elements of a given interval are all on the same side of the cut (i.e. they have all the same color, red or green), and then to partition the series $F_u(x,y)$ into a sum of finitely many series according to those intervals. To make this more explicit we introduce some notation.

A configuration $u[s]$ is a \emph{positive boundary configuration} if $L(s+1)$ is at the right of the cut (of the cylindric diagram of $u$) while $L(s)$ is at its left.
A configuration $u[s]$ is a \emph{negative boundary configuration} if $L(s+1)$ is at the left of the cut (of the cylindric diagram of $u$) while $L(s)$ is at its right.

Let $s_+^{(0)}\leq s_+^{(1)}\leq \ldots \leq s_+^{(k)}$ be the values for which $u[s_+^{(i)}]$ is a positive boundary configuration (there are clearly finitely many), and set $S_u^+=\{s_+^{(0)},s_+^{(1)},\ldots, s_+^{(k)}\}$.
Define similarly $S_u^-=\{s_-^{(0)},s_-^{(1)},\ldots, s_-^{(k-1)}\}$ as the set of values for which $u[s_-^{(i)}]$ is a negative boundary configuration. Observe that the indices of the elements of $S_u^-$ go up to $k-1$: indeed $|S_u^+|=|S_u^-|+1$ since $L(s)$ is at the left of the cut for $-s$ large enough, while it is at its right for $s$ large enough.

Observe that the two sets $S_u^+$ and $S_u^-$ define the partition of $\mathbb{Z}$ that we mentioned above:
$$\mathbb{Z}=]-\infty,s_+^{(0)}] \sqcup \coprod_{i=0}^{k-1} \left([s_+^{(i)}+1,s_-^{(i)}] \sqcup [s_-^{(i)}+1,s_+^{(i+1)}]\right) \sqcup [s_+^{(k)}+1,+\infty[\, .$$

Compare the following lemma with \cite[Lemma 20]{corileborgne}
\begin{lemma} \label{lem:formula_F_u}
For any parking sorted partial configuration $u[*]$ on $K_{m,n}$, we have
$$ F_u(x,y) = \frac{1-xy}{(1-x)(1-y)}\left(\sum_{s_+\in S_u^+ } x^{\xpara(u[s_+])}y^{\ypara(u[s_+])} - \sum_{s_- \in S_u^-} x^{\xpara(u[s_-])}y^{\ypara(u[s_-])}\right).$$
\end{lemma}

\begin{proof}
We already observed that the two sets $S_u^+$ and $S_u^-$ define a partition of $\mathbb{Z}$. By Lemma~\ref{lem:inc_sink}, in each interval $I$ of this partition, the sum $\sum_{s\in I} x^{\xpara(u[s])}y^{\ypara(u[s])}$ is geometric of reason either $x$, reversing the order of summation in this case, or $y$; in both cases it can be summed easily by using the boundary terms defining $I$.

Using the more compact notation $W_u(s):=x^{\xpara(u[s])}y^{\ypara(u[s])}$, it is easy to check that for $i=0,1,\dots,k-1$
$$\begin{array}{lcl}
I = ]-\infty,s_+^{(0)}] & \text{ gives } & \displaystyle \sum_{s\in I} x^{\xpara(u[s])}y^{\ypara(u[s])} = \frac{W_u(s_+^{(0)})}{1-x},\\
I = [s_+^{(i)}+1,s_-^{(i)}-1] & \text{ gives } & \displaystyle \sum_{s\in I} x^{\xpara(u[s])}y^{\ypara(u[s])} = \frac{yW_u(s_+^{(i)})-W_u(s_-^{(i)})}{1-y},\\
I = [s_-^{(i)},s_+^{(i+1)}] & \text{ gives } & \displaystyle \sum_{s\in I} x^{\xpara(u[s])}y^{\ypara(u[s])} = \frac{W_u(s_+^{(i+1)})-xW_u(s_-^{(i)})}{1-x},\text{ and}\\
I = [s_+^{(k)}+1,+\infty[ & \text{ gives } & \displaystyle \sum_{s\in I} x^{\xpara(u[s])}y^{\ypara(u[s])} = \frac{yW_u(s_+^{(k)})}{1-y}.\\
\end{array} $$
Now we sum these series corresponding to the intervals, and we observe that in this sum for $i=0,1,\ldots, k$ (uniformly!) the coefficient of $W_u(s_+^{(i)})$ is
$$ \frac{1}{1-x}+\frac{y}{1-y} = \frac{1-xy}{(1-x)(1-y)},$$
while for $i=0,1,\ldots, k-1$ (still uniformly), the coefficient of $W_u(s_-^{(i)})$
 is
$$ -\frac{1}{1-y}-\frac{x}{1-x} = -\frac{1-xy}{(1-x)(1-y)},$$
finally leading to the desired formula.
\end{proof}

So Lemma \ref{lem:formula_F_u} reduces the computation of $F_u(x,y)$, and hence of $\mathcal{F}(x,y,w,h)$, to the computation of the generating function of the boundary configurations. In order to understand this one, we provide a description of the boundary configurations and their bistatistic $(\xpara,\ypara)$ in terms of pairs of binomial paths.

We define a \emph{pair of binomial paths in the $m\times n$ grid},
or simply a \emph{pair}, as a pair of two binomial paths in the
$m\times n$ grid starting at the southwest corner: a \emph{red
path}, which consists of the shuffling of $n$ north steps and
$m-1$ east steps, and a \emph{green path}, which consists of the
shuffling $n$ north steps and $m$ east steps. When needed, we will
identify these binomial paths with words in the alphabet
$\{\eaststep,\northstep\}$ where $\eaststep$ indicates an east
step and $\northstep$ indicates a north step.

The \emph{east suffix} of binomial path is the longest suffix of the path (i.e. of the corresponding word in $\{\eaststep,\northstep\}$) that contains only east steps.
Such a pair is a \emph{positive boundary pair} if the red path starts with a north step, the green path starts with an east step, and the length of the east suffix of the green path is longer than or equal to the length of the east suffix of the red path.
A pair is a \emph{negative boundary pair} if the red path starts with an east step, the green path starts with a north step, and the length of the east suffix of the green path is shorter than or equal to the length of the east suffix of the red path.

In a pair $p$ of binomial paths in the $m\times n$ grid, we consider on each row of the grid the cells between the red and the green north steps crossing that row.
If the red north step is at the left of the green step, then we say that these cells belongs to the $\xarea$, otherwise we say that they belong to the $\yarea$. Notice that the $\xarea$ corresponds to what we called the \emph{intersection area} in previous sections.
Hence $\xarea(p)$, respectively $\yarea(p)$, will denote the total number of cells in the $\xarea$, respectively in the $\yarea$.
We also denote by $\xrow(p)$ the number of rows where the crossing red north step is weakly to the right of the crossing green north step, and by $\yrow(p)$ the number of the other rows, where the crossing red north step is strictly to the left of the crossing green step.

We also set for any positive pair $p_+$
$$ \xpara_+(p_+) := \xarea(p_+)+\xrow(p_+) \mbox{ and } \ypara_+(p_+) := \yarea(p_+)-\yrow(p_+),$$
and for any negative pair $p_-$,
$$ \xpara_-(p_-) := \xarea(p_-) \mbox{ and } \ypara_-(p_-) := \yarea(p_-).$$

\begin{example}
Consider the pair $p$ of binomial paths in Figure \ref{fig:enum_1}.
\begin{figure}[h]
\includegraphics[width=60mm,clip=true,trim=15mm 215mm 90mm 25mm]{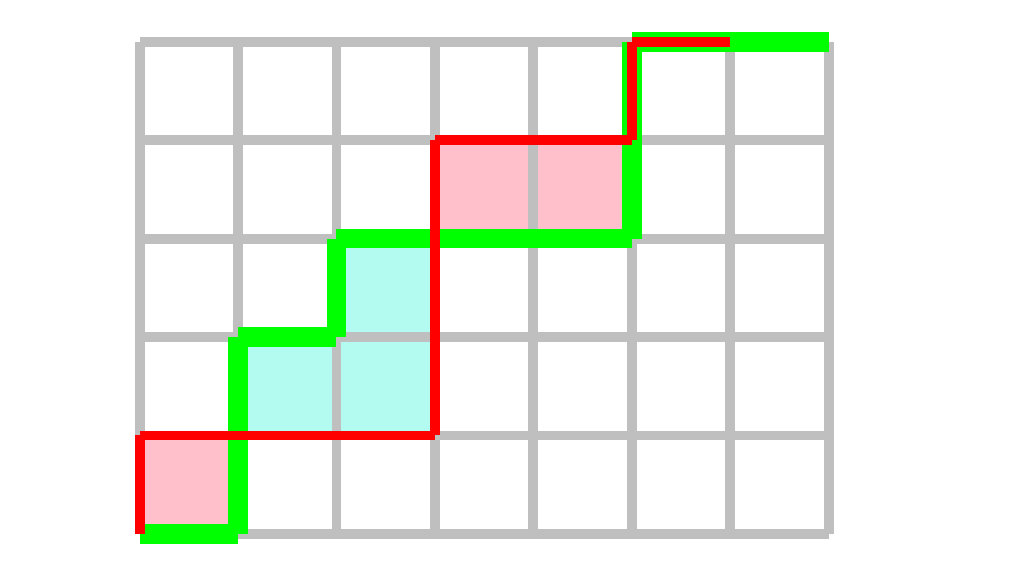}
\caption[ciccia]{A positive boundary pair in the $7\times 5$ grid.}
\label{fig:enum_1}
\end{figure}
Here we have a $7\times 5$ grid, the red path is $\northstep\eaststep \eaststep\eaststep \northstep\northstep\northstep\eaststep\eaststep \northstep\eaststep$, and the green path is $\eaststep \northstep\northstep \eaststep \northstep\eaststep \eaststep\eaststep \northstep\northstep\eaststep\eaststep$. Observe that the red path starts with a north step, and its east suffix is $\eaststep$, while the green path starts with an east step, and its east suffix is $\eaststep\eaststep$, hence $p$ is a positive boundary pair. The $\xarea$ corresponds to the pink cells, hence $\xarea(p)=3$, the $\yarea$ corresponds to the blue cells, hence $\yarea(p)=3$. Observe also that $\xrow(p)=3$ and $\yrow(p)=2$, hence $\xpara_+(p)=\xarea(p)+\xrow(p)=6$, while $\ypara_+(p)=\yarea(p)-\yrow(p)=1$.
\end{example}
\begin{remark} \label{rem:positive_pairs_stable}
Notice that every positive boundary pair corresponds to a certain stable sorted configuration: for example the pair in Figure \ref{fig:enum_1} corresponds to the stable sorted configuration $u=\configuration{0,0,0,3,3,4;*}{0,0,1,4,4}$. Notice that in this case $u$ is stable but not parking, as it has $2$ cells of its intersection area in the same row.
\end{remark}


The following lemma explains our optimistic notations.

\begin{lemma}\label{lem:f+f-}
For any complete bipartite graph $K_{m,n}$ there exists two bijections, one denoted $f_+$, from the positive boundary configurations into the positive boundary pairs (in the $m\times n$ grid), and another one denoted $f_-$, from the negative boundary configurations into the negative boundary pairs, with the following properties:
\begin{enumerate}
\item for any positive boundary configuration $u[s_+]$, if we set $p_+=f_+(u[s_+])$, then
$$ (\xpara(u[s_+]),\ypara(u[s_+])) = (\xpara_+(p_+),\ypara_+(p_+));$$

\item for any negative boundary configuration $u[s_-]$, if we set $p_-=f_-(u[s_-])$, then
$$ (\xpara(u[s_-]),\ypara(u[s_-])) = (\xpara_-(p_-),\ypara_-(p_-)).$$
\end{enumerate}
\end{lemma}

In fact, in order to prove the lemma, we can define explicitly the maps $f_+$ and $f_-$.
Those bijections essentially consist in identifying the appropriate intersections of the red and green periodic paths of $u[*]$ embedded in $\mathbb{Z}^2$, each such intersection being related simultaneously to one boundary pair and one boundary configuration.

In a pair of periodic binomial paths in $\mathbb{Z}^2$, we define a \emph{positive boundary intersection} $i_+$ to be a vertex of $\mathbb{Z}^2$ common to the red and the green periodic paths such that:
\begin{enumerate}
\item $i_+$ is followed (going from southwest towards northeast)
by a north red step and an east green step, and \item the last
green north step before $i_+$ is weakly at the left of the red
north step crossing the same row.
\end{enumerate}

Similarly, a \emph{negative boundary intersection} $i_-$ is a vertex common to the red and green periodic paths such that:
\begin{enumerate}
\item $i_-$ is followed by an east red step and a north green step, and
\item the last green north step before $i_-$ is weakly at the right of the red north step crossing the same row.
\end{enumerate}
Our construction is better understood by looking at an example.
\begin{example}
In Figure~\ref{fig:boundary_intersections} we draw the periodic diagram of the parking sorted configuration $u=\configuration{0,0,0;*}{0,0,1}$ on $K_{4,3}$, whose diagram can be recognized in the lowest blue dashed $4\times 3$ grid. In the picture we included the labeling of the cells to the left of the green north steps: recall that this diagram with this labeling is simply an unrolled version of our cylindric diagram of $u$. Moreover notice that, for graphical practicality, we wrote $\overline{n}$ instead of $-n$ for $n\in \mathbb{N}\setminus \{0\}$.

\begin{figure}[ht!]
\begin{tikzpicture}[scale=0.75]
\draw[draw=black!30!white] (-4,-3) grid (16,12);
\foreach \x in {-1,0,1,2,3}{
\begin{scope}[shift={(4*\x,3*\x)}]
\draw[green,line width=4] (0,0) -- (1,0) -- (1,2) -- (2,2) -- (2,3) -- (4,3);
\end{scope}
}
\foreach \x in {-1,0,1,2,3}{
\begin{scope}[shift={(3*\x,3*\x)}]
\draw[red,line width=2] (0,0) -- (0,1) -- (3,1) -- (3,3);
\end{scope}
}
\foreach \x/\y/\label in {-4/-3/\overline{3},0/0/0,4/3/3,8/6/6,-4/-2/\overline{2},0/1/1,4/4/4,8/7/7,-3/-1/\overline{1},1/2/2,5/5/5,9/8/8,12/9/9,12/10/10,13/11/11}{
  \draw node at (\x+0.5,\y+0.5) {$\label$};
}
\foreach \x/\y/\acolor/\sign/\label in {0/0/blue/+/0,-3/-2/orange/-/0,3/3/blue/+/1,1/1/orange/-/1,6/6/blue/+/2,5/4/orange/-/2,9/8/blue/+/3}{
\draw[\acolor,line width=1] (\x-0.2,\y+0.2) -- (\x+0.2,\y-0.2);
\draw[draw=\acolor,line width=1] (\x+0.2,\y+0.2) -- (\x-0.2,\y-0.2);
\draw node[\acolor] at (\x+0.5,\y-0.5) {$i_\sign^{(\label)}$};
\draw[dashed,line width=1,draw=\acolor,rounded corners] (\x,\y) rectangle (\x+4,\y+3);
}
\foreach \x/\y/\acolor/\sign/\label in {-3/-1/blue/+/0,1/2/blue/+/1,5/5/blue/+/2,8/7/blue/+/3,0/0/orange/-/0,4/3/orange/-/1,8/6/orange/-/2}{
\draw[\acolor] node at (\x-0.3,\y+0.5) {$s_{\sign}^{(\label)}=$};
}
\end{tikzpicture}
\caption{Boundary intersections and related boundary configurations and pairs.}\label{fig:boundary_intersections}
\end{figure}
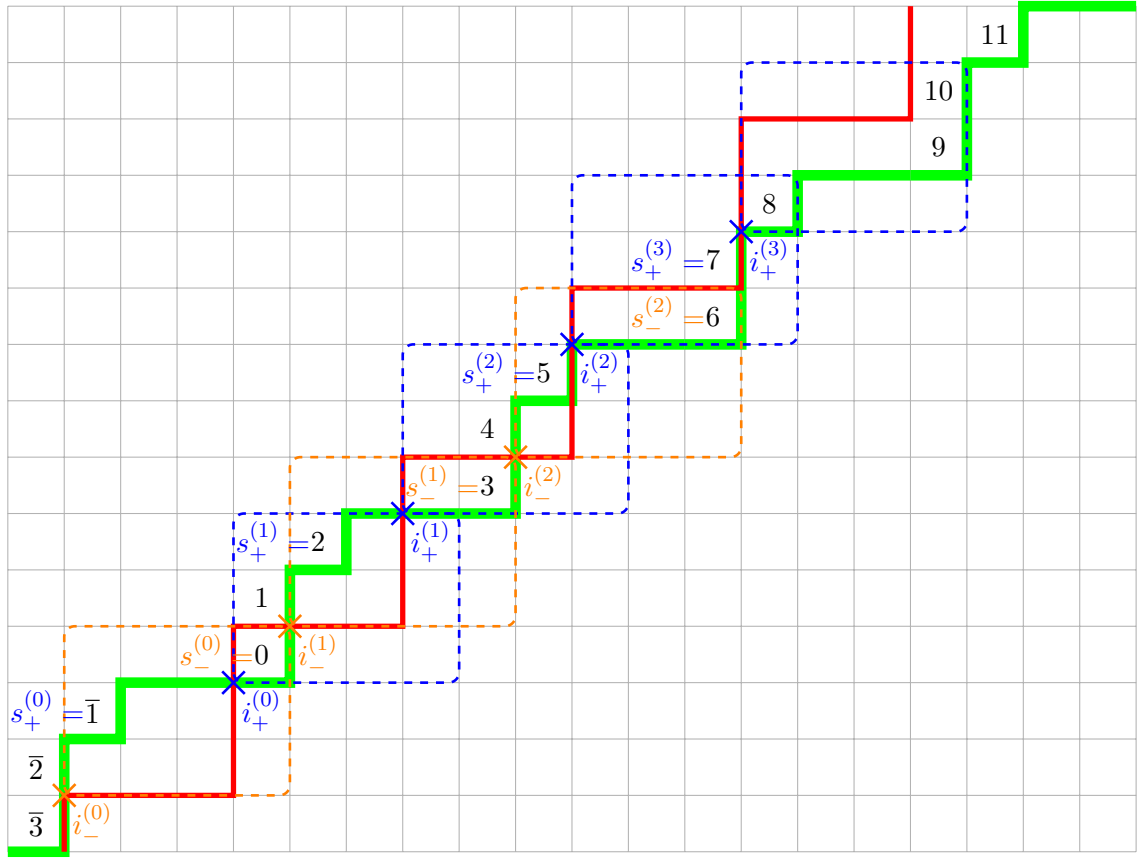

Observe that we labeled the positive boundary intersections by $\{i_+^{(0)},i_+^{(1)},i_+^{(2)}, i_+^{(3)}\}$, and we denoted them by a blue cross, while we labelled the negative boundary intersections by $\{i_-^{(0)},i_-^{(1)},i_-^{(2)}\}$, and we denoted them by an orange cross.

Following the black labels of cells along the green path, we observe that
$$
S_u^{+} = \{s_+^{(0)},s_+^{(1)},s_+^{(2)}, s_+^{(3)}\} = \{-1,2,5,7\}\text{ and } S_u^{-} = \{s_-^{(0)},s_-^{(1)},s_-^{(2)}\}=\{0,3,6\}.
$$

It is clear from the definitions (and it can be checked in our
example) that the value $s_+^{(k)}$, defining the positive
boundary configuration $u[s_+^{(k)}]$, is the label of the cell
left to the north green step crossing the row just below
$i_+^{(k)}$. We define $f_+(u[s_+^{(k)}])$ to be the pair of
binomial paths read in the $m\times n$ grid whose southwest corner
is $i_+^{(k)}$ (the blue dashed rectangle in our picture).

We claim that $f_+(u[s_+^{(k)}])$ is indeed a positive boundary pair. Indeed the conditions on the initial red and green steps follow immediately from the condition (1) of the definition of $i_+^{(k)}$. Notice that the other condition to be a positive boundary pair, i.e. that the length of the east suffix of the green path should be longer than or equal to the length of the east suffix of the red path, can be translated by saying that in the top row of the diagram of the pair the red north step should be at most a unit step to the left of the green north step in the same row. Using the periodicities of the green and the red paths, this is equivalent to condition (2) of the definition of $i_+^{(k)}$ (cf. the diagram of Figure \ref{fig:positive_case}), and this proves the claim.
\begin{figure}[ht!]
\begin{center}
\begin{tikzpicture}[scale=0.9]
\draw[black!15!white] (-5,-1) grid (11,5);
\draw[blue,line width=2,dashed] (0,0) rectangle (8,4);
\draw[green,line width=2] (3,0) -- (3,1);
\fill[fill=red,opacity=0.2,line width=0] (2,0) rectangle (3,1);
\draw node[red] at (2.5,0.5) {$s+1$};
\draw[red,line width=2] (0,0) -- (0,1);
\draw[dashed,<->] (0,0.5) -- (2,0.5);
\draw node at (1,0.9) {$i \geq 0$};
\draw[red,line width=2] (0,0) -- (-2,0);
\draw[dashed,<->] (0,-0.5) -- (-2,-0.5);
\draw node at (-0.8,-1) {$j \geq 0$};
\draw[red,line width=2] (-2,0) -- (-2,-1);
\draw[dashed,<->] (-2,-0.5) -- (-4,-0.5);
\draw node at (-3,-1) {$k\geq 0$};
\fill[fill=red,opacity=0.2,line width=0] (-5,-1) rectangle (-4,0);
\draw node[red] at (-4.5,-0.5) {$s$};
\draw[green,line width=2] (-4,-1) -- (-4,0);
\draw[green,line width=2] (-4,0) -- (-2,0);
\draw[green,line width=2] (0,0) -- (3,0);
\draw[blue,line width=2] (-0.25,-0.25) -- (0.25,0.25);
\draw[blue, line width=2] (-0.25,0.25) -- (0.25,-0.25);
\draw[red, line width=2] (5,3) -- (5,4) -- (7,4) -- (7,5);
\fill[red, opacity=0.2] (3,3) rectangle (4,4);
\draw node[red] at (3.5,3.5) {$\! s\! +\! n$};
\fill[red, opacity=0.2] (10,4) rectangle (11,5);
\draw node[red] at (10.5,4.5) {$\!\!\!\!\!\begin{array}{rl} s\!\!\!\!\! &  + n\\ & -1\end{array}$};
\draw[green, line width=2] (4,3) -- (4,4) -- (5,4);
\draw[green, line width=2] (7,4) -- (11,4) -- (11,5);
\draw[dashed,<->] (4,3.5) -- (5,3.5);
\draw node at (4.5,3) {$k-1$};
\draw[dashed,<->] (5,3.5) -- (7,3.5);
\draw node at (6,3) {$j$};
\draw[dashed,<->] (7,4.5) -- (10,4.5);
\draw node at (8.5,5) {$i+1$};
\draw[blue] node at (-5.2,-0.5) {$s_+^{(k)}=$};
\draw[blue] node at (0.5,-0.5) {$i_+^{(k)}$};
\end{tikzpicture}
\end{center}
\caption{Diagram of $f_+(u[s_+^{(k)}])$.} \label{fig:positive_case}
\end{figure}

Similarly, the value $s_-^{(k)}$, defining the negative boundary
configuration $u[s_-^{(k)}]$, is the label of the cell left to the
north green step that crosses the $n$-th row above $i_-^{(k)}$
(the first one being the row just above $i_-^{(k)}$). We define
$f_-(u[s_-^{(k)}])$ to be the pair of binomial paths read in the
$m\times n$ grid whose southwest corner is $i_-^{(k)}$ (the dashed
orange rectangle in our picture).

We claim that $f_-(u[s_-^{(k)}])$ is indeed a negative boundary pair: the proof is similar to the one for $f_+$ and it is left to the reader (cf. the diagram of Figure \ref{fig:negative_case}).
\begin{figure}[ht!]
\begin{center}
\begin{tikzpicture}[scale=0.9]
\draw[black!15!white] (-5,-1) grid (11,5);
\draw[orange,line width=2,dashed] (0,0) rectangle (8,4);
\draw[red,line width=2] (2,0) -- (2,1);
\fill[fill=red,opacity=0.2,line width=0] (-1,0) rectangle (0,1);
\draw node[red] at (-0.5,0.5) {$\!\!\!\!\!\begin{array}{rl} s\!\!\!\!\! &  - n\\ & +1\end{array}$};
\draw[green,line width=2] (0,0) -- (0,1);
\draw[dashed,<->] (0,0.5) -- (2,0.5);
\draw node at (1,0.9) {$i \geq 1$};
\draw[green,line width=2] (0,0) -- (-2,0);
\draw[dashed,<->] (0,-0.5) -- (-2,-0.5);
\draw node at (-0.8,-1) {$j \geq 0$};
\draw[green,line width=2] (-2,0) -- (-2,-1);
\draw[dashed,<->] (-2,-0.5) -- (-5,-0.5);
\draw node at (-3.5,-1) {$k\geq 0$};
\fill[fill=red,opacity=0.2,line width=0] (-3,-1) rectangle (-2,0);
\draw node[red] at (-2.5,-0.5) {$\!\!\!\!\begin{array}{rl} s\!\!\!\!\! &  - n\\ & \,\,\end{array}$};
\draw[red,line width=2] (-5,-1) -- (-5,0);
\draw[red,line width=2] (-5,0) -- (-2,0);
\draw[red,line width=2] (0,0) -- (2,0);
\draw[orange,line width=2] (-0.25,-0.25) -- (0.25,0.25);
\draw[orange, line width=2] (-0.25,0.25) -- (0.25,-0.25);
\draw[green, line width=2] (6,3) -- (6,4) -- (8,4) -- (8,5);
\fill[red, opacity=0.2] (5,3) rectangle (6,4);
\draw node[orange] at (4.8,3.5) {$s_-^{(k)}=$};
\draw node[red] at (5.5,3.5) {$s$};
\fill[red, opacity=0.2] (7,4) rectangle (8,5);
\draw node[red] at (7.5,4.5) {$s+1$};
\draw[red, line width=2] (2,3) -- (2,4) -- (6,4);
\draw[red, line width=2] (8,4) -- (9,4) -- (9,5);
\draw[dashed,<->] (2,3.2) -- (6,3.2);
\draw node at (4,3) {$k+1$};
\draw[dashed,<->] (6,3.5) -- (8,3.5);
\draw node at (7,3) {$j$};
\draw[dashed,<->] (8,4.5) -- (9,4.5);
\draw node at (8.5,5) {$i-1$};
\draw node[orange] at (0.5,-0.5) {$i_-^{(k)}$};
\end{tikzpicture}
\end{center}
\caption{Diagram of $f_-(u[s_-^{(k)}])$. \label{fig:negative_case}}
\end{figure}
\end{example}

Now that we defined our maps $f_+$ and $f_-$ we can turn to the proof that they have all the claimed properties.
\begin{proof}[Proof of Lemma~\ref{lem:f+f-}]
We first discuss the bijectivities of $f_+$ and $f_-$. Both they
rely on some results from \cite{addl}. Indeed in \cite[Proposition
5.12]{addl} it was show that given a pair of binomial paths in the
$m\times n$ grid, if we look at the corresponding periodic
diagram, and we move a $m\times n$ grid with its southwest corner
anchored to the green path, we always pass by the diagrams of $m$
distinct stable sorted configurations; moreover all the stable
sorted configurations occur in precisely one of these periodic
diagrams (in fact this is the essence of the Cyclic Lemma in
\cite[Lemma 3.1]{addl}).

We already observed that $f_+(u[s_+^{(k)}])$ is a positive boundary pair and that such pairs correspond to stable sorted configurations (cf. Remark  \ref{rem:positive_pairs_stable}). By construction, $f_+(u[s_+^{(k)}])$ is obtained by looking at the periodic diagram of the parking sorted configuration $u$, and then by moving along the diagram according to $s_+^{(k)}$ in order to get the corresponding pair. By the results mentioned above, starting with distinct parking sorted partial configurations $u[*]$ and $u'[*]$ will lead necessarily to distinct periodic diagrams, hence to distinct pairs, while for $k\neq j$ the configurations $f_+(u[s_+^{(k)}])$ and $f_+(u[s_+^{(j)}])$ will be two of the stable configurations of the periodic diagram of the same $u[*]$, hence once again they will be distinct. This proves the injectivity of $f_+$. For the surjectivity we can use again the same results: given a positive boundary pair, this corresponds to a stable sorted configuration, hence looking at the corresponding periodic diagram we can deduce the corresponding parking sorted configuration $u[*]$, and finally the correct value $s_+$ for the sink. This finishes the proof of the bijectivity of $f_+$.

The proof of the bijectivity of $f_-$ is analogous: we need only
to observe that a negative boundary pair of binomials paths
contains necessarily a positive boundary intersection (cf. the
diagram of Figure \ref{fig:negative_case}). So, in the periodic
diagram of our pair, we can move the southwest corner of the
$m\times n$ grid on the lowest of such positive boundary
intersections, and this by definition will give us a positive
boundary pair. By the periodicity of the red and the green paths,
such a pair is uniquely determined by our original pair. Now the
same argument that we used to prove the bijectivity of $f_+$ gives
us also the bijectivity of $f_-$.

It remains to show the coincidence of the bistatistics $(\xpara,\ypara)$ for $u[s_{\pm}^{(k)}]$ and $(\xpara_{\pm},\ypara_{\pm})$ for $f_{\pm}(u[s_{\pm}^{(k)}])$.

In order to read $(\xpara(u[s_+^{(k)}]),\ypara(u[s_+^{(k)}]))$ from the positive boundary pair $f_+(u[s_+^{(k)}])$ we will use a local cylindric labeling of the diagram of the pair. The construction is better understood in an example: consider the diagram to the left of Figure~\ref{fig:cyl_lab_pairs}. This is the diagram of a positive boundary pair. We labelled the cells of this diagram in the usual cylindric way, following the green path, starting with $0$ in the cell to the left of the lowest north green step. We already observed that this cell is the one labelled by $s_+^{(k)}+1$ in the periodic diagram of the corresponding parking sorted configuration (cf. the diagram of Figure \ref{fig:positive_case}). Again, we denoted $-n$ by $\overline{n}$ for all $n\in \mathbb{N}\setminus \{0\}$.

\begin{figure}[ht!]
\begin{center}
\begin{tikzpicture}[scale=0.75]
\draw[black!30!white] (0,0) grid (9,7);
\fill[fill=red,opacity=0.3] (0,0) -- (2,0) -- (2,2) -- (1,2) -- (1,1) -- (0,1) -- cycle;
\fill[fill=blue,opacity=0.3] (2,2) -- (5,2) -- (5,5) -- (4,5) -- (4,4) -- (2,4) -- cycle;
\fill[fill=blue,opacity=0.3] (5,6) -- (7,6) -- (7,7) -- (5,7) -- cycle;
\draw[dashed,draw=blue,line width=2] (0,0) rectangle (9,7);
\draw[line width=4,draw=green] (0,0) -- (3,0) -- (3,4) -- (5,4) -- (5,5) -- (6,5) -- (6,7) -- (9,7);
\draw[line width=2,draw=red] (0,0) -- (0,1) -- (1,1) -- (1,2) -- (2,2) -- (5,2) -- (5,6) -- (7,6) -- (7,7) -- (8,7);
\draw[blue,line width=2] (-0.25,0.25) -- (0.25,-0.25);
\draw[blue,line width=2] (-0.25,-0.25) -- (0.25,0.25);
\foreach \x/\y/\label in {0/0/$\overline{14}$,1/0/$\overline{7}$,2/0/$0$,3/0/$7$,4/0/$14$,5/0/$21$,6/0/$\ldots$,
0/1/$\overline{13}$,1/1/$\overline{6}$,2/1/$1$,3/1/$8$,4/1/$15$,5/1/$22$,6/1/$\ldots$,
0/2/$\overline{12}$,1/2/$\overline{5}$,2/2/$2$,3/2/$9$,4/2/$16$,5/2/$23$,6/2/$\ldots$,
0/3/$\overline{11}$,1/3/$\overline{4}$,2/3/$3$,3/3/$10$,4/3/$17$,5/3/$24$,6/3/$\ldots$,
2/4/$\overline{10}$,3/4/$\overline{3}$,4/4/$4$,5/4/$11$,6/4/$18$,7/4/$25$,8/4/$\ldots$,
3/5/$\overline{9}$,4/5/$\overline{2}$,5/5/$5$,6/5/$12$,7/5/$19$,8/5/$26$,9/5/$\ldots$,
3/6/$\overline{8}$,4/6/$\overline{1}$,5/6/$6$,6/6/$13$,7/6/$20$,8/6/$27$,9/6/$\ldots$}{
\draw node at (\x+0.5,\y+0.5) {\label};
}
\end{tikzpicture}
\begin{tikzpicture}[scale=0.75]
\draw[black!30!white] (0,0) grid (9,7);
\fill[red,opacity=0.3] (2,1) -- (4,1) -- (4,2) -- (5,2) -- (5,3) -- (2,3) -- cycle;
\fill[red,opacity=0.3] (6,5) -- (8,5) -- (8,7) -- (7,7) -- (7,6) -- (6,6) -- cycle;
\fill[blue,opacity=0.3] (0,0) -- (2,0) -- (2,1) -- (0,1) -- cycle;
\fill[blue,opacity=0.3] (5,3) -- (6,3) -- (6,5) -- (5,5) -- cycle;
\draw[dashed,draw=orange,line width=2] (0,0) rectangle (9,7);
\draw[line width=4,draw=green] (0,0) -- (0,1) -- (4,1) -- (4,2) -- (5,2) -- (5,5) -- (8,5) -- (8,7) -- (9,7);
\draw[line width=2,draw=red] (0,0) -- (2,0) -- (2,3) -- (6,3) -- (6,6) -- (7,6) -- (7,7) -- (8,7);
\draw[orange,line width=2] (-0.25,0.25) -- (0.25,-0.25);
\draw[orange,line width=2] (-0.25,-0.25) -- (0.25,0.25);
\foreach \x/\y/\label in {-1/0/$\overline{7}$,0/0/$0$,1/0/$7$,2/0/$\ldots$,
1/1/$\overline{20}$,2/1/$\overline{13}$,3/1/$\overline{6}$,4/1/$1$,5/1/$8$,6/1/$\ldots$,
1/2/$\overline{26}$,2/2/$\overline{19}$,3/2/$\overline{12}$,4/2/$\overline{5}$,5/2/$2$,6/2/$9$,7/2/$\ldots$,
3/3/$\overline{11}$,4/3/$\overline{4}$,5/3/$3$,6/3/$10$,7/3/$17$,8/3/$\ldots$,
3/4/$\overline{10}$,4/4/$\overline{3}$,5/4/$4$,6/4/$11$,7/4/$18$,8/4/$\ldots$,
5/5/$\overline{16}$,6/5/$\overline{9}$,7/5/$\overline{2}$,8/5/$5$,9/5/$\ldots$,
6/6/$\overline{8}$,7/6/$\overline{1}$,8/6/$6$,9/6/$\ldots$}{
\draw node at (\x+0.5,\y+0.5) {\label};
}
\end{tikzpicture}
\end{center}

\caption{Local cyclindric labeling of (positive then negative) boundary pairs \label{fig:cyl_lab_pairs}}
\end{figure}

The general remark is that this new cylindric diagram still allows us to compute $\xpara$ and $\ypara$ of the original parking sorted configuration $u[s_+^{(k)}]$: indeed it turns out that in this way we labelled cylindrically by negative values the cells visited by the algorithm computing the rank of this configuration, and by non-negative values the unvisited cells. The periodicity of the diagram implies now that $\xpara(u[s_+^{(k)}])$ is precisely equal to the number of cells to the left of the red path that we now labelled by non-negative numbers (blue in our picture), while $\ypara(u[s_+^{(k)}])$ is the number of cells to the right of the red path that we labelled by negative numbers (pink in our picture).

It is now clear (and it can be checked in our example) that the $\xpara(u[s_+^{(k)}])$ unvisited cells (hence with non-negative labels) to the left of the red path are also counted by $\xarea(p_+)+\xrow(p_+) = \xpara_+(p_+)$ and that the $\ypara(u[s_+^{(k)}])$ visited cells (hence with negative labels) to the right of the red path are also counted by $\yarea(p_+)-\yrow(p_+) = \ypara_+(p_+)$.

In order to read $(\xpara(u[s_-^{(k)}]),\ypara(u[s_-^{(k)}]))$ from the negative boundary pair $f_-(u[s_-^{(k)}])$ we can use a similar construction, but with a different convention. Consider the diagram to the right of Figure~\ref{fig:cyl_lab_pairs}. This is the diagram of a negative boundary pair. We labelled the cells of this diagram in the usual cylindric way, following the green path, starting with $\overline{1}=-1$ in the cell to the left of the highest green north step. We already observed that this cell is the one labelled by $s_-^{(k)}$ in the periodic diagram of the corresponding parking sorted configuration (cf. the diagram of Figure \ref{fig:negative_case}). Again, we denoted $-n$ by $\overline{n}$ for all $n\in \mathbb{N}\setminus \{0\}$.

As before, also in this case the negative labels correspond to cells visited by the algorithm that computes the rank of $u[s_-^{(k)}]$, while the non-negative labels correspond to the unvisited cells.

It is again clear (and it can be checked in our example) that the $\xpara(u[s_-^{(k)}])$ unvisited cells to the left of the red path (blue in our picture) are also counted by $\xarea(p_-) = \xpara_-(p_-)$ and the $\ypara(u[s_-^{(k)}])$ visited cells to the right of the red path (pink in our picture) are also counted by $\yarea(p_-) = \ypara_-(p_-)$.

This completes the proof of Lemma~\ref{lem:f+f-}.
\end{proof}

Now that we identified the bistatistics on boundary configurations with the ones on boundary pairs, we need only one last ingredient to get our formula for $\mathcal{F}(x,y,w,h)$: parallelogram polyominoes.

A \emph{parallelogram} (or \emph{staircase}) \emph{polyomino}
can be defined as the \underline{non-empty} union of unit squares
delimited by a pair of binomial paths in a grid that cross only in
the southwest corner and in the northeast corner of the grid.
Naturally the $\area$ of such a polyomino is just the number of
these unit squares, its $\width$ is the the width of the bounding
box, and its $\height$ is the height of the bounding box.
\begin{remark}
Notice that Proposition \ref{prop:sorted_recurrent} says that the recurrent sorted configurations on $K_{m,n}$ are the stable sorted configurations on $K_{m,n}$ whose diagram is a polyomino of $\width$ $m$ ad $\height$ $n$.
\end{remark}

In \cite{bousquetmelouviennot} the generating function of parallelogram polyominoes according to their area (counted by $q$), their width (counted by $w$) and their height (counted by $h$)
$$P(q;w,h) := \sum_{P} q^{\area(P)}w^{\width(P)}h^{\height(P)}  $$
is shown to be equal to
$$ P(q;w,h) = qwh \frac{L(qw,qh)}{L(w,h)}$$
where
$$ L(w,h) := \sum_{n\geq 0,m\geq 0} \frac{(-1)^{m+n}h^nw^mq^{m+n+1 \choose 2}}{(q)_n(q)_m},$$
and $(a)_n := \prod_{i=0}^{n-1}(1-q^ia)$.

We want to prove the following formula (compare it with \cite[Theorem 21]{corileborgne}).
\begin{theorem} \label{thm:gf_main_formula}
$$
\mathcal{F}(x,y,w,h)=\frac{(1-xy)(hw-P(x;w,h)P(y;w,h))}{(1-x)(1-y)(1-h-w-P(x;w,h)-P(y;w,h))}.
$$
\end{theorem}
\begin{remark}
Notice that from the formula of Theorem \ref{thm:gf_main_formula} the symmetries of $\mathcal{F}(x,y,w,h)$ in $x$ and $y$ (proved combinatorially in Theorem \ref{thm:x_y_symmetry}) as well as in $w$ and $h$ (immediate from \eqref{eq:x_y_from_d_r} and the obvious $\widetilde{K}_{m,n}(d,r)=\widetilde{K}_{n,m}(d,r)$), are both apparent.
\end{remark}

We set
$$
K_{m,n}^+(x,y):=\sum_u x^{\xpara(u)}y^{\ypara(u)}
$$
where the sum is taken over all positive boundary configurations $u$ on $K_{m,n}$, and similarly
$$
K_{m,n}^-(x,y):=\sum_u x^{\xpara(u)}y^{\ypara(u)}
$$
where the sum is taken over all negative boundary configurations $u$ on $K_{m,n}$.

\begin{lemma}\label{lem:enumeration_of_boundaries}
We have
\begin{align*}
\sum_{n\geq 1,m\geq 1}K_{m,n}^+(x,y)w^mh^n & = \frac{(1-hx-w)hw +(w-h)P(x;w,xh)P(y;w,h/y) }{(1-w)(1-w-xh-P(y;w,h/y)-P(x;w,xh))}\\
 & + \frac{(1-hx - w + xw)hP(y;w,h/y) -hw P(x;w,xh)}{(1-w)(1-w-xh-P(y;w,h/y)-P(x;w,xh))}
\end{align*}
and
$$ \sum_{n\geq 1,m\geq 1}K_{m,n}^-(x,y)w^mh^n = \frac{P(x;w,h)P(y;w,h)}{(1-w)(1-h-w-P(x;w,h)-P(y;w,h))} $$
where $P(q;w,h)$ is the previously defined generating function of parallelogram polynominoes according to $\area$, $\width$ and $\height$.
\end{lemma}

\begin{proof}
\newcommand{\redgreeneast}{\tikz[scale=0.3]{\draw[green,line width=1] (0,0.05) -- (1,0.05);\draw[red,line width=1] (0,-0.05) -- (1,-0.05);}}
\newcommand{\redgreennorth}{\tikz[scale=0.3]{\draw[green,line width=1] (0.05,0) -- (0.05,1);\draw[red,line width=1] (-0.05,0) -- (-0.05,1);}}
\newcommand{\redpolynomino}{\tikz[scale=0.3]{\draw[red,line width=1] (0,0) -- (0,1) -- (1,1);\draw[green,line width=1] (0,0) -- (1,0) -- (1,1);}}
\newcommand{\redpolynominotopone}{\tikz[scale=0.3]{\draw[red,line width=1] (0,0) -- (0,0.8) -- (0.8,0.8) -- (0.8,1) -- ((1,1);\draw[green,line width=1] (0,0) -- (1,0) -- (1,1);}}
\newcommand{\greenpolynomino}{\tikz[scale=0.3]{\draw[green,line width=1] (0,0) -- (0,1) -- (1,1);\draw[red,line width=1] (0,0) -- (1,0) -- (1,1);}}
\newcommand{\greenpolynominotopone}{\tikz[scale=0.3]{\draw[green,line width=1] (0,0) -- (0,0.8) -- (0.8,0.8) -- (0.8,1) -- ((1,1);\draw[red,line width=1] (0,0) -- (1,0) -- (1,1);}}

We use the notation of regular (formal) languages: given a \emph{language} $L$, i.e. a set of words in some \emph{alphabet}, we denote by $L^*$ the language of all possible concatenations of a finite (possibly empty) sequence of words each coming from the language $L$. In particular we will denote any language $L$ as the formal sum of its words: so for example if $L=\{010,101,11011\}$ is our language in the alphabet $\{0,1\}$, then we identify $L$ with $010+101+11011$ and we will denote $L^*$ by $(010+101+11011)^*$. Finally, given two languages $L$ and $K$, we will denote by $LK$ the language of all possible concatenations of a word $v$ from $L$ with a word $w$ from $K$.

Consider the diagram of a boundary pair, with the red path incremented by a final east step. Such a pair can be decomposed into factors made up of superposed red and green north (\redgreennorth) or east (\redgreeneast) steps, or parallelogram polyominoes whose upper path is either red (\redpolynomino), which we call \emph{red polyominoes}, or green (\greenpolynomino), which we call \emph{green polyominoes}. The constraints on prefixes or suffixes of these paths coming from the definition of positive or negative boundary pairs may be translated into this decomposition. We want to express the languages of positive and negative boundary pairs in terms of these languages.

We denote by $\redpolynomino$ the language consisting of (i.e. the formal sum of) all the red polyominoes. Similarly we denote by $\greenpolynomino$ the language consisting of all the green polyomionoes. Here $\redgreennorth$ and $\redgreeneast$ will denote the languages consisting just of a superposed red and green north step and a superposed red and green east step respectively.

We discuss first the case of positive boundary pairs.

The first red step is north and the first green step is east, hence a positive boundary pair always starts with a red polynomino.
The constraint on east suffixes lengths is taking into account by a maximal suffix of superposed red and green east steps.
In addition, we have to discuss if the factor just before the maximal suffix in $(\redgreeneast)^*$ is a red polyomino ($\redpolynomino$) or something else ($\redgreennorth$ or $\greenpolynomino$).

$\ \circ$ If the factor just before the maximal suffix in $(\redgreeneast)^*$ is a red polyomino, then its topmost row contains exactly one cell: we denote by $\redpolynominotopone$ the language consisting of such polyominoes.
We then have to discuss if the first and last red polynominos are the same or if they are distinct.

If they are the same we are lead to
$$ \redpolynominotopone(\redgreeneast)^* $$
otherwise to
$$ \redpolynomino(\redgreeneast+\redgreennorth+\redpolynomino+\greenpolynomino)^*\redpolynominotopone(\redgreeneast)^*.$$

$\ \circ$ If the last factor before the maximal suffix in $(\redgreeneast)^*$ is either $\redgreennorth$ or $\greenpolynomino$, then, by the periodicity of the red path, the maximal suffix in $(\redgreeneast)^*$ must have length at least $1$. There are no further constraints, hence the language of these pairs can be described as
$$ \redpolynomino(\redgreeneast+\redgreennorth+\redpolynomino+\greenpolynomino)^*(\redgreennorth+\greenpolynomino)\redgreeneast(\redgreeneast)^*.$$

In conclusion, the language of positive boundary pairs can be described as
\begin{equation}  \label{eq:enum_positive}
\redpolynominotopone(\redgreeneast)^*+\redpolynomino(\redgreeneast+\redgreennorth+\redpolynomino+\greenpolynomino)^*\redpolynominotopone(\redgreeneast)^*+ \redpolynomino(\redgreeneast+\redgreennorth+\redpolynomino+\greenpolynomino)^*(\redgreennorth+\greenpolynomino)\redgreeneast(\redgreeneast)^*.
\end{equation}

We discuss now the case of negative boundary pairs.

Its first red step is east and the first green step is north, hence a negative boundary pair starts by a green polynomino.
The constraint on east suffixes length is taken into account by the maximal suffix in $(\redgreeneast)^*$ and the fact that there is a final red polynomino just before this suffix.
Therefore all the negative pairs are described by
\begin{equation}  \label{eq:enum_negative}
\greenpolynomino(\redgreeneast+\redgreennorth+\redpolynomino+\greenpolynomino)^*\redpolynomino(\redgreeneast)^*.
\end{equation}

Now that the positive and negative negative pairs are described by such non-ambiguous regular expressions, it remains to compute the generating functions that take into account the parameters $(\xpara_+,\ypara_+)$, respectively $(\xpara_-,\ypara_-)$, and of course the width and the height of the grids.

The case of negative boundary pairs $p_-$ is easier: according to the end of the proof of Lemma~\ref{lem:f+f-}, the area of red polyominoes counts exactly all the contributions to the statistic $\ypara_-(p_-)$, while the area of green polynominoes describes exactly all the contributions to the statistic $\xpara_-(p_-)$.

Hence our generating function for $\greenpolynomino$ is $P(x;w,h)$, for $\redpolynomino$ is $P(y;w,h)$, and trivially for $\redgreeneast$ and $\redgreennorth$ is simply
$w$ and $h$ respectively. All this together with \eqref{eq:enum_negative} gives the claimed
$$ \sum_{n\geq 1,m\geq 1}K_{m,n}^-(x,y)w^mh^n =\qquad \qquad \qquad\qquad \qquad \qquad  \qquad \qquad \qquad \qquad \qquad $$
\begin{align*}
& = P(x;w,h)\frac{1}{1-(w+h+P(x;w,h)+P(y;w,h))} P(y;w,h) \frac{1}{1-w}\\
& =\frac{P(x;w,h) P(y;w,h) }{(1-w)(1-w-h-P(x;w,h)-P(y;w,h))} .
\end{align*}


The case of positive boundary pairs $p_+$ is slightly more complicate since, still according to the end of the proof of Lemma~\ref{lem:f+f-}, the area of red polyominoes should be decremented by one on each row for the contribution of $\ypara_+(p_+)$, while the area of green polynominoes and vertical superposed green and red north steps should be incremented by one on each row for the contribution of $\xpara_+(p_+)$.
These modifications may be taken into account by suitable changes of variables, counting the height as follows: our generating function for $\greenpolynomino$ is $P(x;w,xh)$, for $\redpolynomino$ is $P(y;w,h/y)$, for $\redgreeneast$ is just $w$, and for $\redgreennorth$ is just $xh$. In order to compute the generating function of $\redpolynominotopone$, we observe that adding one extra top row of one cell to a parallelogram polyomino is unambiguous, hence its generating function is $h(w+P(y;w,h/y))$.
All this together with \eqref{eq:enum_positive} gives the claimed
$$ \sum_{n\geq 1,m\geq 1}K_{m,n}^+(x,y)w^mh^n =\qquad \qquad \qquad \qquad \qquad \qquad \qquad \qquad \qquad \qquad$$
\begin{align*}
 & =  h(w+P(y;w,h/y))\frac{1}{1-w}+\\
 & + P(y;w,h/y)\frac{1}{1-w-xh-P(y;w,h/y)-P(x;w;xh)}h(w+P(y;w,h/y))\frac{1}{1-w}\\
 & + P(y;w,h/y)\frac{1}{1-w-xh-P(y;w,h/y)-P(x;w;xh)}(xh+P(x;w,xh))\frac{w}{1-w}\\
  & = \frac{(1-hx-w)hw +(w-h)P(x;w,xh)P(y;w,h/y) }{(1-w)(1-w-xh-P(y;w,h/y)-P(x;w,xh))}\\
 & + \frac{(1-hx - w + xw)hP(y;w,h/y) -hw P(x;w,xh)}{(1-w)(1-w-xh-P(y;w,h/y)-P(x;w,xh))}.
\end{align*}
\end{proof}

%

We are now in a position to prove Theorem \ref{thm:gf_main_formula}.
\begin{proof}[Proof of Theorem \ref{thm:gf_main_formula}]
In the formulae of Lemma~\ref{lem:enumeration_of_boundaries} we can remove the substitutions on the variable $h$ by using the following decomposition of parallelogram polyominoes.

We look at the lowest occurrence of two rows of unit cells of the parallelogram polyomino whose intersection consists of a single unit edge.

If there is not such a pair of rows, then either this is the one cell polyomino, counted by $whq$, or we can remove the first cell of each row and still get a parallelogram polyomino: these polyominoes are counted by $wP(q;w,qh)$.

If there is such a pair of rows, then our polyomino decomposes in two parallelogram polyominoes: the first one consists of the rows below (and including) the lowest row of our pair, while the second one is an unrestricted parallelogram polyomino. Hence these polyominoes are counted by $(qwh+wP(q;w,qh))\frac{P(q;w,h)}{w}$.

This decomposition leads to the identity
$$ P(q;w,h) = (qh+P(q;w,qh))(w+P(q;w,h)).$$
From this equation we deduce that
$$ P(q;w,h) =  \frac{w(qh + P(q;w,qh))}{1-qh - P(q;w,qh)} \mbox{ and } P(q;w,qh) = \frac{(1-qh)P(q;w,h)-hqw}{w+P(q;w,h)}.$$
Making the substitution $q=x$ in the second equation gives an expression of $P(x;w,xh)$ in terms of $P(x;w,h)$, namely
$$
P(x;w,xh)=\frac{(1-xh)P(x;w,h)-hxw}{w+P(x;w,h)},
$$
while making the substitutions $q=y$ and $h=h/y$ in the first equation gives an expression of $P(y;w,h/y)$ in terms of $P(y;w,h)$, namely
$$
P(y;w,h/y)=\frac{w(h + P(y;w,h))}{1-h - P(y;w,h)}.
$$
Using these formulae, Lemma~\ref{lem:formula_F_u} and Lemma~\ref{lem:enumeration_of_boundaries}, we get
\begin{align*}
\mathcal{F}(x,y,w,h) & =\sum_{n\geq 1,m\geq 1} K_{m,n}(x,y)w^mh^n\\
 & =\frac{1-xy}{(1-x)(1-y)}\left( \sum_{n\geq 1,m\geq 1}K_{m,n}^+(x,y)w^mh^n-\sum_{n\geq 1,m\geq 1}K_{m,n}^-(x,y)w^mh^n\right)\\
 & =\frac{1-xy}{(1-x)(1-y)}\left(\frac{(1-hx-w)hw +(w-h)P(x;w,xh)P(y;w,h/y) }{(1-w)(1-w-xh-P(y;w,h/y)-P(x;w,xh))}\right.\\
 & + \frac{(1-hx - w + xw)hP(y;w,h/y) -hw P(x;w,xh)}{(1-w)(1-w-xh-P(y;w,h/y)-P(x;w,xh))}\\
 & -\left.\frac{P(x;w,h) P(y;w,h) }{(1-w)(1-h-w-P(x;w,h)-P(y;w,h))}\right)\\
 & =\frac{(1-xy)(hw-P(x;w,h)P(y;w,h))}{(1-x)(1-y)(1-h-w-P(x;w,h)-P(y;w,h))}.
\end{align*}
This completes the proof of the theorem.
\end{proof}

\end{document}